\documentclass[12pt]{article}
\usepackage[colorlinks=true,linkcolor=blue,citecolor=Red]{hyperref}
\usepackage{breakcites}

\usepackage{mathtools}

\usepackage{subfigure}

\usepackage{algorithm}
\usepackage{algpseudocode}

\usepackage{bbm}

\usepackage{latexsym}
\usepackage{graphics}
\usepackage{amsmath}
\usepackage[algo2e]{algorithm2e}
\usepackage{xspace}
\usepackage{amssymb}
\usepackage{psfrag}
\usepackage{epsfig}
\usepackage{amsthm}
\usepackage{url}
\usepackage{pst-all}
\usepackage{bbm}

\usepackage{color}

\definecolor{Red}{rgb}{1,0,0}
\definecolor{Blue}{rgb}{0,0,1}
\definecolor{Olive}{rgb}{0.41,0.55,0.13}
\definecolor{Yarok}{rgb}{0,0.5,0}
\definecolor{Green}{rgb}{0,1,0}
\definecolor{MGreen}{rgb}{0,0.8,0}
\definecolor{DGreen}{rgb}{0,0.55,0}
\definecolor{Yellow}{rgb}{1,1,0}
\definecolor{Cyan}{rgb}{0,1,1}
\definecolor{Magenta}{rgb}{1,0,1}
\definecolor{Orange}{rgb}{1,.5,0}
\definecolor{Violet}{rgb}{.5,0,.5}
\definecolor{Purple}{rgb}{.75,0,.25}
\definecolor{Brown}{rgb}{.75,.5,.25}
\definecolor{Grey}{rgb}{.5,.5,.5}

\newcommand{\ind}{\mathbbm{1}}

\newcommand{\Gr}{\mathbb{G}}

\newcommand{\E}[1]{\mathbb{E}\left[#1\right]}

\newcommand{\G}{\mathbb{G}}

\newcommand{\Z}{\mathbb{Z}}

\newcommand{\R}{\mathbb{R}}
\newcommand{\N}{\mathbb{N}}

\renewcommand{\Z}{\mathbb{Z}}

\newcommand\given[1][]{\:#1\vert\:}

\newcommand{\rvline}{\hspace*{-\arraycolsep}\vline\hspace*{-\arraycolsep}}

\newcommand{\ip}[2]{\langle{#1},{#2}\rangle} 
\newcommand{\ipbig}[2]{\left\langle{#1},{#2}\right\rangle}

\newcommand{\pr}{\mathbb{P}}
\renewcommand{\E}[1]{\mathbb{E}\!\left[#1\right]}
\renewcommand{\R}{\mathbb{R}}
\renewcommand{\Z}{\mathbb{Z}}
\renewcommand{\G}{\mathbb{G}}

\newcommand{\distr}{\stackrel{d}{=}}
\newcommand{\npp}{\texttt{NPP} }
\newcommand{\nppp}{\texttt{NPP}}
\newcommand{\vbp}{\texttt{VBP} }
\newcommand{\bincube}{\mathcal{B}_n}
\newcommand{\A}{\mathcal{A}}

\newcommand{\Overlap}{\mathcal{\mathcal{O}}}
\newcommand{\OBar}{\overline{\mathcal{\mathcal{O}}}}

\newcommand{\ignore}[1]{\relax}

\newlength\myindent
\newtheorem{theorem}{Theorem}[section]
\newtheorem{remark}[theorem]{Remark}
\newtheorem{lemma}[theorem]{Lemma}

\newtheorem{proposition}[theorem]{Proposition}

\newtheorem{definition}[theorem]{Definition}

\newcommand{\compgap}{\emph{statistical-to-computational gap} }
\newcommand{\compgapp}{\emph{statistical-to-computational gap}}
\newcommand{\ogp}{\emph{Overlap Gap Property} }
\usepackage{tcolorbox}


\newcounter{parentnumber}

\makeatletter
\def\BState{\State\hskip-\ALG@thistlm}
\makeatother

\definecolor{Red}{rgb}{1,0,0}
\definecolor{Blue}{rgb}{0,0,1}
\definecolor{Olive}{rgb}{0.41,0.55,0.13}
\definecolor{Green}{rgb}{0,1,0}
\definecolor{MGreen}{rgb}{0,0.8,0}
\definecolor{DGreen}{rgb}{0,0.55,0}
\definecolor{Yellow}{rgb}{1,1,0}
\definecolor{Cyan}{rgb}{0,1,1}
\definecolor{Magenta}{rgb}{1,0,1}
\definecolor{Orange}{rgb}{1,.5,0}
\definecolor{Violet}{rgb}{.5,0,.5}
\definecolor{Purple}{rgb}{.75,0,.25}
\definecolor{Brown}{rgb}{.75,.5,.25}
\definecolor{Grey}{rgb}{.5,.5,.5}
\definecolor{Pink}{rgb}{1,0,1}
\definecolor{DBrown}{rgb}{.5,.34,.16}
\definecolor{Black}{rgb}{0,0,0}

\usepackage{float}

\usepackage{float}

\author{{\sf David Gamarnik}\thanks{MIT; e-mail: {\tt gamarnik@mit.edu}. Research supported  by the NSF grants DMS-2015517.}
\and
{\sf Eren C. K{\i}z{\i}lda\u{g}}\thanks{MIT; e-mail: {\tt kizildag@mit.edu}.}
}

\begin{document}

\title{Algorithmic Obstructions in the Random Number Partitioning Problem}
\date{\today}

\maketitle
\begin{abstract}
We consider the algorithmic problem of finding a near-optimal solution for the number partitioning problem (\nppp). This problem appears in many practical applications, including the design of randomized controlled trials, multiprocessor scheduling, and cryptography; and is also of  theoretical significance. 
The \npp possesses the so-called \compgapp: when its input $X$ has distribution $\mathcal{N}(0,I_n)$, the optimal value of the \npp is $\Theta\left(\sqrt{n}2^{-n}\right)$ w.h.p.; whereas the best polynomial-time algorithm achieves the objective value of only $2^{-\Theta(\log^2 n)}$, w.h.p. 

In this paper,  we initiate the  study of the nature of this gap. Inspired by insights from statistical physics, 
we study the landscape of the \npp and establish the presence of the Overlap Gap Property (OGP), 
an intricate geometric property which is known to be a rigorous evidence of an algorithmic hardness for large classes of algorithms. By leveraging the OGP, we establish that (a)  any sufficiently stable algorithm, appropriately defined, fails to find a near-optimal solution with energy 
below $2^{-\omega(n \log^{-1/5} n)}$; 
and (b)  a very natural Markov Chain Monte Carlo dynamics fails for find  near-optimal solutions. Our simulation results suggest that the state of the art algorithm achieving  the value $2^{-\Theta(\log^2 n)}$ is indeed stable, but formally verifying this is left as an open
problem.

OGP regards the overlap structure of $m-$tuples of  solutions achieving a certain objective value. 
When $m$ is constant we prove the presence of OGP for the objective values
of order
$2^{-\Theta(n)}$, and the absence of it in the regime $2^{-o(n)}$. Interestingly, though, by considering overlaps with growing values of $m$
we prove the presence of the OGP up to the level $2^{-\omega(\sqrt{n\log n})}$. Our proof  of the failure of stable algorithms at
values $2^{-\omega(n \log^{-1/5} n)}$ employs
methods from Ramsey Theory from the extremal combinatorics, and is of independent interest.
\end{abstract}
\newpage
\tableofcontents
\newpage
\section{Introduction}
In this paper, we study the number partitioning problem (\texttt{NPP}): given $n$ ``items" with associated weights (where $n$ is a positive integer), partition them into two ``bins", $A$ and $B$, such that the subset sums corresponding to $A$ and $B$ are as close as possible. More formally, given $n$ numbers $X_i\in\mathbb{R}$, $1\le i\le n$; find a subset $A\subset [n]\triangleq \{1,2,\dots,n\}$ such that the discrepancy
$
\mathcal{D}(A) \triangleq \left|\sum_{i\in A}X_i - \sum_{i\in A^c}X_i\right|$
is minimized. Encoding the membership $X_i\in A$ as a $+1$  and $X_i\in B$ as a $-1$; \npp can equivalently be posed as a combinatorial optimization problem over the binary cube $\bincube\triangleq \{-1,1\}^n$:
\begin{equation}\label{eq:NPP-main}
   \displaystyle \min_{\sigma\in \bincube}\left|\sum_{1\le i\le n}\sigma_i X_i\right|.
\end{equation}
Our focus is on the algorithmic problem of solving the minimization problem \eqref{eq:NPP-main} ``approximately" and ``efficiently" (in polynomial time) when the numbers $X_i\in\mathbb{R}$, $1\le i\le n$, are i.i.d.\,standard normal. We refer to ${\bf X}=(X_i:1\le i\le n)\in \R^n$ as an \emph{instance} of the \nppp. Moreover, motivated from a statistical physics perspective, we refer to $\sigma\in\bincube$ as a \emph{spin configuration}; and to any approximate minimum $\sigma$ of the problem \eqref{eq:NPP-main} as a \emph{near ground-state}. In the sequel, we slightly abuse the terminology; and use the word ``discrepancy" to refer to the optimal value of the combinatorial optimization problem \npp \eqref{eq:NPP-main} and its high-dimensional variant \eqref{eq:vbp} (see below); as well as to refer to the discrepancy achieved by any partition and the spin configuration induced by this partition. 

\npp is a special case of what is called as the \emph{vector balancing problem} (\texttt{VBP}), where the goal is to minimize the discrepancy 
\begin{equation}\label{eq:vbp}
\mathcal{D}_n\triangleq \min_{\sigma\in\bincube}\left\|\sum_{1\le i\le n}\sigma_i X_i\right\|_\infty
\end{equation}
of a collection $X_i\in\R^d$, $1\le i\le n$, of vectors. This problem is at the heart of a very important application in statistics, dubbed as \emph{randomized controlled trials}, which is often considered to be the gold standard for clinical trials \cite{krieger2019nearly,harshaw2019balancing}. Consider $n$ individuals participating in a randomized study that seeks inference for an additive treatment effect. Each individual $i$, $1\le i\le n$, has associated with them a set of covariate information $X_i\in\R^d$, a vector carrying the statistics relevant to them such as their age, weight, height, and so on. The individuals are divided into two groups, the treatment group (denoted by a $+$) and the control group (denoted by a $-$). Each group is then subject to a different condition; and a response is evaluated. Based on this response, one seeks to infer the effect of the treatment. To ensure accurate inference based on the response, it is desirable for the groups to have roughly the same covariates. See the very recent work on the design of such randomized controlled experiments by Harshaw, S{\"a}vje, Spielman, and Zhang \cite{harshaw2019balancing} (and the references therein) for a more elaborate discussion on this front. 

Besides its significance in statistics, \npp appears in many other practical applications.\,One such application is the \emph{multiprocessor scheduling}: each item represents the running time of a certain job and each bin represents a group of items that are run on the same processor in a multiprocessor environment \cite{tsai1992asymptotic}. Other practical applications of the \npp include minimizing the size and the delay of VLSI circuits \cite{coffman1991probabilistic,tsai1992asymptotic}, and the so-called Merkle-Hellman cryptosystem \cite{merkle1978hiding}, one of the earliest public key cryptosystem. For more practical applications of \texttt{NPP}, see the book by Coffman and Lueker \cite{coffman1991probabilistic}. 

In addition to its important role in statistics and its wide practical applications, \npp is also of great theoretical importance, especially in theoretical computer science, statistical physics, and combinatorial discrepancy theory (see below). \npp is included in the list of \emph{six basic NP-complete problems} by Garey and Johnson \cite{gareyjohnson}; and is the only such problem in this list dealing with numbers. For this reason, it is often used as a basis for establishing the NP-hardness of other problems dealing with numbers, including bin packing, quadratic programming; and the knapsack problem. In statistical physics, \npp is the first system for which the local REM conjecture was established \cite{borgs2009proof,borgs2009proof2}. That is, \npp is the first system which was shown to behave locally like Derrida's random energy model \cite{derrida1980random,derrida1981random}, a feature that was conjectured to be universal in random discrete systems \cite{bauke2004universality}. Last but not the least, \npp is one of the first NP-hard problems for which a certain phase transition is established rigorously, which we now discuss. Let $X_i$, $1\le i\le n$, be i.i.d. uniform from the set $\{1,2,\dots,M\}$ where $M=2^m$ (namely $X_i$ consists of $m$-bits). As a function of a certain control parameter $\kappa\triangleq m/n$ suggested by Gent and Walsh \cite{gent1996phase}, Mertens \cite{mertens1998phase} gave, a very elegant yet nonrigorous statistical mechanics argument, for the existence of a phase transition depending on whether $\kappa<1$ or $\kappa>1$: the property of finding a perfect partition (that is, a partition with zero discrepancy if $\sum_{1\le i\le n}X_i$ is even, and that with a discrepancy of one if $\sum_{1\le i\le n}X_i$ is odd) undergoes as phase transition as $\kappa$ crosses one from the above. It has been observed empirically that this phase transition is linked with the change of character of typical computational hardness of this problem. Subsequent work by Borgs, Chayes, and Pittel \cite{borgs2001phase} rigorously confirmed the existence of this phase transition. These results further highlight the significance of \npp at the intersection of computer science, statistical mechanics, and statistics.

As already mentioned, much work has been done on the \npp and its multi-dimensional version, \texttt{VBP}. 
The prior work visited below can be broadly classified into two categories, namely uncovering the value of the optimal discrepancy; and finding a near ground-state  $\sigma\in\bincube$ by means of an efficient algorithm. This was done broadly in for two settings, where the inputs $X_i\in\R^d$, $1\le i\le n$, are treated as \emph{worst-case}; and where they are treated as i.i.d. samples of a distribution, referred to as the \emph{average-case} setting. 

We first visit the \emph{worst-case} results which operate under minimal structural assumptions on the input vectors $X_i\in\R^d$, $1\le i\le n$. A landmark result of discrepancy theory in this setting is due to Spencer \cite{spencer1985six}. He established, using a 
very elegant argument called the \emph{partial coloring}, that the discrepancy $\mathcal{D}_n$ of \vbp per \eqref{eq:vbp} is at most $6\sqrt{n}$ if $d=n$ and $\max_{1\le i\le n}\|X_i\|_\infty \le 1$. Spencer's method, however, is non-constructive.
Later research on this front focused on the algorithmic problem of \emph{efficiently} finding a spin configuration $\sigma\in\bincube$ that \emph{approximately} attains a small discrepancy value. These papers are based on techniques including random walks \cite{bansal2010constructive,lovett2015constructive}, multiplicative weights \cite{levy2017deterministic}, random weights \cite{rothvoss2017constructive}; and are tight in the regime $d\ge n$: these algorithms return a spin configuration $\sigma$ with ``objective value" $O(\sqrt{n\log(2d/n)})$; and there exist examples whose discrepancy matches this value.

We next visit the \emph{average-case} results, starting with the typical value of the optimal discrepancy. A canonical assumption that the reader should keep in mind is that the inputs are i.i.d.\,standard normal. The first result to this end is due to Karmarkar et al. \cite{karmarkar1986probabilistic}. They established, using the second moment method, that  the objective value of \npp \eqref{eq:NPP-main} is $\Theta\left(\sqrt{n}2^{-n}\right)$ with high probability as $n\to\infty$. Their result remains valid when $X_i\in\R$, $1\le i\le n$ are i.i.d. samples of a distribution that is sufficiently regular. Later research extended this result to the multi-dimensional version, \texttt{VBP}. In the case where the dimension $d$ is constant, $d=O(1)$, Costello established in~\cite{costello2009balancing} that the objective value of \vbp \eqref{eq:vbp} is $\Theta\left(\sqrt{n}2^{-n/d}\right)$ with high probability. When the dimension $d$ is super-linear, in particular $d\ge 2n$, Chandrasekaran and Vempala \cite{chandrasekaran2014integer} established that the optimal discrepancy for \vbp per \eqref{eq:vbp} is essentially $O(\sqrt{n\log(2d/n)})$, ignoring certain polylogarithmic factors. In the regime where $\omega(1)\le d\le o(n)$, Turner et al. \cite{turner2020balancing} showed that the optimal discrepancy achieved per \eqref{eq:vbp} is $\Theta\left(\sqrt{n}2^{-n/d}\right)$. Moreover, their result transfer also to the case when $X_i\in\R^d$, $1\le i\le n$ consists of i.i.d. coordinates drawn from a density $f$ that is sufficiently regular (in particular, $f$ is square integrable, even; and the coordinates of $X_i$ have a finite fourth moment) and $d=O(n/\log n)$. In addition to the sub-linear regime $d=o(n)$; \cite{turner2020balancing} studies also the regime where $d\le \delta n$ for a sufficiently small constant $\delta$. For this regime, they establish that the objective value of \eqref{eq:vbp} is $O\left(\sqrt{n}2^{-1/\delta}\right)$ with probability at least $99\%$. This, together with the results of \cite{chandrasekaran2014integer} implies that there exists an explicit function $c(\delta)$ such that the discrepancy is $\Theta\left(c(\delta)\sqrt{n}\right)$ with probability at least $99\%$ for $d=\delta n$ and all $\delta>0$. This is a step towards proving the following conjecture by Aubin et al.~\cite{aubin2019storage}: there exists an explicit function $c(\delta)$ such that  the discrepancy is $c(\delta)\sqrt{n}$ \emph{with high probability} for the regime $d=\delta n$ and any $\delta>0$.

We now focus on the available algorithmic results. The best known (polynomial-time) algorithm for the \npp is due to Karmarkar and Karp \cite{karmarkar1982differencing} which, for a broad class of distributions, produces a discrepancy of $O\left(n^{-\alpha \log n}\right)$ with high probability as $n\to\infty$. The original algorithm that they analyzed rigorously is a rather complicated one. Their algorithm, however, is based on a strikingly simple yet a quite elegant, idea; called the \emph{differencing method}, which is based on the following observation. Given a list $L$ of items, placing $x,y\in L$ to the different sides of the partition amounts to removing $x$ and $y$ from $L$, and adding $|x-y|$ to $L$ instead, an operation that we refer to as \emph{differencing}. Namely, the \emph{differencing} operations applied on $x,y\in L$ returns a new list  $L\cup\{|x-y|\}\setminus \{x,y\}$. Using the \emph{differencing}, Karmarkar and Karp proposes two simple (alternative) ways of creating a partition (though they do not rigorously analyze them): the paired differencing method (PDM) and the largest differencing method (LDM). In the former, the items are ordered, and then $\lfloor n/2\rfloor$ \emph{differencing} operations are performed on the largest and second largest items, on the third and fourth largest items, and so on. The remaining $\lceil n/2\rceil$ numbers are ordered again, and the aforementioned procedure is repeated until a single item remains, which is the discrepancy achieved by PDM. In LDM, the numbers are again ordered. The \emph{differencing} operation is now applied on the largest and second largest items. The remaining list (now consisting of $n-1$ items) is ordered again, and the procedure is repeated until a single number remains. Recalling that $n$ items can be sorted in near-linear time $O(n\log n)$, the running times of PDM and LDM are indeed polynomial (in $n$). They conjectured that these two simple natural heuristics also achieve an objective value of $O\left(n^{-\alpha \log n}\right)$ with high probability. For PDM, this conjecture was disproven by Lueker \cite{lueker1987note} who showed that when the items $X_i$, $1\le i\le n$, are i.i.d.\,uniform on $[0,1]$ then the expected discrepancy achieved by the PDM algorithm is rather poor, $\Theta(n^{-1})$. For LDM, however, Yakir \cite{yakir1996differencing} confirmed this conjecture, and showed that the expected discrepancy achieved by the LDM is $n^{-\Theta(\log n)}$, when the items $X_i$ are i.i.d.\,uniform on $[0,1]$. His proof extends to the case when the items $X_i$ follow the exponential distribution, as well. Later, Boettcher and Mertens \cite{boettcher2008analysis} studied the constant in the exponent, and argued, based on non-rigorous calculations, that the expected discrepancy for LDM is $n^{-\alpha \log n}$ for $\alpha = \frac{1}{2\ln 2}=0.721\dots$. 

Another algorithm is due to Krieger et al. \cite{krieger2019nearly} which achieves an objective value of $O\left(n^{-2}\right)$. It is worth noting that albeit having a poor performance, the algorithm of Krieger et al. finds a \emph{balanced} partition: a spin configuration $\sigma\in\bincube$ with $\sum_{1\le i\le n}\sigma_i\in\{0,1\}$ depending on the parity of $n$. This is of practical relevance in the design of randomized trials where the treatment and control groups are often desired to have roughly similar size. Moreover, for the multi-dimensional case $d\ge 2$, they also argue that their algorithm achieves a performance of $O\left(n^{-2/d}\right)$.  Finally, Turner et al. \cite{turner2020balancing} devised a generalized version of the Karmarkar-Karp algorithm \cite{karmarkar1982differencing}, which returns a partition with discrepancy $2^{-\Theta\left(\log^2 n/d\right)}$ provided the dimension $d\ge 2$ satisfies $d = O\left(\sqrt{\log n}\right)$. 

The results recorded above highlight a striking gap between what the existential methods (such as the second moment method) guarantee and what the polynomial-time algorithms achieve. To recap, in the case when $X_i$, $1\le i\le n$, are i.i.d.\,standard normal, the optimal discrepancy of the \npp per \eqref{eq:NPP-main} is $\Theta\left(\sqrt{n}2^{-n}\right)$ with high probability; whereas the-state-of-the-art  algorithm (by Karmarkar and Karp) only achieves a performance of $2^{-\Theta(\log^2 n)}$, which is exponentially worse. On the negative side, Hoberg et al.~\cite{hoberg2017number} provides an evidence of computational hardness for the problem of approximating the discrepancy per \eqref{eq:NPP-main} in worst-case by showing that any (polynomial-time) oracle that can approximate the discrepancy  
to within a multiplicative factor of $O\left(2^{\sqrt{n}}\right)$ is also an (polynomial-time) approximation oracle for Minkowski's problem. 

{\bf A \emph{Statistical-to-computational Gap.}} In light of these findings, it is plausible to conjecture that \npp exhibits a \compgapp: a gap between what can be achieved \emph{information-theoretically} (with unbounded computational power) and what algorithms with bounded computational power (such as polynomial time algorithms) can promise. Such gaps are a universal feature of many algorithmic problems in high-dimensional statistics and in the study of random combinatorial structures; and the study of such gaps is at the forefront of current research. A partial and evergrowing list of problems with a \compgap  includes certain ``non-planted models", such as the random constraint satisfaction problems \cite{mezard2005clustering,achlioptas2008algorithmic,kothari2017sum}, the problem of finding maximum independent sets in sparse random graphs \cite{gamarnik2017,coja2015independent}, largest submatrix problem \cite{gamarnik2018finding}, the $p$-spin model \cite{montanari2019optimization,gamarnik2021overlap} and the diluted $p$-spin model \cite{chen2019suboptimality}; as well as certain ``planted" models arising in high-dimensional statistical inference tasks, such as the matrix principle component analysis (PCA) \cite{berthet2013computational,lesieur2015mmse,lesieur2015phase} and its variant, tensor PCA \cite{hopkins2015tensor,hopkins2017power,arous2020algorithmic}, high-dimensional linear regression \cite{gamarnik2017high,gamarnik2017sparse}; and the infamous planted clique problem \cite{jerrum1992large,deshpande2015improved,meka2015sum,barak2019nearly,gamarnik2019landscape}.

Unfortunately, there is as yet no analogue of the standard NP-completeness theory for these \emph{average-case} problems; and current techniques fall short of proving the hardness of such problems even under the assumption that $P\ne NP$. A notable exception to this though is when the problem possesses \emph{random self-reducibility}. As an example, Gamarnik and K{\i}z{\i}lda\u{g} \cite{gk2020} established the average-case hardness of the algorithmic problem of exactly computing the partition function of the Sherrington-Kirkpatrick spin glass under the assumption $P\ne \#P$, an assumption that is much weaker than $P\ne NP$.

Nevertheless, a very promising direction of research proposed various approaches that serve as rigorous evidence of hardness for such problems. A non-exhaustive list includes the failure of Markov chain algorithms (such as the MCMC and the Glauber Dynamics) \cite{jerrum1992large}, methods from statistical physics and in particular the failure of approximate message passing (AMP) algorithms \cite{zdeborova2016statistical, bandeira2018notes}, reductions from the infamous planted clique problem---a canonical problem widely believed to be hard on average---\cite{berthet2013computational,brennan2018reducibility,brennan2019optimal},  lower bounds against the Sum-of-Squares hierarchy \cite{hopkins2015tensor,hopkins2017power,raghavendra2018high,barak2019nearly}, lower bounds in the statistical query (SQ) model \cite{kearns1998efficient, diakonikolas2017statistical,feldman2017statistical}, low-degree methods \cite{hopkins2018statistical} and the low-degree likelihood ratio \cite{kunisky2019notes}; and so on, see \cite{kunisky2019notes} and the references therein. Another such approach, which is also our focus, is based on the insights gained from statistical physics described below.

{\bf The Overlap Gap Property (OGP).}
A relatively recent, and very promising, form of formal evidence of the average-case hardness is the presence of a certain intricate geometric property in the ``energy landscape" of the problem, dubbed as the \emph{Overlap Gap Property (OGP)}. Roughly speaking, the OGP is a disconnectivity property; and states that for every two near ground-state (appropriately defined) spin configurations $\sigma_1,\sigma_2\in\bincube$, their ``normalized" overlap do not take intermediate values: for some $0<\nu_1<\nu_2<1$, $\Overlap\left(\sigma_1,\sigma_2\right)\triangleq n^{-1}\left|\ip{\sigma_1}{\sigma_2}\right|\in[0,\nu_1]\cup [\nu_2,1]$. It was previously shown that the OGP, whenever present, is an impediment to the success of certain classes of algorithms (see below).

{\bf Origins of the OGP.} The OGP emerged originally in spin glass theory \cite{talagrand2010mean}. A precursory link between the OGP and the formal algorithmic hardness was first made in the context of random constraint satisfaction problems (k-SAT), in a series of papers by Achlioptas and Coja-Oghlan \cite{achlioptas2008algorithmic}; Achlioptas, Coja-Oghlan, and Ricci-Tersenghi \cite{achlioptas2011solution}; and by M{\'e}zard, Mora, and Zecchina \cite{mezard2005clustering}. These papers show an intriguing ``clustering" property: they establish that a large portion of the set of satisfying assignments is essentially partitioned into ``clusters" that are disconnected with respect to the natural topology of the solution space. As the onset of this clustering property coincides roughly with the regime where the known polynomial-time algorithms fail, this property was conjecturally linked with the formal algorithmic hardness. Strictly speaking, these papers do not establish the OGP. However, an inspection of their proof techniques reveals that their arguments show that it does: the normalized overlap between two satisfying assignments takes values in a set $[0,\nu_1]\cup[\nu_2,1]$ for some $0<\nu_1<\nu_2<1$ (the normalization ensures that resulting overlap values lie in $[0,1]$). The aforementioned clustering property is then inferred as a consequence of the OGP.

{\bf First algorithmic implications of OGP.} The first formal algoritmic implication of the OGP is due to Gamarnik and Sudan \cite{gamarnik2017}. In that paper, the authors study the problem of finding maximum independent sets in (sparse) random $d$-regular graphs. It is known, see in particular~\cite{frieze1990independence,frieze1992independence,bayati2010combinatorial}, that  the largest independent set of this model is of size $2(\log d/d)n$ w.h.p., in the double limit as $n\to\infty$ followed by $d\to\infty$; whereas the best known polynomial-time algorithm---a straightforward greedy algorithm---returns an independent set of cardinality at most $(\log d/d)n$. Namely, the problem exhibits a \compgapp. Gamarnik and Sudan took a rigorous look at the nature of this gap; and established, through a first moment argument, that any two independent sets with cardinality at least $\left(1+1/\sqrt{2}\right)(\log d/d)n$ have either a significant intersection (overlap) or a small intersection (namely the intermediate values are not permitted). As a consequence, they show, through an interpolation argument, that a class of powerful graph algorithms called the \emph{local algorithms/factors of i.i.d.} fails to find independent sets of cardinality larger than $\left(1+1/\sqrt{2}\right)(\log d/d)n$; and thus refuting an earlier conjecture by Hatami, Lov{\'a}sz,  and Szegedy \cite{hatami2014limits}. Later research, again through the lens of the OGP, established that low-degree polynomials also cannot find independent sets of size larger than $\left(1+1/\sqrt{2}\right)(\log d/d)n$ \cite{gamarnik2020lowFOCS}---which recovers the result of \cite{gamarnik2017} as a special case, see \cite[Appendix~A]{gamarnik2020lowFOCS}. The ``oversampling" factor, $1/\sqrt{2}$, is an artifact of their analysis; and subsequent research removed this factor for the case of local algorithms by Rahman and Vir{\'a}g~\cite{rahman2017}, and for the case of the low-degree polynomials by Wein~\cite{wein2020optimal}. This is achieved by studying the overlap structure corresponding to $m-$tuples of independent sets (as opposed to the pairs), and is tight: independent sets of cardinality near $(\log d/d)n$ can be found by means of local algorithms~\cite{lauer2007large}. The idea of looking at the overlap structure between $m$-tuples of configurations is also at the core of this paper, and is elaborated further next. 

{\bf Multioverlap Version of OGP: $m$-OGP.} It was previously observed that the idea of looking at the multioverlap structure (as opposed to the overlap of a pair) can potentially lower the phase transition point, which we detail now. As was mentioned already, for the problem of finding a maximum independent set of a (sparse) random $d$-regular graphs, 
Gamarnik and Sudan \cite{gamarnik2017} established that the local algorithms fail to find an independent set of size larger than $(1+\beta)(\log d/d)n$, where $\beta>1/\sqrt{2}$, which is still a factor of $\beta$ off the computational threshold, $(\log d/d)n$. Subsequent research by Rahman and Vir{\'a}g removed the extra oversampling factor, $1/\sqrt{2}$: instead of looking at the ``forbidden" intersection pattern for a pair of independent  sets of large cardinality, they instead proposed to look at a more intricate intersection pattern, involving many independent sets of sufficient cardinality. That way, they managed to pull the threshold (above which the local algorithms provably fail) down to $(\log d/d)n$---below which polynomial-time algorithms are known to exist. This idea of looking at the overlap structure of multiple independent sets is also employed recently by Wein~\cite{wein2020optimal} to show that the low-degree polynomials also fail to find independent sets of size greater than $(\log d/d)n$. Yet another instance, where the same theme has recurred, is the so-called Not-All-Equal-K-SAT (NAE-K-SAT) problem in the context of random constraint satisfaction problems. It was established in~\cite{coja2012catching} that such random formulas are satisfiable w.h.p. when $d<d_s\triangleq 2^{K-1}\ln 2 - \ln 2/2-1/4-o_K(1)$ where $d$ is the clause-to-variable ratio, dubbed as the \emph{density of formula}; and are non-satiable w.h.p. when $d>d_s$. Nevertheless, the best known polynomial-time algorithm---which is rather quite simple---works provided $d<\rho 2^{K-1}/K\simeq d_s/K$ \cite{achlioptas2002two} where $\rho$ is a universal constant. In particular, the NAE-K-SAT problem also exhibits a \compgapp. Gamarnik and Sudan \cite{gamarnik2017performance} established that a class of algorithms, dubbed as \emph{sequential local algorithms}---an abstraction capturing local implementations of various powerful algorithms---with a number of iterations growing moderately in the number of variables; fail to find satisfying assignments when $d>(2^{K-1}/K)\ln^2 K \simeq (d_S/K)\ln^2 K$ (which is essentially the computational threshold modulo the $\ln^2 K$ factor). The crux of their analysis is again based on establishing the aforementioned intricate geometric property of the landscape by studying at the overlap structure of  $m-$tuples of ``nearly" satisfying assignment, for an appropriate constant $m$. More specifically, they show, using a first moment argument, that w.h.p. there exists no $m$-tuple $\left(\sigma^{(i)}:1\le i\le m\right)$ of assignments such that each $\sigma^{(i)}$ satisfies a certain minimum number of clauses; and the overlap between any pair $\sigma^{(i)}$ and $\sigma^{(j)}$ of assignments, $1\le i<j\le m$, lies in a fixed interval. If one, instead, considers only pairs of satisfying assignments; then the sequential local algorithms can be shown to fail only for very high densities $d$, specifically for $d>d_s/2$.


\subsection*{Our Contributions}
In this paper, we initiate the  study of the nature of the apparent \compgap of the \npp and \texttt{VBP}. Our approach is through the lens of the intricate geometry of the energy landscape of this problem. Specifically, our approach is based on proving and leveraging the aforementioned \ogp (OGP). 
For the sake of a clear presentation, it is convenient to interpret the aforementioned gap in terms of the \emph{``exponent"} $E_n$ of the energy level $2^{-E_n}$. Thus the information-theoretical guarantee is $E_n =n$; whereas the best (efficient) computational guarantee available is only $E_n=\Theta(\log^2 n)$. Our main contributions are now in order.
\paragraph{ The regime $E_n = \Theta(n)$.} In this regime, our main result is the following. Let $X\in\R^n$ be a random vector with i.i.d. standard normal coordinates. Then for any $\epsilon>0$, there exist  $m\in\mathbb{N}$ and $\beta>\eta>0$, such that with high probability as $n$ diverges, 
there does not exist an $m$-tuple $\left(\sigma^{(i)}:1\le i\le m\right)$ of spin configurations $\sigma^{(i)}\in\bincube$ such that each $\sigma^{(i)}$ is a near ground-state in the sense $\left|\ip{\sigma^{(i)}}{X}\right|=O\left(\sqrt{n}2^{-n\epsilon}\right)$, $1\le i\le m$; and their pairwise overlaps satisfy $\Overlap\left(\sigma^{(i)},\sigma^{(j)}\right)\in[\beta-\eta,\beta]$, $1\le i<j\le m$. This is the $m$-OGP and it is the 
subject of Theorem~\ref{thm:main-eps-energy}.
We establish Theorem~\ref{thm:main-eps-energy} using the so-called \emph{first moment method}; 
and the smallest $m$ (for fixed $\epsilon>0$) for which this result holds true is of order $1/\epsilon$. 
While we state and prove this result for the \npp \eqref{eq:NPP-main} for simplicity, an inspection of our proof reveals that it extends to the \vbp \eqref{eq:vbp}, when $d=o(n)$. 

Note that this geometric result pertains the overlap structure of an $m$-tuple, rather than a pair, of configurations. This is necessary to cover 
all values of $\epsilon\in (0,1]$, since as we show in Theorem~\ref{thm:2-ogp}, the OGP for pairs holds only up to $\epsilon\in(1/2,1]$.
The idea of studying the $m-$OGP in order to lower the ``threshold"---as we have done---was employed in the earlier works by Rahman and Vir{\'a}g~\cite{rahman2017}, Gamarnik and Sudan~\cite{gamarnik2017performance}, and more recently by Wein~\cite{wein2020optimal}. In particular, the overlap structure we rule out is essentially the same as the one considered in~\cite{gamarnik2017performance}. Moreover, as we establish; this result holds also for a family of correlated random vectors $X_i\in\R^n$, $1\le i\le m$ rather than a single instance. This is known as the ``ensemble" variant of the OGP, and it is instrumental in proving the failure of any ``sufficiently stable" algorithm.

\paragraph{ The regime $E_n=o(n)$.} To complement our first result, we investigate the overlap structure when the exponent $E_n$ is sublinear, $E_n=o(n)$. Perhaps rather surprisingly, we establish the \emph{absence} of $m-$OGP---for $m=O(1)$---when $E_n=o(n)$. To that end, let $X\in\R^n$ be a random vector with i.i.d. standard normal coordinates. We establish that for every $E_n\in o(n)$, $m\in\mathbb{N}$, $\rho\in(0,1)$ and $\bar{\rho}\ll \rho$, it is the case that with high probability there exists an $m$-tuple $\left(\sigma^{(i)}:1\le i\le m\right)$ of spin configurations $\sigma^{(i)}\in \bincube$ such that 
they are near ground-states, namely $\left|\ip{\sigma^{(i)}}{X}\right|=O\left(\sqrt{n}2^{-E_n}\right)$, $1\le i\le m$, and their pairwise overlaps satisfy $\Overlap\left(\sigma^{(i)},\sigma^{(j)}\right)\in [\rho-\bar{\rho},\rho+\bar{\rho}]$, $1\le i<j\le m$. Namely, the overlaps ``span" the interval $[0,1]$. This is our next main result; and is the subject of Theorem~\ref{thm:ogp-absent}.

Theorem~\ref{thm:ogp-absent} is shown by using the so-called \emph{second moment method} together with a careful overcounting idea. While we state and prove this result for a \emph{single} instance $X\in\R^n$ for simplicity, it is conceivable that our technique extends also to correlated instances $X_i\in \R^n$, $1\le i\le m$ albeit perhaps at the cost of more computations and details. It is worth recalling once more that this result is shown under the assumption that $m$ is constant $O(1)$ (with respect to $n$).

Despite Theorem~\ref{thm:main-eps-energy} in discussed previously, the aforementioned \compgap of \npp still persists. That is, the exponents ``ruled out" in Theorem~\ref{thm:main-eps-energy}, $E_n=\epsilon n$ for $0<\epsilon<1$, are still far greater than the current computational limit, $\Theta\left(\log^2 n\right)$. Furthermore, in the case $E_n$ is sub-linear, $E_n = o(n)$; the $m-$OGP (for $m=O(1)$) is actually \emph{absent} as shown in Theorem~\ref{thm:ogp-absent}.

The rationale for studying the multioverlap version of the OGP ($m-$OGP), noted first by Rahman and Vir{\'a}g~\cite{rahman2017}, was the observation that studying the overlap structures of $m-$tuples (of spin configurations), as opposed to pairs, lowers the ``threshold" above which the algorithms can be ruled out. The prior work studying $m-$OGP gave high probability guarantees for the overlap structures of the $m-$tuples as the size $n$ of the problem tends to infinity, while $m$ remains constant with respect to $n$, $m=O(1)$. For instance, in the case of NAE-K-SAT problem, the OGP is shown for $m=\lceil \frac{\epsilon^2 K}{\ln K}\rceil$ to rule out densities $d\ge (1+\epsilon)2^{K-1}\ln^2 K/K$ as the number $n$ of Boolean variables tend to infinity, see \cite[Theorem~4.1]{gamarnik2017performance}. Likewise, in the context of maximum independent set problem, Wein considered the ``forbidden" structure corresponding to $K\ge 1+5/\epsilon^2$ independent sets to rule out independent sets of size $(1+\epsilon)(\log d/d)n$, again as $n\to\infty$, see \cite[Proposition~2.3]{wein2020optimal}. This is also the case for our first $m-$OGP result, where to rule out energy levels of form $2^{-\epsilon n}$ for $0<\epsilon<1$, we consider $m-$tuples with $m\sim \frac{2}{\epsilon}$. 

However, the $m-$OGP with $m=O(1)$ still falls short of going from $\Theta(n)$ all the way down to current computational threshold, $\Theta\left(\log^2 n\right)$. 
Having observed that the aforementioned \compgap still persists when $E_n=o(n)$, 
it is quite natural to ask what happens when $m$ is super-constant, $m=\omega_n(1)$. To the best of our knowledge, this line of research has not been investigated previously---presumably due to the fact that it always sufficed to take $m=O(1)$ to reach the thresholds below which polynomial-time algorithms are known to exist. In order to penetrate further into the nature of this persisting gap, we then study the $m-$OGP in the case when $m$ is super-constant, $m=\omega_n(1)$. In this regime, we establish the presence of the $m-$OGP all the way down to $E_n=\omega\left(\sqrt{n\log n}\right)$. This is the subject of Theorem~\ref{thm:m-ogp-superconstant-m}. Furthermore in Section~\ref{sec:m-ogp-beyond-is-impossible}, we give an informal argument which explains that $\omega\left(\sqrt{n\log n}\right)$ is the best exponent eliminated through this technique that one could hope for.


Like Theorem~\ref{thm:main-eps-energy}, Theorem~\ref{thm:m-ogp-superconstant-m} also pertains to the the case of the ``ensemble" variant of the OGP. We later leverage Theorem~\ref{thm:m-ogp-superconstant-m} to rule out any ``sufficiently stable" algorithm, appropriately defined.

To the best of our knowledge, ours is the first work establishing the need to consider the $m-$OGP for super-constant values of $m$. The potential gain of considering superconstant overlaps for other models is an interesting question for future research. 

\paragraph{ Failure of ``Stable" Algorithms.} We then focus on the algorithmic front, where we view an algorithm $\mathcal{A}$ (potentially randomized) 
as a mapping $\A:\R^n\to\bincube$, which takes an $X\in\R^n$ as its input (numbers/items to be partitioned) and returns a spin configuration $\A(X)$ (from which the partition is inferred). Our main algorithmic result is summarized as follows: the ``ensemble" version of the $m-$OGP with $m=\omega(1)$ (Theorem~\ref{thm:m-ogp-superconstant-m}) we have described above 
is an obstruction for any ``sufficiently stable" algorithm.  In particular, we establish the following result. Let $\epsilon \in\left(0,\frac15\right)$ be arbitrary; and $E_n$ be an energy exponent with
\[
\omega\left(n\log^{-\frac15+\epsilon} n\right)\le E_n \le o(n).
\]
Then, there exists no ``sufficiently stable" (in an appropriate sense), and potentially randomized, algorithm $\mathcal{A}$ such that with high probability, $n^{-1/2}\bigl|\ip{X}{\A(X)}\bigr|=2^{-E_n}$. Here, the probability is taken with respect to the randomness in $X\distr \mathcal{N}(0,I_n)$, as well as the coin flips of the algorithm. This is the subject of Theorem~\ref{thm:Main}. It is worth noting that the algorithm $\A$ need not be a polynomial-time algorithm: as long as $\A$ is stable in an appropriate sense, there is no restriction on its runtime. As was shown in \cite{gamarnik2020lowFOCS} stable algorithms
include many special classes of algorithms such as algorithms based on low-degree polynomials and through that the approximate
message passing type algorithms.

It is thus natural to inquire the stability property of the  algorithms known the be successful for the \npp, 
in particular the LDM algorithm which achieves the state of the 
art $n^{-\Theta(\log n)}$. We were not able to establish the stability of this algorithm, and instead resorted to simulation study which
is reported in Subsection~\ref{subsection:simulations}. The simulations are conducting by running the LDM on two correlated instances
of the \npp and measuring the overlap of the algorithm results as a function of the correlation.
The simulation results suggest that indeed the LDM algorithm is stable in the sense
we define. Curiously, it reveals additionally an interesting property. Recall the constant $\alpha = \frac{1}{2\ln 2}=0.721\dots$ which
is suggested heuristically as the leading constant in the performance of the algorithm. We discover a phase transition: when
the correlation between two instances is of the order at least approximately $1-n^{-\alpha \log_2n}$ (in other words the level of ''perturbation
is order $n^{-\alpha \log_2 n}$), the two outputs of the algorithm are identical or nearly identical. Whereas, when the correlation is smaller
than this value, there appears to be a linear discrepancy between the two outcomes. The coincidence of this phase transition with the
objective value $n^{-\alpha \log_2 n}$ is remarkable and at this stage we do not have an explanation for it.

\paragraph{ Failure of an MCMC Family.} A consequence of the $2-$OGP established in Theorem~\ref{thm:2-ogp} (which holds for energy levels $E_n = \epsilon n$, $\epsilon\in\left(\frac12,1\right]$) is the presence of a certain property, called a \emph{free energy well} (FEW), in the landscape of the \nppp. This property is known to be a rigorous barrier for a family of Markov Chain Monte Carlo (MCMC) methods~\cite{arous2020algorithmic} and has been previously employed for other average-case problems~\cite{gjs2019overlap,gamarnik2019landscape} to establish slow mixing of the Markov chain associated with the MCMC method
and thus the failure of the method. We establish the presence of a FEW in the landscape of \npp in Theorem~\ref{thm:FEW}; and leverage this property in Theorem~\ref{thm:slow-mixing} to establish the failure of a very natural class of MCMC dynamics tailored for the \nppp. More concretely, Theorem~\ref{thm:FEW} establishes the presence of the FEW of exponentially small ``Gibbs mass" in the landscape of \nppp. Theorem~\ref{thm:slow-mixing} then leverages this property, and shows that for a very natural MCMC dynamics with an appropriate initialization, it takes an exponential time for this chain to reach a region of non-trivial Gibbs mass. See the corresponding  section for further details. 
\paragraph{ Study of Local Optima.} Our final focus is on the local optima of this model. For any spin configuration $\sigma\in\bincube$, denote by $\sigma^{(i)}\in\bincube$, $1\le i\le n$, the configuration obtained by flipping the $i-$th bit of $\sigma$. A spin configuration $\sigma$ is called a \emph{local optimum} if $\left|\ip{\sigma^{(i)}}{X}\right|\ge \left|\ip{\sigma}{X}\right|$ for $1\le i\le n$. Namely, $\sigma$ is a local optimum if 
the ``swapping" the place (with respect to $\sigma$) of any ``item" returns a worse partition. To further complement our landscape analysis in the ``hard" regime, $E_n=\Theta(n)$, we study the expectation of number $N_\epsilon$ of local optima with energy value $O\left(\sqrt{n}2^{-n\epsilon}\right)$, $0<\epsilon<1$. We show that this expectation is exponential in $n$, and we give a precise, linear, trade-off between the ``exponent" of $\mathbb{E}[N_\epsilon]$ and $\epsilon$. This is the subject of Theorem~\ref{thm:local-opt}.  This suggests that a very simple greedy algorithm, starting from an arbitrary $\sigma\in\bincube$ and proceeding by flipping a single spin so as to ``reduce energy" as long  as there is such a spin, will likely fail to find a ground-state solution for the \nppp. 

This analysis is inspired by the work of Addario-Berry et al.~\cite{addario2019local} who carried out an analogous analysis for the local optima of the Hamiltonian of the Sherrington-Kirkpatrick spin glass model.
\subsection*{Overview of Our Techniques}
\paragraph{ Presence of the OGP.} We establish the presence of the overlap gap property (Theorems~\ref{thm:2-ogp}, \ref{thm:main-eps-energy}, and \ref{thm:m-ogp-superconstant-m}) using the so-called first moment method. Specifically, we let a certain random variable count the number of tuples (either pairs, or $m-$tuples with $m=O(1)$ or $m=\omega_n(1)$) of near ground-state spin configurations with a prescribed overlap pattern. We then show that expectation of these random variables are exponentially small, establishing the presence of the OGP via Markov inequality. 
At a technical level, this requires a delicate analysis of a certain covariance structure governing the joint probability. 

\paragraph{ Absence of the OGP.} In the regime $E_n=o(n)$, we establish in Theorem~\ref{thm:ogp-absent} that the $m-$OGP (for $m=O(1)$) is absent. 
That is, $m$-tuples with a prescribed pairwise overlaps are attainable for every overlap level. 
This is done by letting a certain random variable count the number of allowed configurations like above; and then using the so-called \emph{second moment method}. In addition, the proof requires a novel overcounting idea, in order to ``decorrelate" pairs of tuples of spin configurations encountered during the second moment computation. Again at a technical level, the proof also requires a delicate analysis of a block covariance matrix; as well as a probabilistic method argument. 

\paragraph{ Failure of Stable Algorithms.} Our theorem~\ref{thm:Main}  establishing the failure of stable algorithms (appropriately defined) is arguably the most technically involved proof; and combines many different ideas, including the $m-$OGP result shown in Theorem~\ref{thm:m-ogp-superconstant-m} and certain concentration inequalities. Furthermore,  interestingly, the proof also uses ideas from the extremal combinatorics and Ramsey Theory. In particular, see Theorems~\ref{thm:ramsey} and~\ref{thm:ramsey-1}; and Propositions~\ref{thm:clique-exist} and~\ref{prop:G-mc-clique}. This is necessary so as to generate a forbidden configuration which contradicts with the OGP. 
In order to guide the reader, we provide in Section~\ref{sec:pf-thm-Main} a brief outline of the proof.
\paragraph{ Failure of an MCMC Family.} We establish the failure of the MCMC families by first establishing the so-called FEW property. This is shown by i) leveraging our $2-$OGP result, Theorem~\ref{thm:2-ogp}, and ii) using a slightly more refined property on the energy landscape of \nppp, borrowed from \cite{karmarkar1986probabilistic}. The proof is a rather direct one. The failure of the MCMC family, Theorem~\ref{thm:slow-mixing}, uses fairly routine arguments; but included nevertheless in full detail for completeness.
\paragraph{Paper Organization.} The rest of the paper is organized as follows. Our main results regarding the geometry of the energy landscape of \npp are found in Section~\ref{sec:main:geo-results}. Specifically, our result establishing the presence of the $m-$OGP for $m=O(1)$ for energy levels $2^{-\Theta(n)}$ is presented in  Section~\ref{sec:m-ogp-O(1)-present}; our result showing the absence of $m-$OGP for energy levels $2^{-o(n)}$ is presented in Section~\ref{sec:m-ogp-absent}; our result showing the presence of $m-$OGP for energy levels $2^{-E_n}$ with $\omega\left(\sqrt{n\log_2 n}\right)\le E_n\le o(n)$, for $m=\omega_n(1)$ is presented in Section~\ref{sec:super-constant-ogp-present}. Our result on the the expected number of local optima is found in Section~\ref{sec:exp-local-opt}. The failure of stable algorithms (appropriately defined) is shown in Section~\ref{sec:failure}. The same section contains the simulation results.
The limitations of our proofs for establishing the $m-$OGP in the case  when $m$ is super-constant, $m=\omega_n(1)$, are studied in Section~\ref{sec:m-ogp-beyond-is-impossible}. We briefly recapitulate our conclusions and outline several interesting open problems and future research directions in Section~\ref{sec:conclusion}. Finally, the proofs of all of our results are presented in Section~\ref{sec:proofs}.
\paragraph{Notation.}  The set of real numbers is denoted by $\R$. The sets $\N$ and $\Z_+$ denote the set of positive integers. For any $N\in\mathbb{N}$, the set $\{1,2,\dots,N\}$ is denoted by $[N]$. For two sets $A,B$; their Cartesian product $\{(a,b):a\in A,b\in B\}$ is denoted by $A\times B$. For any set $A$, $|A|$ denotes its cardinality. For any $r\in\mathbb{R}$, the largest integer not exceeding $r$ (that is, the floor of $r$) is denoted by $\lfloor r\rfloor$; and the smallest integer not less than $r$ (that is, the ceiling of $r$) is denoted by $\lceil r\rceil$. For any $X=(X_i:1\le i\le n)\in\R^n$, its Euclidean $\ell_2$ norm, $\sqrt{\sum_{1\le i\le n}X_i^2}$ and its Euclidean $\ell_\infty$ norm, $\max_{1\le i\le n}|X_i|$ are denoted respectively by $\|X\|_2$ and $\|X\|_\infty$. For any $X,Y\in\R^n$, their Euclidean inner product, $\sum_{1\le i\le n}X_iY_i$, is denoted by $\ip{X}{Y}$. The symbol $\ind\{\mathcal{E}\}$ denotes the indicator of $\mathcal{E}$, which is equal to one if $\mathcal{E}$ is true; and equal to zero if $\mathcal{E}$ is false. $\bincube$ denotes the discrete cube $\{-1,1\}^n$. For any $\sigma,\sigma'\in\bincube$, their Hamming distance $\sum_{1\le i\le n}\ind\{\sigma_i\ne \sigma_i'\}$ is denoted by $d_H\left(\sigma,\sigma'\right)$, their normalized overlap $n^{-1}\left|\ip{\sigma}{\sigma'}\right|$ is denoted by $\Overlap\left(\sigma,\sigma'\right)$; and their normalized inner product $n^{-1}\ip{\sigma}{\sigma'}$ is denoted by $\OBar\left(\sigma,\sigma'\right)$. $\log$ and $\log_2$ denote respectively the logarithms with respect to base $e$ and with respect to base $2$. For any $r\in\R$, $2^r$ is denoted by $\exp_2(\alpha)$; and $e^r$ is denoted by $\exp(r)$. Binary entropy function (that is, the entropy of a Bernoulli random variable with parameter $p$) is denoted by $h(p)=-p\log_2 p-(1-p)\log_2(1-p)$. $\mathcal{N}(0,1)$ denotes the standard normal random variable; and $\mathcal{N}(0,I_n)$ denotes the distribution of a random vector $X=(X_i:1\le i\le n)\in\R^n$ where $X_i\distr\mathcal{N}(0,1)$, i.i.d. For any matrix $\mathcal{M}$, we denote its Frobenius norm, spectral norm, spectrum, smallest singular value, largest singular value, determinant, and  trace by $\|\mathcal{M}\|_F$, $\|\mathcal{M}\|_2$, $\sigma(\mathcal{M})$, $\sigma_{\min}(\mathcal{M})$, $\sigma_{\max}(\mathcal{M})$, $|\mathcal{M}|$, and ${\rm trace}(\mathcal{M})$, respectively. A graph $\Gr=(V,E)$ is a collection of vertices $V$ with some edges $(v,v')\in E$ between $v,v'\in E$. In the sequel, we consider only \emph{simple} graphs, that is, graphs that are undirected with no loops. A clique is a complete graph, that is a graph $\G=(V,E)$ where for every distinct $v,v'\in V$; $(v,v')\in E$. The clique on $m-$vertices is denoted by $K_m$. A subset $S\subset V$ of vertices (of $\G$) is called an independent set if for every distinct $v,v'\in S$; $(v,v')\notin E$. The largest cardinality of such an independent set is called the independence number of $\G$; and is denoted by $\alpha(\G)$. A $q-$coloring of a graph $\Gr=(V,E)$ is a function $\varphi:V\to\{1,2,\dots,q\}$ assigning to each edge of $\Gr$ one of $q$ available colors. 

We employ the standard Bachmann-Landau asymptotic notation, e.g. $\Theta(\cdot),O(\cdot),o(\cdot)$, $\omega(\cdot)$, and $\Omega(n)$ throughout the paper. Whenever a function $f(n)$, say, has growth $o(n)$, we either denote by $f(n)=o(n)$ or $f(n)\in o(n)$. Finally, whenever $f$ has a lower and upper bound on its growth, we abuse the notation slightly and use inequalities. For instance, when $f(n)\in \omega_n(1)$ and $f(n)\in o(n)$ (that is, $f$ is super-constant but sub-linear), we often find it convenient to write $\omega_n(1)\le f(n)\le o(n)$.

Finally, in order to keep our presentation simple, we omit all floor and ceiling operators.
\section{Main Results. The Landscape of the \npp}\label{sec:main:geo-results}
In this section, we present our results regarding the geometry of the energy landscape of the number partitioning problem (\nppp). 

Our results concern the overlap structures of the tuples of near ground-state configurations, formalized next.

\begin{definition}\label{def:overlap-set}
Fix an $m\in\mathbb{N}$, and $0<\eta<\beta<1$. Let $X_i\distr \mathcal{N}(0,I_n)$, $0\le i\le m$, be i.i.d.\,random vectors; and let $\mathcal{I}$ be any subset of $[0,1]$.  Denote by $\mathcal{S}(\beta,\eta,m,E_n,\mathcal{I})$ the set of all $m-$tuples $\left(\sigma^{(i)}:1\le i\le m\right)$ of spin configurations $\sigma^{(i)}\in \bincube$, such that the following holds:
\begin{itemize}
    \item[(a)] {\bf (Pairwise Overlap Condition) } For any $1\le i<j\le m$,
    \[
    \beta-\eta\le \Overlap\left(\sigma^{(i)},\sigma^{(j)}\right) \le \beta
    \]
    where $\Overlap(\sigma^{(i)},\sigma^{(j)})\triangleq n^{-1}\left|\ip{\sigma^{(i)}}{\sigma^{(j)}}\right|$ is the (normalized) overlap between spin configurations $\sigma^{(i)},\sigma^{(j)}\in\bincube$. 
    \item[(b)] {\bf (Near Ground-State Condition) } There exists $\tau_i\in\mathcal{I}$, $1\le i\le m$, such that 
    \[
    \frac{1}{\sqrt{n}}\left|\ip{\sigma^{(i)}}{Y_i\left(\tau_i\right)}\right|\le 2^{-E_n}, \quad\text{where}\quad Y_i(\tau_i)=\sqrt{1-\tau_i^2}X_0+\tau_i X_i,\quad\text{for}\quad 1\le i\le m.
    \]
\end{itemize}
\end{definition}
Here, $m$ refers to size of the tuple we investigate; the quantities $\beta$ and $\eta$ control the overlap region; $E_n$ controls the ``exponent" of the energy level $2^{-E_n}$ with respect to which $\sigma^{(i)}\in \bincube$ are near ground-state; and $\mathcal{I}$ is a certain index set for describing the correlated instances (more on this later). 

The set $\mathcal{S}(\beta,\eta,m,E_n,\mathcal{I})$ is the set of all $m-$tuples of spin configurations $\sigma^{(i)}\in\bincube$,  $1\le i\le m$; where i) the pairwise overlaps between $\sigma^{(i)}$ lie in the interval $[\beta-\eta,\beta]$; and ii) each $\sigma^{(i)}$, $1\le i\le m$, is a (near) ground-state with respect to an instance of the \npp dictated by the entries of the vector $Y_i(\tau_i)\in \R^n$. Note that the instances $Y_i\left(\tau_i\right)$ with respect to which $\sigma^{(i)}$ are near-optimal need not be the same; each individually distributed as $\mathcal{N}(0,I_n)$; and are correlated. This will later turn out to be useful in ruling out ``sufficiently stable" algorithms, appropriately defined.

Our main results are now in order.
\subsection{Overlap Gap Property for the Energy Levels $2^{-\Theta(n)}$}\label{sec:m-ogp-O(1)-present}
We start by recalling from the introduction that the number partitioning problem (\nppp) exhibits a \compgapp: while the optimal value of \eqref{eq:NPP-main} for the i.i.d. standard normal inputs is $\Theta\left(\sqrt{n}2^{-n}\right)$ w.h.p.; the best known polynomial-time algorithm achieves a performance of only $2^{-\Theta(\log^2 n)}$. In terms of the exponent $E_n$ of the energy level $2^{-E_n}$---a quantity that is more convenient to work with---the existential methods guarantee an exponent of $E_n=n$, whereas the best known polynomial-time algorithm achieves $E_n=\Theta\left(\log^2 n\right)$.

In this section, we establish certain geometric properties regarding the overlaps of the tuples of near ground-state configurations of the \nppp; in order the better comprehend the nature of the aforementioned gap. 

Our first focus is on the pairs of near ground-state configurations; and on the energy levels $E_n = \epsilon n$ where $\epsilon\in\left(\frac12,1\right]$.
\begin{theorem}\label{thm:2-ogp}
Let $X\distr \mathcal{N}(0,I_n)$; and $\epsilon\in\left(\frac12,1\right]$ be arbitrary. Then, there exists a $\rho\triangleq \rho(\epsilon)\in(0,1)$ such that with probability $1-\exp\left(-\Theta(n)\right)$, there are no pairs $\left(\sigma,\sigma'\right)\in\bincube\times \bincube$ of spin configurations for which $\Overlap\left(\sigma,\sigma'\right)\in\left[\rho,\frac{n-2}{n}\right]$; $\frac{1}{\sqrt{n}}\left|\ip{\sigma}{X}\right| = O\left(2^{-n\epsilon}\right)$, and $\frac{1}{\sqrt{n}}\left|\ip{\sigma'}{X}\right| =O\left(2^{-n\epsilon}\right)$.
\end{theorem} The proof of Theorem~\ref{thm:2-ogp} is based on a first moment argument, and is provided in Section~\ref{sec:pf-thm:2-ogp}. 

Several remarks are now in order. Theorem~\ref{thm:2-ogp} establishes that for the energy levels $E_n=\epsilon n$ with $\epsilon\in\left(\frac12,1\right]$, it is the case that with high probability, the overlap of any pair of near ground-state spin configurations exhibits a gap. Namely, in the language of Definition~\ref{def:overlap-set}, Theorem~\ref{thm:2-ogp} establishes that for any  $\epsilon\in\left(\frac12,1\right]$ there exists a $\rho\triangleq \rho(\epsilon)\in (0,1)$ such that the set $\mathcal{S}(\beta,\eta,m,E_n,\mathcal{I})$ with parameters $\beta =1-\frac2n$, $\eta = 1-\frac2n-\rho$, $m=2$, $E_n=2^{-\epsilon n}$ and $\mathcal{I}=\{0\}$ is empty with probability at least $1-\exp\left(-\Theta(n)\right)$.

Note that the rightmost end of this gap is independent of $\epsilon$; and reaches $\frac{n-2}{n}$: this is the largest overlap that can be attained by two spin configurations $\sigma,\sigma'$ with $\sigma\ne\pm \sigma'$. That is, if $1\le d_H\left(\sigma,\sigma'\right)\le n-1$ then $\Overlap\left(\sigma,\sigma'\right)\le \frac{n-2}{n}$ with equality if and only if $d_H\left(\sigma,\sigma'\right)\in\{1,n-1\}$. 

We will later use our pairwise OGP result, Theorem~\ref{thm:2-ogp}, to establish the existence of what is known as a \emph{free energy well} (FEW), which is a provable barrier for the Markov chain type algorithms (such as the Glauber Dynamics). While proving the existence of a FEW, we will leverage the feature that the prohibit overlap region extends all the way up to $\frac{n-2}{n}$.

Observe that the energy levels $E_n$ ``ruled out" per Theorem~\ref{thm:2-ogp} are still far above the current computational threshold, $\Theta\left(\log^2 n\right)$. In order to better comprehend the aforementioned \compgap of the \nppp; and to address the energy levels $E_n=\Theta(n)$, that is, energy levels of form $E_n=\epsilon n$, where $\epsilon\in (0,1]$ is a constant independent of $n$; we now consider $m-$tuples of near ground-state configurations. We will establish a similar geometric property, this time regarding the overlaps of $m-$tuples of near ground-state configurations of the \nppp.

Our next main result shows that the \npp exhibits $m-$\emph{Overlap Gap Property} ($m-$OGP)---for constant $m$, $m=O(1)$---for such energy levels. 
\begin{theorem}\label{thm:main-eps-energy}
Let $\epsilon>0$. Then there exists an $m\triangleq m(\epsilon)\in\mathbb{N}$; $\beta$, and $\eta$ with $0<\eta<\beta<1$ such that the following holds. For i.i.d. random  vectors $X_i\in\R^n$, $0\le i\le m$ with distribution $\mathcal{N}(0,I_n)$ and any subset $\mathcal{I}\subset [0,1]$ with $|\mathcal{I}|=2^{o(n)}$, 
\[
\mathbb{P}\Bigl(\mathcal{S}\left(\beta,\eta,m,\epsilon,\mathcal{I}\right)\ne\varnothing\Bigr)\le \exp_2\left(-\Theta\left(n\right)\right).
\]
Here $\mathcal{S}\left(\beta,\eta,m,\epsilon,\mathcal{I}\right)$ is a shorthand notation for the set $\mathcal{S}\left(\beta,\eta,m,E_n,\mathcal{I}\right)$ with $E_n = n\epsilon$ introduced in the Definition~\ref{def:overlap-set}.
\end{theorem}
The proof of Theorem~\ref{thm:main-eps-energy} is based on a first moment argument; and is provided in Section~\ref{sec:pf-eps-m-OGP}.

Several important remarks are now in order. Theorem~\ref{thm:main-eps-energy} asserts that for any $\epsilon>0$, the \npp indeed exhibits the $m-$OGP for energy level $2^{-\epsilon n}$ for an appropriate $m\in\mathbb{N}$: there exists $1>\beta>\eta>0$ such that with high probability, it is the case that for any $m-$tuple of near ground-state spin  configurations $\sigma^{(i)}\in \bincube$, $1\le i\le m$, one  can find indices $1\le i<j\le m$ such that $\Overlap\left(\sigma^{(i)},\sigma^{(j)}\right)\in [0,\beta-\eta)\cup (\beta,1]$. Our analysis reveals that for a fixed $\epsilon>0$, taking $m\sim \frac{2}{\epsilon}$ suffices. That is, roughly speaking; one can establish $m-$OGP for energy levels $\exp_2\left(-\frac2m n\right)$. The proof will use $\eta$ which is much smaller than $\beta$. Hence the structure we rule out can be viewed as a near ground-state configuration consisting of (nearly) equidistant points of $\bincube$. 

It is important to highlight that Theorem~\ref{thm:main-eps-energy} studies the overlap structure of $m-$tuples spin configurations. This is in contrast to Theorem~\ref{thm:2-ogp} which studies the overlap structure of pairs. The study of the overlap structures of $m-$tuples is necessary, in order to cover the regime $E_n=\epsilon n$ for $\epsilon\in(0,1]$: if we consider instead the overlap structure for pairs, then the OGP can be established only for very low energy levels $2^{-\epsilon n}$ where $\epsilon \in (\frac12,1]$, as was shown in Theorem~\ref{thm:2-ogp}. 

Furthermore, Theorem~\ref{thm:main-eps-energy} pertains the ``ensemble" variant of the OGP: the spin configurations $\sigma^{(i)}$, $1\le i\le m$, need not be near ground-states for the same instance of the problem; and are  near ground-states for potentially correlated instances. This idea was employed also in other works \cite{gamarnik2021overlap,gamarnik2020lowFOCS,wein2020optimal}. Using the ensemble variant of OGP, it appears possible that virtually any sufficiently stable algorithm can be ruled out (here, ``stability" refers to the property that a small change in the input of the algorithm induces a small change in the output). 

\begin{remark}\label{remark:m-ogp-applies-to-vbp-aswell}
It is worth noting that while we state and prove Theorem~\ref{thm:main-eps-energy} for the \npp \eqref{eq:NPP-main} for simplicity; our result still remains valid for the high-dimensional version, \vbp\eqref{eq:vbp}. More concretely, recalling that the optimal value of \eqref{eq:vbp} for random i.i.d. standard normal inputs $X_i\distr \mathcal{N}(0,I_d)$, $1\le i\le n$ is $\Theta\left(\sqrt{n}2^{-n/d}\right)$ for $\omega(1)\le d\le o(n)$; our approach still remains valid and the $m-$OGP still takes place for energy levels $\Theta\left(\sqrt{n}2^{-\epsilon n/d}\right)$ for any $\epsilon\in(0,1]$. 
\end{remark}
\subsection{Absence of $m$-Overlap Gap Property for Energy Levels $2^{-o(n)}$}\label{sec:m-ogp-absent}
The energy exponents (of form $\epsilon n$ for $0<\epsilon\le 1)$ that Theorem~\ref{thm:main-eps-energy} rules out are still far above the current best computational result, $\Theta\left(\log^2 n\right)$. In order to handle this issue, we now focus our attention to the sub-linear exponent regime, $E_n=o(n)$. 

Perhaps rather surprisingly, we establish that the $m$-OGP is actually \emph{absent} in this regime, when $m$ is constant with respect to $n$, $m=O(1)$. That is, the overlaps ``span" the entire interval, in an certain sense concretized as follows.
\begin{theorem}\label{thm:ogp-absent} 
Let $X\distr \mathcal{N}\left(0,I_n\right)$. Fix any $\eta>0$ and $m\in\mathbb{N}$. Suppose that $f(n):\mathbb{N}\to \R^+$ is any arbitrary function with  $f(n)=\omega_n(1)$ and $f(n)=o(n)$. Then,
\[
\lim_{n\to\infty}\mathbb{P}\Bigl(\forall \beta\in[0,1]:\mathcal{S}\left(m,\beta,\eta,f(n),\{0\}\right)\ne \varnothing\Bigr) =1;
\]
where the set $\mathcal{S}$ is introduced in Definition~\ref{def:overlap-set} with the following modification on the pairwise overlap condition: $\beta-\eta  \le \Overlap\left(\sigma^{(i)},\sigma^{(j)}\right)\le \beta+\eta$ for $1\le i<j\le m$. 
\end{theorem}
In particular, since $\eta$ in the statement of Theorem~\ref{thm:ogp-absent} is arbitrary, we conclude that the overlaps indeed ``span" the entire interval. The $m-$tuples $\left(\sigma^{(i)}:1\le i\le m\right)$ that we consider in Theorem~\ref{thm:ogp-absent} consist of spin configurations that are near ground-state with respect to the same instance of the problem: $n^{-1/2}\left|\ip{\sigma^{(i)}}{X}\right| \le \exp_2(-f(n))$ for $1\le i\le m$. Moreover, our proof will demonstrate something stronger: one can find such $m-$tuples $\sigma^{(i)} \in\bincube$, $1\le i\le m$ satisfying not only the constraints on absolute values of inner products but inner products themselves: $\beta-\eta \le \frac1n\ip{\sigma^{(i)}}{\sigma^{(j)}}\le \beta+\eta$.   The slight modification of Definition~\ref{def:overlap-set} (where the interval $[\beta-\eta,\beta+\eta]$ is considered instead of $[\beta-\eta,\beta]$) is for convenience. 


The proof of Theorem~\ref{thm:ogp-absent} uses the probabilistic method and the second moment method together with a crucial overcounting idea; and is provided in Section~\ref{sec:pf-absent-OGP}.
\subsection{$m$-Overlap Gap Property Above $2^{-\Theta(n)}$: Super-Constant $m$}\label{sec:super-constant-ogp-present}

We now establish the existence of $m-$OGP, where $m$ is super-constant, $m=\omega_n(1)$, for certain energy levels whose exponents are sub-linear.
\begin{theorem}\label{thm:m-ogp-superconstant-m}
Let $E_n:\mathbb{N}\to\R^+$ be any arbitrary ``energy exponent" with growth condition 
\[
E_n \in\omega\left(\sqrt{n\log_2 n}\right)\quad\text{and} \quad E_n \in o(n).
\]
Suppose that $X_i\in\R^n$, $1\le i \le m$, are i.i.d. with distribution $\mathcal{N}(0,I_n)$, and $\mathcal{I}\subset[0,1]$ with $|\mathcal{I}|=n^{O(1)}$. 
Define the sequences $(m_n)_{n\ge 1}$, $(\beta_n)_{n\ge 1}$ and $(\eta_n)_{n\ge 1}$ with $1>\beta_n>\eta_n>0$, $n\ge 1$ by
\begin{equation}\label{eq:beta-eta-m-n-main-supconstant}
m_n \triangleq \frac{2n}{E_n},\quad \beta_n\triangleq 1-\frac{2g(n)}{E_n},\quad\text{and}\quad \eta_n = \frac{g(n)}{2n},
\end{equation}
where $g(n)$ is any arbitrary function with growth condition
\begin{equation}\label{g-n-appear-main-thm}
g(n) \in \omega(1) \quad\text{and}\quad g(n) = o\left(\frac{E_n^2}{n\log_2 n}\right). 
\end{equation}
Then, 
\begin{equation}\label{eq:thesis-m-ogp-superconstant-m}
\mathbb{P}\Bigl(\mathcal{S}\left(\beta_n,\eta_n,m_n,E_n,\mathcal{I}\Bigr)\ne\varnothing\right)\le \exp\left(-\Theta(n)\right).
\end{equation}
Here, $\mathcal{S}\left(\beta_n,\eta_n,m_n,E_n,\mathcal{I}\right)$ is the set introduced in Definition~\ref{def:overlap-set} with the modification that the pairwise inner products (as opposed to the overlaps) are constrained, that is
\[
\beta-\eta\le \frac1n \ip{\sigma^{(i)}}{\sigma^{(j)}}  \le \beta, \quad\text{for}\quad 1\le i<j\le m.
\]
Moreover, in the special case where
\[
E_n \in \omega\left(n\cdot \log^{-\frac15+\epsilon} n\right) \quad\text{and}\quad E_n = o(n)
\]
(with $\epsilon\in(0,\frac15)$ being arbitrary),~\eqref{eq:thesis-m-ogp-superconstant-m} still remains valid with $g(n)$ satisfying
\begin{equation}\label{eq:main-g-n-for-main-failure-result}
g(n) = n\cdot \left(\frac{E_n}{n}\right)^{2+\frac{\epsilon}{8}}. 
\end{equation}
\end{theorem}
The idea of the proof of Theorem~\ref{thm:m-ogp-superconstant-m} is quite similar to that of Theorem~\ref{thm:main-eps-energy}, yet it does not follow directly from Theorem~\ref{thm:main-eps-energy}. This is due to the fact that Theorem~\ref{thm:m-ogp-superconstant-m} uses different asymptotic bounds for certain cardinality terms; and requires a more careful asymptotic analysis. For this reason, we provide a separate and complete proof in 
Section~\ref{sec:m-ogp-superconstant-m}.

Several remarks are in order. In what follows, we suppress the subscript $n$ from $m_n,\beta_n,\eta_n$; while the reader should keep in his mind that all these quantities are functions of $n$. 

Theorem~\ref{thm:m-ogp-superconstant-m} states that the $m-$OGP still takes place in some portion of the sub-exponential energy regime $2^{-E_n}$, specifically when the exponent $E_n$ satisfies
\[
\omega\left(\sqrt{n\log_2 n}\right)\le E_n \le o(n);
\]
provided $m$ is super-constant, $m=\omega_n(1)$. 

Furthermore, analogous to Theorem~\ref{thm:main-eps-energy}; Theorem~\ref{thm:m-ogp-superconstant-m} also pertains to the ``ensemble" variant of the OGP, where $\sigma^{(i)}\in \bincube$, $1\le i\le m$, need not be near ground-states with respect to the \emph{same} instance of the problem. Later in Section~\ref{sec:failure}, we will indeed leverage this feature and rule out ``sufficiently stable" algorithms, appropriately defined, for the \nppp. 

We treat the case $E_n = \omega(n\cdot \log^{-1/5 +\epsilon} n)$ with a different choice of $g(n)$ (though keeping $m,\beta,\eta$---as functions of $g(n)$---the same). In particular, in this case the $g(n)$ parameter (hence the $\eta$ parameter) appearing in Theorem~\ref{thm:m-ogp-superconstant-m} can be taken to be larger. Note that this yields a stronger conclusion: it implies that the length $\eta$ of the forbidden region is larger. Later in Theorem~\ref{thm:Main}, we will leverage Theorem~\ref{thm:m-ogp-superconstant-m} for the case $E_n = \omega(n\cdot \log^{-1/5+\epsilon} n)$ with $g(n)$ chosen as in~\eqref{eq:main-g-n-for-main-failure-result} (and $\beta,\eta,m$ prescribed according to~\eqref{eq:beta-eta-m-n-main-supconstant}) to rule out sufficiently stable algorithms, appropriately defined. 

Our next remark pertains to the size of the index set, $\mathcal{I}$. While we restricted our attention to sets with $|\mathcal{I}|\triangleq |\mathcal{I}(n)| = n^{O(1)}={\rm poly}(n)$, it appears that our technique still remains valid, so long as
$|\mathcal{I}(n)|\le 2^{CE_n}$, where $C$ is a small enough constant.


We now comment on the energy exponent, $E_n$. For Theorem~\ref{thm:m-ogp-superconstant-m} to hold true, $E_n$ should grow faster than $\omega\left(\sqrt{n\log_2 n}\right)$. While we do not rule out the OGP for smaller values of $E_n$, we will provide an argument which shows that $\omega\left(\sqrt{n\log_2 n}\right)$ is tight. It uses the first moment argument employed here. Whether this growth rate is indeed tight is left as an open problem. See Section~\ref{sec:m-ogp-beyond-is-impossible} for more details.


\subsubsection*{An Illustration of Theorem~\ref{thm:m-ogp-superconstant-m} with a Concrete Choice of Parameters.}
We now illustrate Theorem~\ref{thm:m-ogp-superconstant-m} with a concrete choice of the energy exponent $E_n$ and concrete choices of parameters $m,\beta$, and $\eta$.  

Fix $\delta\in(0,\frac12)$ and consider ``ruling out" the energy levels $E_n = n^{1-\delta}$. That is, our goal is to establish the presence of the $m$-OGP for $E_n = n^{1-\delta}$ for appropriate $m,\beta$, and $\eta$ parameters.  Next, take $m=2n/E_n = 2n^\delta$, per~\eqref{eq:beta-eta-m-n-main-supconstant}. Choose a $\delta'>0$ such that $\delta'+2\delta<1$. Then, set $g(n) = n^{\delta'}$. It is easily verified that
 \[
 g(n) = \omega_n(1)\quad\text{and}\quad g(n) = o\left(\frac{E_n^2}{n\log_2 n}\right) = o\left(\frac{n^{1-2\delta}}{\log_2n}\right).
 \]
 We then take, again per~\eqref{eq:beta-eta-m-n-main-supconstant},
 \[
 \beta_n = 1 - \frac{2g(n)}{E_n} = 1-2n^{\delta'+\delta-1}\quad \text{and}\quad \eta_n = \frac{g(n)}{2n} = \frac12 n^{-1+\delta'}.
 \]
Note that the overlap region, $[\beta-\eta,\beta]$, has a length $\eta$.
For the statement of the theorem to be non-vacuous, $n$ times the overlap length must contain some integer values: $\left|[n\beta-n\eta,n\beta]\cap \mathbb{Z}\right|=\Omega(1)$ must hold.  We verify that indeed $n\eta = \frac12 n^{\delta'}=\omega_n(1)$. Namely, $n$ times the length of the overlap interval grows polynomially in $n$. 
\subsection{Expected Number of Local Optima}\label{sec:exp-local-opt}
In this section, we complement our earlier analysis in the ``hard" regime, $2^{-\Theta(n)}$. Specifically, we focus on the local optima at these  energy levels. 

\begin{definition}\label{def:loc-opt}
Let $\sigma\in\bincube$ be a spin configuration. For every $1\le i\le n$, denote by $\sigma^{(i)}\in \bincube$ the spin configuration obtained by flipping $i-$th bit of $\sigma$. That is, $\sigma^{(i)}(i) = -\sigma(i)$ and $\sigma^{(i)}(j) = \sigma(j)$ for  $j\ne i$. Given $X\in\R^n$, a spin configuration  $\sigma\in\bincube$ is called a {\bf local optimum} if
\[
\left|\ip{\sigma^{(i)}}{X}\right|\ge \left|\ip{\sigma}{X}\right|,\quad\text{for}\quad 1\le i\le n.
\]
\end{definition}
For energy exponents $E_n$ of form $\epsilon n$, $0<\epsilon<1$, we now compute the expected number of local optima below the energy level $2^{-E_n}$.
\begin{theorem}\label{thm:local-opt} 
Let $X\distr\mathcal{N}(0,I_n)$. Fix any $\epsilon\in(0,1)$, and let $N_\epsilon$ be the number of spin configurations $\sigma\in\bincube$ which satisfies the following:
\begin{itemize}
    \item[(a)] $\sigma$ is a local optimum in the sense of Definition~\ref{def:loc-opt}.
    \item[(b)]$\displaystyle \frac{1}{\sqrt{n}}|\ip{\sigma}{X}| = O(2^{-n\epsilon})$. 
\end{itemize}
Then, $
\displaystyle\lim_{n\to\infty}\frac1n \log \E{N_\epsilon} = 1-\epsilon$.
\end{theorem}
The proof of Theorem~\ref{thm:local-opt} is provided in Section~\ref{sec:pf-thm-local-opt}.

Several remarks are now in order. First, the notion of local optimality per Definition~\ref{def:loc-opt} is the same one that Addario-Berry et al. considered in~\cite{addario2019local}. In particular, they study the local optima of the Hamiltonian of the Sherrington-Kirkpatrick spin glass; and carry out a very similar analysis---namely they show the expected number of local optima is exponentially large, and compute the exponent (though the proofs corresponding to these models are different).

Second, Theorem~\ref{thm:local-opt} gives a precise trade-off between the exponent of the energy value and the exponent of the expectation: the exponent of the (expected) number of local optima decays linearly in the exponent $\epsilon$ of the energy level, as $\epsilon$ varies in $(0,1)$. In particular, the expected number of the local optima is exponential with exponent growing linearly as the energy level moves away from the energy of the ground state. 

Third, Theorem~\ref{thm:local-opt} suggests the likely failure of a very simple, yet natural, greedy algorithm. Consider a greedy algorithm which starts from a spin configuration $\sigma\in\bincube$ and performs a sequence of local, greedy, moves: at each step, flip a spin configuration that decreases the energy, $\left|\ip{\sigma}{X}\right|$. This greedy algorithm continues until one cannot move any further, therefore reaching a local optimum, in the sense of Definition~\ref{def:loc-opt}. Theorem~\ref{thm:local-opt} shows that there exists, in expectation, exponentially many such local optima. This suggests that the greedy algorithm will likely to fail in finding a ground-state solution. 
\section{Main Results. Failure of Algorithms}
\subsection{$m-$Overlap Gap Property Implies Failure of Stable Algorithms}\label{sec:failure}
Our focus in this section is to understand how well one can solve the optimization problem \eqref{eq:NPP-main} via ``sufficiently stable" algorithms, when the input $X_i$, $1\le i\le n$, consists of i.i.d. standard normal weights. 

\paragraph{ Algorithmic Setting.} We interpret an algorithm $\mathcal{A}$ as a mapping from the Euclidean space $\R^n$ to the binary cube $\bincube\triangleq\{-1,1\}^n$. We also allow $\A$ to be potentially randomized. More concretely, we assume that there exists a probability space $\left(\Omega,\mathbb{P}_\omega\right)$, such that $\A:\R^n\times \Omega \to \bincube$ and for every $\omega\in\Omega$, $\A(\cdot,\omega):\R^n\to\bincube$. Here, $X\in\R^n$ denotes the ``items" to be partitioned; whereas for a fixed $\omega\in\Omega$, $\A(X,\omega)\in\bincube$ is the spin configuration returned by this potentially randomized algorithm, $\A$; which encodes a partition. 

We now formalize the class of ``sufficiently stable" algorithms we study herein by specifying the relevant performance parameters. 
\begin{definition}\label{def:optimal-alg}
Let $E>0$; $f\in\mathbb{N}$, $L\in\mathbb{R}^+$ and $\rho',p_f,p_{\rm st}\in[0,1]$. A randomized algorithm $\A:\R^n\times \Omega\to\bincube$ for the \npp \eqref{eq:NPP-main} is called $\left(E,f,L,\rho',p_f,p_{\rm st}\right)-$optimal if the following are satisfied.
\begin{itemize}
    \item {\bf (Near-Optimality)} For $\left(X,\omega\right)\sim \mathcal{N}(0,I_n)\otimes \mathbb{P}_\omega$, 
    \[
    \mathbb{P}_{\left(X,\omega\right)}\left(\frac{1}{\sqrt{n}}\left|\ip{X}{\A\left(X,\omega\right)}\right|\le E\right)\ge 1-p_f.
    \]
    \item {\bf (Stability)} For every $\rho \in [\rho',1]$, it holds that
    \[
    \mathbb{P}_{(X,Y,\omega):X\sim_\rho Y}\Bigl( d_H\left(\A(X,\omega),\A(Y,\omega)\right)\le f+L\|X-Y\|_2^2\Bigr)\ge 1-p_{\rm st}.
    \]
    Here, the probability is taken with respect to joint randomness $\mathbb{P}_{X,Y}\otimes \mathbb{P}_\omega$ of $(X,Y,\omega)$:
    $X,Y\distr \mathcal{N}(0,I_n)$ with ${\rm Cov}(X,Y)=\rho I_n$ (which together uniquely specify the joint distribution $\mathbb{P}_{X,Y}$ denoted by $X\sim_\rho Y$); and $\omega\sim \mathbb{P}_\omega$, which is the ``coin flips" of the algorithm.
\end{itemize}
\end{definition}
In what follows, we will abuse the notation, and refer to $\A:\R^n\to\bincube$ (by suppressing $\omega$) as a \emph{randomized algorithm}.

We next comment on the performance parameters appearing in Definition~\ref{def:optimal-alg}. The parameter, $E$, refers to the cost (i.e., objective value) achieved by the partition returned by $\A$. The parameter, $p_f$, controls the ``failure" probability---the probability that  algorithm fails to return a partition with cost  below $E$. 

An important feature of Definition~\ref{def:optimal-alg} is that the stability guarantee is probabilistic; and the parameters, $f,L,\rho',p_{\rm st}$, control the stability of the algorithm. Specifically, in order to talk about stability in a probabilistic setting, one has to consider two random input vector that are potentially correlated. $\rho'$ controls the region of correlation parameters that the inputs are allowed to take. 
The parameter, $p_{\rm st}$, controls the stability probability. $L$ essentially acts like a Lipschitz constant, whereas  $f$  is introduced so that when $X$ and $Y$ are ``too close", the algorithm is still allowed to make roughly ``$f$ flips". This ``extra room" of $f$ bits is necessary: in the absence of the $f$ term, the algorithm is vacuous, since any map $\A: \R^n \to\bincube$ that is Lipschitz is trivially constant. In our application, the parameter $f$ will depend on $n$, and will essentially be $\omega\left(n\log^{-O(1)}n\right)$ (see below). 

We now state our main result regarding the failure of stable algorithms for solving the \nppp. 


\begin{theorem}\label{thm:Main}
Fix any $\epsilon\in\left(0,\frac15\right)$ and $L>0$. Let $E_n:\mathbb{N}\to\R^+$ be an energy exponent satisfying
\[
\omega\left(n\cdot \log^{-\frac15+\epsilon} n\right)\le E_n \le o(n). 
\]
For any $c>0$, define 
\begin{equation}\label{eq:main-choice-of-T}
T(c)\triangleq \exp_2\left(2^{8c L \left(\frac{n}{E_n}\right)^{5+\frac{\epsilon}{4}} \log_2\left(cL\left(\frac{n}{E_n}\right)^{4+\frac{\epsilon}{4}}\right)} \right);
\end{equation} 
and set
\begin{equation}  \label{eq:main-choice-of-param}
\rho_n'(c) \triangleq 1 - \frac{1}{cL}\left(\frac{E_n}{n}\right)^{4+\frac{\epsilon}{4}},\quad p_{f,n}(c)\triangleq \frac{1}{4T(c) \left(c L \left(\frac{n}{E_n}\right)^{4+\frac{\epsilon}{4}}+1\right)},\quad\text{and}\quad p_{{\rm st},n}(c)\triangleq \frac{\left(E_n/n\right)^{8+\frac{\epsilon}{2}}}{9c^2 L^2T(c)}.
\end{equation}
Then, there exists constant $c_1,c_2>0$ and an $N^*\in\mathbb{N}$ such that the following holds. For every $n\ge N^*$, there exists no randomized algorithm, $\A:\R^n\to \bincube$ such that $\A$ is 
\[
\Bigl(2^{-E_n}, c_1\cdot n\cdot (E_n/n)^{4+\frac{\epsilon}{4}},L,\rho_n'(c_2),p_{f,n}(c_2),p_{{\rm st},n}(c_2)\Bigr)-\text{optimal}
\]
 (for the \nppp) in the sense of Definition~\ref{def:optimal-alg}.
\end{theorem}

The proof of Theorem~\ref{thm:Main} is provided in Section~\ref{sec:pf-thm-Main}.

Several remarks are in order. In what follows, one should keep in their mind that $E_n/n = \log^{-O(1)} n$. 
Note first that there is no restriction on the runtime of $\A$, provided that it is stable. The algorithms that are ruled out satisfy 
\[
d_H\left(\A(X),\A(Y)\right)\le c_1\cdot n\cdot \log^{-O(1)} n +L\|X-Y\|_2^2.
\]
Namely, while $\A$ is stable, it is still allowed to make $\Omega \left(n\log^{-O(1)}n \right)$ ``flips" even when $X$ and $Y$ are ``too close".  

Next, since $c_2$ and $L$ are constant in $n$, 
\[
\rho_n'(c_2) = 1- \frac{1}{c_2L} \left(\frac{E_n}{n}\right)^{4+\frac{\epsilon}{4}} = 1-\log^{-O(1)} n.
\] 
Namely, for our stability assumption, we restrict our attention to $\rho \in[1-\log^{-O(1)} n,1]$. 
It is worth noting that in the case when $\rho$ is constant, $\rho=O(1)$, the stability per Definition~\ref{def:optimal-alg} holds (with a sufficiently large constant $L$) irrespective of the algorithm: by the law of large numbers, $\|X-Y\|_2^2$ is $\Theta(n)$, whereas $d_H(\A(X,\omega),\A(Y,\omega))\le n$ for any $X,Y\in\R^n$ and $\omega\in\Omega$; hence for $L$ large enough (though constant), $L\|X-Y\|_2^2 > d_H(\A(X,\omega),\A(Y,\omega))$. In particular, in some sense the interesting regime is indeed when $\rho =1-o_n(1)$, as we investigate here.

Our next remark pertains to the term $T(c)$ appearing in~\eqref{eq:main-choice-of-T}. Keeping in mind that $c$ and $L$ are constants (in $n$); and $n/E_n$ is $\omega(1)$, it follows that
\[
8cL \left(\frac{n}{E_n}\right)^{5+\frac{\epsilon}{4}} \log_2\left(cL\left(\frac{n}{E_n} \right)^{4+\frac{\epsilon}{4}}\right) = \Theta\left(\left(\frac{n}{E_n}\right)^{5+\frac{\epsilon}{4}} \log_2 \left(\frac{n}{E_n}\right)\right).
\]
By assumption on $E_n$, $E_n/n=\log^{O(1)} n$. This yields the following order of growth for $T(c)$:
\[
T(c) = \exp_2\Bigl(2^{o\left(\log^{c'}n\right)}\Bigr)\quad\text{for some}\quad c'\in(0,1).
\]
In fact, any $c'>\left(\frac15-\epsilon\right)\left(5+\frac{\epsilon}{2}\right)$ works above. Since $\left(\frac15-\epsilon\right)\left(5+\frac{\epsilon}{2}\right) = 1 -49\epsilon/10 +\Theta(\epsilon^2)$, the interval for $c'$ is indeed non-vacuous as long as $\epsilon>0$. Moreover, this interval gets larger as $\epsilon\to 1/5$ (more on this below). 

An inspection of the terms $p_{f,n}(c)$ and $p_{{\rm st},n}(c)$ appearing in~\eqref{eq:main-choice-of-param} reveals that they have the same order of growth as $T(c)^{-1}$. That is,
\[
p_{f,n}(c), p_{{\rm st},n}(c)  =\exp_2\Bigl(-2^{o\left(\log^{c'}n\right)}\Bigr)\quad\text{for any}\quad c'>\left(\frac15-\epsilon\right)\left(5+\frac{\epsilon}{2}\right).
\]  

In particular, while Theorem~\ref{thm:Main} requires high probability guarantees, these guarantees need not be exponential: a sub-exponential choice suffices. Moreover, as $\epsilon\to \frac15$, the restrictions become milder. In particular, in the limit (which corresponds essentially to $E_n = \Theta(n)$); it suffices to take a (large) constant probability of success and stability (as we elaborate below).

While the lower bound on the energy exponent $E_n$ can potentially be improved slightly to $E_n = \omega\left(n\cdot \log^{-1/5}n \cdot \log \log n\right)$; it appears that $E_n = \Omega\left(n\cdot \log^{-1/5} n \right)$ is, in fact, necessary; see Section~\ref{sec:ramsey-limited} for an informal argument. For the sake of keeping our presentation simple, we do not pursue this improvement . 


An inspection of the proof of Theorem~\ref{thm:Main} reveals certain other trade-offs, which we now discuss. Theorem~\ref{thm:Main} is proven using the $m-$OGP result (with $m=\omega(1)$) established in Theorem~\ref{thm:m-ogp-superconstant-m}. The crux of this argument is that sufficiently stable algorithms cannot overcome the overlap barrier. Now, the ``forbidden" region of overlaps, $[\beta-\eta,\beta]$, shrinks as $E_n$ gets smaller. As the forbidden region ``shrinks", the algorithm to be ruled out should be ``more stable". Now, Theorem~\ref{thm:Main} rules out algorithms whose corresponding ``$f$ term" per Definition~\ref{def:optimal-alg} is of form $\log^{-\mathcal{C}} n$ for some constant $\mathcal{C}>0$. As $E_1$ gets smaller, $\mathcal{C}$ should therefore get larger.  
Moreover, while we consider the Lipschitz constant $L$ to be $O(1)$; it appears from our analysis that $L$ can be pushed all the way up to $\Theta\left(\log^{\overline{c}} n\right)$ for a sufficiently small, though positive, constant $\overline{c}>0$. The aforementioned trade-offs are mainly due to technical reasons. More specifically, the proof requires a discretization argument, namely $[0,1]$ should be discretized into $Q$ pieces so as to use the ``stability" of algorithm towards the goal of reaching a contradiction. Now, as the overlap region shrinks or $L$ increases, the discretization should be finer: $Q$ should be larger. The parameters, $p_f$ and $p_{\rm st}$, should then be tuned down, so that a certain union bound argument over $Q$ discrete steps works. See the proof for further details. 

It is worth noting that while Theorem~\ref{thm:Main} rules out algorithms that are sufficiently stable  in the sense of Definition~\ref{def:optimal-alg}, we are unable to prove that the LDM algorithm of Karmarkar and Karp~\cite{karmarkar1982differencing} is stable with appropriate parameters, even though our simulation results, reported in Section~\ref{subsection:simulations}, suggest that it is. We leave this as a very interesting, yet we believe an approachable, open problem.


\subsubsection*{On Energy Levels $2^{-\Theta(n)}$.} 
Theorem~\ref{thm:Main} addresses energy levels $2^{-E_n}$ with $E_n = o(n)$, which naturally includes the energy levels $2^{-\Theta(n)}$. It appears, however, that for energy levels $2^{-n\epsilon}$ with $\epsilon>0$; it is possible to strengthen Theorem~\ref{thm:Main} in various aspects, which we comment now. 

It appears that a straightforward modification of Theorem~\ref{thm:Main}---in particular invoking the $m-$OGP result, Theorem~\ref{thm:main-eps-energy}, with $m=O(1)$ as opposed to Theorem~\ref{thm:m-ogp-superconstant-m}---yields that $f/n$ can be taken to be constant (in $n$): the algorithm is then allowed to make $\Theta(n)$ flips even when $X$ and $Y$ are too close. Perhaps more importantly, the probability of success and the stability guarantee can also be boosted: this yields the failure of ``stable" algorithms even with a  \emph{constant probability}  of success/stability (where the constant is sufficiently close to one). 

The trade-offs discussed above apply to this case, as well. In particular, as $\epsilon\to 0$, the forbidden region shrinks. Letting $f=C_1 n$, it is the case $C_1$ should be chosen smaller as $\epsilon\to 0$. In fact, $C_1\to 0$ as $\epsilon\to 0$. 
Likewise, as $\epsilon\to 0$, one should reduce the probability $p_{\rm st}$ of failure, as well. Again, $L$ is considered to be constant. If one, instead, decides to stick to super-constant $L$, $L=\omega(1)$; then $p_f$ and $p_{\rm st}$ should be chosen $o(1)$. These trade-offs, again, are due to technical reasons and an artifact of a certain discretization argument employed for taking advantage of the stability of algorithm.

\subsection{Stability of the LDM algorithm. Simulation results}\label{subsection:simulations}
In this section we report simulation results on running the LDM on correlated pairs of $n$-dimensional gaussian vectors. Thus let 
$X,X'\distr \mathcal{N}(0,I_n)$ be independent, and let  $Y_i=\sqrt{1-\tau^2}X_i+\tau X_i', 1\le i\le n$ for a fixed
value $\tau\in [0,1]$. Then $Y\distr \mathcal{N}(0,I_n)$ as well.
We run the LDM algorithm on instances $X$ and $Y$ and denote the results by $\sigma$ and $\sigma(\tau)$ respectively. We measure
the overlap as $(1/n)\langle \sigma,\sigma(\tau)\rangle$ and report the results. The simulations were conducted for $n=50, 100$ and $500$
and reported on Figures~\ref{fig:N50},\ref{fig:N100} and \ref{fig:N500} respectively. The horizontal axis corresponds to the 
value $\rho \triangleq -\log_2(\tau)$. The logarithmic scale is motivated by scaling purposes explained below. 
Increasing $\rho$ corresponds to higher level of correlation between
$X$ and $Y$ and thus should reduce the overlap, as indeed is seen on the figures. 
For each fixed value of $n$ and $\rho$ we compute the average overlap
of $10$ runs of the experiment and this is the value reported on the figure. 
We see that increasing the continuously correlation leads to continuous increase of the average overlap, suggesting
that the stability indeed takes place. Curiously though, in addition to the observed stability, 
the empirical average of the  overlaps drops to a nearly zero level \emph{precisely}
at $\tau\approx n^{-\alpha \log n}=\exp(-\alpha \log^2 n)$, corresponding to $\rho=\alpha\log^2 n/\log 2$,
at the threshold $\alpha = \frac{1}{2\ln 2}=0.721\dots$ which is the leading constant conjectured for the performance of the LDM,
as discussed in the introduction. To check this, note that the values of $\rho$ above for $n=50, 100$ and  $500$
are $15.91, 22.05$ and $40.17$ respectively, for this choice of $\alpha$ and this is close to the values where the overlaps touch
the zero axis. At this point, we don't have a theoretical explanation for this phase transition. It is conceivable that the algorithm
produces the smallest possible discrepancy which is stable under the perturbation above. We leave it as an interesting challenge for
further investigation.

\begin{figure}
\begin{center}
\scalebox{.4}{
\includegraphics{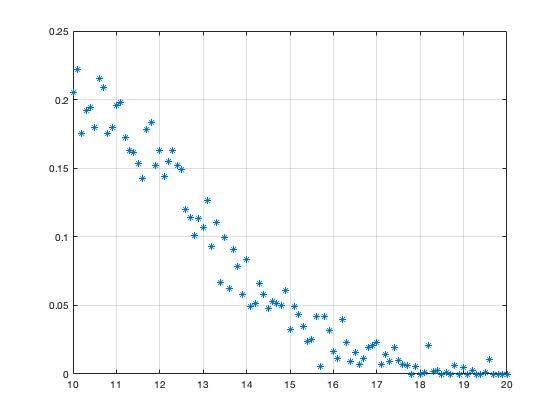}
}
\end{center}
\caption{Average overlap as a function of correlation parameter $\rho$ for $n=50$.}
\label{fig:N50}
\end{figure}

\begin{figure}
\begin{center}
\scalebox{.4}{
\includegraphics{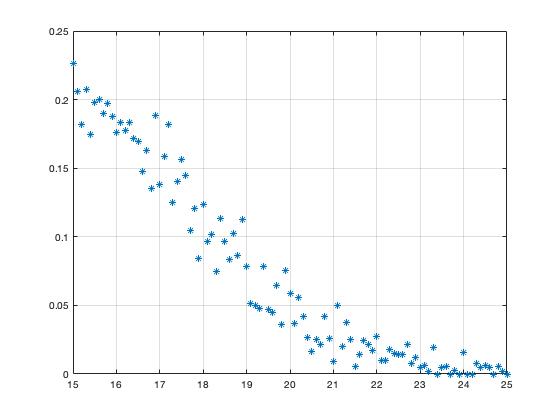}
}
\end{center}
\caption{Average overlap as a function of correlation parameter $\rho$ for $n=100$.}
\label{fig:N100}
\end{figure}

\begin{figure}
\begin{center}
\scalebox{.4}{
\includegraphics{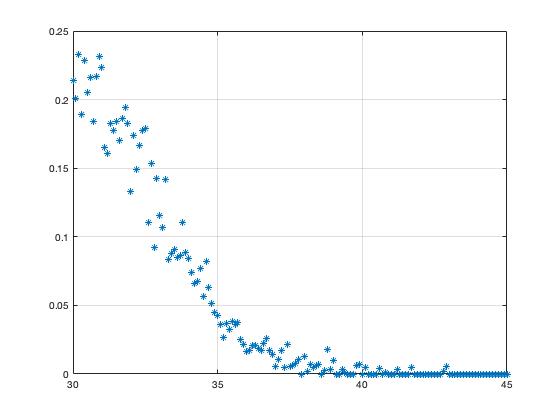}
}
\end{center}
\caption{Average overlap as a function of correlation parameter $\rho$ for $n=500$.}
\label{fig:N500}
\end{figure}

\subsection{$2-$Overlap Gap Property Implies Failure of an MCMC Family}\label{sec:mcmc-fails}
In this section, we will show that the overlap gap property for pairs of spin configurations established in Theorem~\ref{thm:2-ogp}, is a  barrier for a family of Markov Chain Monte Carlo (MCMC) methods for solving the \nppp.

Let $X\distr\mathcal{N}(0,I_n)$, denoting the ``numbers to be partitioned".
\paragraph{ The MCMC dynamics.} We begin with specifying the relevant dynamics. Let $\beta = \left(\beta_n\right)_{n\ge 1}$ with $\left(\beta_n\right)_{n\ge 1}\subset \R^+$ be a sequence of inverse temperatures. For any $\sigma\in\bincube$, define the Hamiltonian $H(\sigma)$ by
\[
H\left(\sigma\right) = \frac{1}{\sqrt{n}}\left|\ipbig{\sigma}{X}\right|.
\]
The Gibbs measure $\pi_\beta(\cdot)$ at temperature $\beta^{-1}$ defined on $\mathcal{B}_n$ is specified by the probability mass function
\begin{equation}\label{eq:gibbs-measure}
\pi_\beta(\sigma) = \frac{1}{Z_\beta}\exp\left(-\beta H(\sigma)\right) \quad\text{where}\quad Z_\beta \triangleq \sum_{\sigma\in\bincube} \exp\left(-\beta H(\sigma)\right).
\end{equation}
Here, $Z_\beta$ is the ``partition function" at inverse temperature $\beta$, which ensures proper normalization for $\pi_\beta$. It is worth noting that a minus sign is added in front of $H(\sigma)$ in order to ensure that for $\beta$ sufficiently large, that is for low enough temperatures, the Gibbs measure is concentrated on near ground-state configurations (i.e., on those $\sigma\in\bincube$ with a small $H(\sigma)$). Indeed, observe that
\[
2\exp\left(-\beta H\left(\sigma^*\right)\right) \le Z_\beta = \sum_{\sigma\in\bincube}\exp\left(-\beta H(\sigma)\right) \le 2^n \exp\left(-\beta H\left(\sigma^*\right)\right).
\]
Taking logarithms and dividing by $\beta>0$, we arrive at
\[
\frac{\ln 2}{\beta}-H(\sigma^*) \le \frac{\ln Z_\beta}{\beta} \le n\frac{\ln 2}{\beta} - H(\sigma^*).
\]
Hence, for $\beta$ sufficiently large, specifically when $\beta =\Omega\left(n 2^{n\epsilon}\right)$ (which will be our eventual choice, see below) we have
\[
\frac{\ln Z_\beta}{\beta} +H\left(\sigma^*\right)\le n\frac{\ln 2}{\beta} = O\left(2^{-n\epsilon}\right).
\]
Hence, for $\beta = \Omega\left(n 2^{n\epsilon}\right)$, it is the case that the Gibbs distribution $\pi_\beta(\cdot)$ is essentially concentrated on those $\sigma\in\bincube$ with $H(\sigma) = O\left(2^{-n\epsilon}\right)$. 

We next construct the undirected graph $\Gr$ on $2^n$ vertices with edge set $E$ on which the aforementioned MCMC dynamics is run.
\begin{itemize}
    \item Each vertex corresponds to a spin configuration $\sigma\in\bincube$.
    \item For $\sigma,\sigma'\in\bincube$, $(\sigma,\sigma')\in E$ iff $d_H\left(\sigma,\sigma'\right)=1$. 
\end{itemize}
Let $X_0\in\bincube$ be a spin configuration at which we initialize the MCMC dynamics. Let $\left(X_t\right)_{t\ge 0}$ be any nearest neighbor
discrete time Markov chain on $\Gr$ initialized at $X_0$ and reversible with respect to the stationary distribution $\pi_\beta$. 
For example, $X_t$ is discretized version of the Markov process with rates from $\sigma$ to $\sigma'$ defined by 
$\exp(\beta(H(\sigma')-H(\sigma)))$ when $\sigma'$ is a neighbor of $\sigma$ and is zero otherwise.
Then the transition matrix $Q(\cdot,\cdot)$ for $\left(X_t\right)_{t\ge 0}$ satisfies the detailed balance equations for $\pi_\beta$:
$\pi_\beta(\sigma)Q(\sigma,\sigma')=\pi_\beta(\sigma')Q(\sigma',\sigma)$ for every pair $\sigma,\sigma'$ with 
$d_H\left(\sigma,\sigma'\right)=1$.

\paragraph{ Free energy wells.} We now establish that the overlap gap property (shown in Theorem~\ref{thm:2-ogp}) induces a property called a \emph{free energy well} (FEW) in the landscape of the \nppp. This is a provable barrier for the MCMC methods, and has been employed to show slow mixing in other settings, see~\cite{arous2020algorithmic,gjs2019overlap,gamarnik2019landscape}. 

Let  $\epsilon\in\left(\frac12,1\right)$ and $\rho\in(0,1)$ be the parameter dictated by Theorem~\ref{thm:2-ogp}. We define the following sets.
\begin{itemize}
    \item $I_1 = \left\{\sigma\in\bincube : -\rho \le \frac1n\ip{\sigma}{\sigma^*}\le \rho \right\}$.
    \item $I_2\triangleq \left\{\sigma\in\bincube : \rho\le \frac1n \ip{\sigma}{\sigma^*}\le \frac{n-2}{n}\right\}$, and $\overline{I_2} = \{-\sigma:\sigma\in I_2\}$. 
    \item $I_3 = \{\sigma^*\}$ and $\overline{I_3}  = \{-\sigma^*\}$.
\end{itemize}

We now establish the FEW property.
\begin{theorem}\label{thm:FEW}
Let $\epsilon \in\left(\frac12,1\right)$ be arbitrary; and $\beta = \Omega\left(n2^{n\epsilon}\right)$. Then
\[
\min\left\{\pi_\beta\left(I_1\right),\pi_\beta\left(I_3\right)\right\}\ge \exp\left(\Omega\left(\beta 2^{-n\epsilon}\right)\right)\pi_\beta \left(I_2\right)
\]
with high probability (with respect to $X$), as $n\to\infty$. 
\end{theorem}
The proof of Theorem~\ref{thm:FEW} is provided in Section~\ref{sec:pf-thm-FEW}.

Namely, the FEW property simply states that the set $I_2$ (of spins $\sigma\in\bincube$ having a ``medium" overlap with $\sigma^*$) is a ``well" of exponentially small (Gibbs) mass separating $I_3$ and $I_1\cup \overline{I_2}\cup \overline{I_3}$. 

\paragraph{ Failure of MCMC.} We now establish, as a consequence  of the  FEW property, Theorem~\ref{thm:FEW}, that the very natural MCMC dynamics introduced earlier provably fails for solving the \npp for ``low enough temperatures", specifically when the temperature is exponentially small. This is a slow mixing result. More concretely, we establish that under an appropriate initialization, it requires an exponential amount of time for the aforementioned MCMC dynamics to ``hit" a region of ``non-trivial Gibbs mass".

To set the stage, let 
\[
\partial S \triangleq \left\{\sigma\in\bincube:d_H\left(\sigma,\sigma^*\right)=1\right\}.
\]
Clearly for any $\sigma\in\partial S$, $\Overlap(\sigma,\sigma^*)= \frac{n-2}{n}$. Thus $\partial S\subset I_2$. Now, let us initialize the MCMC via $X_0\distr \pi_\beta\left(\cdot|I_3\cup \partial S\right)$. Define also the ``escape time"
\begin{equation}\label{eq:escape-time}
\tau_\beta \triangleq \inf\left\{t\in\mathbb{N}:X_t\notin I_3\cup \partial S\mid X_0\sim \pi_\beta\left(\cdot\mid I_3\cup \partial S\right)\right\}.
\end{equation}
We  now establish the following ``slow mixing" result.
\begin{theorem}\label{thm:slow-mixing}
Let $\epsilon\in\left(\frac12,1\right)$, and $\beta=\Omega\left(n2^{n\epsilon}\right)$. Then, the following holds.
\begin{itemize} 
    \item[(a)] $I_1$ and $\overline{I_3}$ collectively contain at least a constant proportion of the Gibbs mass: 
    \[
    \pi_\beta\left(I_1\cup \overline{I_3}\right)\ge \frac12 \left(1+o_n(1)\right),
    \]
    with high probability as $n\to\infty$.
    \item[(b)] With high probability (over $X\distr \mathcal{N}(0,I_n)$) as $n\to\infty$
    \[
    \tau_\beta =\exp\left(\Omega\left(\beta 2^{-n\epsilon}\right)\right).
    \]
    In particular, for $\beta = \omega\left(n2^{n\epsilon}\right)$, we obtain $\tau_\beta = \exp\left(\Omega(n)\right)$ w.h.p. as $n\to\infty$.
\end{itemize}
\end{theorem}
The proof of Theorem~\ref{thm:slow-mixing} is provided in Section~\ref{sec:pf-slow-mixing}.

Per Theorem~\ref{thm:slow-mixing}, it takes an exponential amount time for $(X_t)_{t\ge 0}$ initialized by $\pi_\beta\left(\cdot|I_3\cup \partial S\right)$ to escape $\sigma^*$, and enter a region of nearly half of Gibbs mass; implying slow mixing. 

It is worth noting that Theorem~\ref{thm:slow-mixing} is shown when the temperature $\beta^{-1}$ is low enough, more specifically is exponentially small. This ensures the Gibbs measure is well-concentrated on ground states. We leave the analysis of the MCMC dynamics in the high-temperature regime (i.e., lower values of $\beta$) as an interesting open problem for future work. 
\section{Certain Natural Limitations of Our Techniques}
Given that our methods fall short of addressing the \compgap of the \npp all the way down to $2^{-\Theta(\log^2 n)}$; it is natural to inquire into their limitations. 
\subsection{Limitation of the $m-$Overlap Gap Property for Super-constant $m$}\label{sec:m-ogp-beyond-is-impossible}
In this section, we give an informal argument suggesting the absence of the $m-$OGP when the energy level $E_n$ is $O\left(\sqrt{n\log_2 n}\right)$.
Our informal argument will reveal the following. It appears not possible to establish $m-$OGP (for super-constant $m$) as we do in Theorem~\ref{thm:m-ogp-superconstant-m},  for energy levels above $\exp_2\left(-\omega\left(\sqrt{n\log_2 n}\right)\right)$. We now detail this.
\subsection*{Step 1: $E_n=\omega(\sqrt{n})$ is necessary.} We first note, upon studying the proof of Theorem~\ref{thm:m-ogp-superconstant-m} more carefully, that for the first moment argument to work, one should take $\beta=1-o_n(1)$. 
For convenience, let $\beta=1-2\nu_n$, where $\nu_n = o_n(1)$ is a sequence of positive reals. Furthermore, to ensure the invertibility of a certain covariance matrix arising in the analysis, one should also take $\eta \lesssim \nu_n/m$ (see the proof for further details on this matter). 

Now, for the OGP to be meaningful, it should be the case that $n\eta = \Omega(1)$, as noted already previously. Indeed, otherwise the overlap region is void, since no admissible overlap values $\rho$ can be found within the interval $[\beta-\eta,\beta]$. Now, since $\eta \lesssim \nu_n/m$, 
\[
n\eta = \Omega(1) \implies \frac{n\nu_n}{m} = \Omega(1)\implies n\nu_n = \Omega(m).
\]
Next, for an $m-$tuple $\left(\sigma^{(i)}:1\le i\le m\right)$; the energy value, $2^{-E_n}$, contributes to a $-mE_n$ in the exponent (we again refer the reader to the proof for further details). Finally, a very crude cardinality bound on the number of $m-$tuples $\left(\sigma^{(i)}:1\le i\le m\right)$ with pairwise inner products $\frac1n\ip{\sigma^{(i)}}{\sigma^{(j)}} \in[\beta-\eta,\beta]$, $1\le i<j\le m$, is the following: using the na\"ive approximation $\log_2\binom{n}{k}=(1+o_n(1))k\log_2 \frac{n}{k}$ valid for $k=o(n)$, we arrive at
\[
2^n \binom{n}{n\frac{1-\beta}{2}}^{m-1}\sim \exp_2\left(n+mn\nu_n\log \frac{1}{\nu_n}\right),
\]
where we have used $m=\omega_n(1)$ and $\beta=1-2\nu_n$; while ignoring the lower order terms for convenience. Blending these observations together, we arrive at the following formula, for the exponent of the first moment:
\[
\xi(n)\triangleq n+mn \nu_n \log\frac{1}{\nu_n} -mE_n.
\]
Now, for the first moment argument to work, it should be the case $-\xi(n) = \omega_n(1)$. Hence, $mE_n=\Omega(n)$ must hold. Since $n\nu_n = \Omega(m)$ as shown above, this yields 
\[
n\nu_n = \Omega(m)=\Omega\left(\frac{n}{E_n}\right).
\]
Now, a final constraint is
\[
mE_n = \Omega\left(mn\nu_n\log\frac{1}{\nu_n}\right)\Leftrightarrow E_n = \Omega\left(n\nu_n\log \frac{1}{\nu_n}\right).
\]
But since $\nu_n=o_n(1)$, we  have $\log\frac{1}{\nu_n} = \omega_n(1)$, and consequently, we must have, at the very least,
\[
E_n = \omega\left(\frac{n}{E_n}\right)\Leftrightarrow E_n = \omega(\sqrt{n}). 
\]
\subsection*{Step 2: from $E_n=\omega\left(\sqrt{n}\right)$ to $E_n=\omega\left(\sqrt{n\log_2 n}\right)$.} We now let $E_n=\phi(n)\sqrt{n}$, where $\phi(n)=\omega_n(1)$, and plug this in above to study the parameters numerically. Inspecting the lines above, one should take $ m=C\frac{\sqrt{n}}{\phi(n)}$, where $C>1$ is some constant. This, in turn, yields that we require $ \nu_n = \frac{g(n)}{\phi(n)\sqrt{n}}$, where $g(n) = \Omega(1)$. In particular, observe that
\[
\log \frac{1}{\nu_n} = \frac12\log_2 n(1+o_n(1))=\Theta(\log_2 n).
\]

Now, a final constraint, as one might recall from above, is that the exponent, $-\xi(n)$, should be $\omega_n(1)$ as $n\to\infty$. With this, it should hold
\[
n\nu_n \log \frac{1}{\nu_n} = O(E_n) = O(\phi(n)\sqrt{n}).
\]
Since
\[
n\nu_n \log \frac{1}{\nu_n} = \Theta\left(\frac{g(n)\sqrt{n}\log_2 n}{\phi(n)}\right),
\]
it should be the case
\[
\frac{g(n)\sqrt{n}\log_2 n}{\phi(n)}\lesssim \phi(n)\sqrt{n},
\]
which implies $\phi(n)=\Omega(\sqrt{\log_2 n})$. 

Namely, this argument demonstrates the following: if one wants to establish the overlap gap property for an energy exponent $E_n$ through a first moment technique, $E_n$ should have a growth of at least $\sqrt{n\log_2 n}$; otherwise the moment argument fails. 
\subsection{Limitation of the Ramsey Argument}\label{sec:ramsey-limited}
An important question that remains is whether one can leverage further the $m-$OGP result (Theorem~\ref{thm:m-ogp-superconstant-m}) to establish an analogue of our hardness result (Theorem~\ref{thm:Main}) for energy levels with an exponent $E_n$ that is at least slightly below $\omega\left(n \log^{-1/5+\epsilon} n\right)$ or all the way to $\omega\left(\sqrt{n\log_2 n}\right)$. We now argue that using our line of argument based on the Ramsey Theory,  $E_n = \Omega\left(n \log^{-\frac15}n \right)$, no beyond, is essentially the best exponent one would hope to address. 


Let $E_n$ be a target exponent for which one wants to establish the hardness; and $m$ be the OGP parameter required per Theorem~\ref{thm:m-ogp-superconstant-m}. Our proof uses, in a crucial way, certain properties regarding Ramsey numbers arising in extremal combinatorics. To that end, let $R_Q(m)$ denotes the smallest $n\in\mathbb{N}$ such that any $Q$ (edge) coloring of $K_n$ contains a monochromatic $K_m$ (see Theorem~\ref{thm:ramsey} for more details). Our argument then contains the following ingredients. We generate a certain number $T$ of ``instances" (of the \texttt{NPP}) such that $T\ge R_2(M)$ for $M\ge R_Q(m)$ where $Q$ corresponds to a discretization level we need to address $E_n$. When then essentially (a) construct a graph $\Gr$ on $T$ vertices satisfying certain properties, in particular $\alpha(\Gr)\le M-1$ (where $\alpha(\Gr)$ is the cardinality of any largest independent set of $\Gr$) (b) extract a clique $K_M$ of $\Gr$ whose edges are colored with one of $Q$ available colors; and (c) use $M\ge R_Q(m)$ to conclude that the original graph, $\Gr$, contains a monochromatic $K_m$. From here, we then argue that this yields a forbidden configuration, a contradiction with the $m-$OGP. 

Using well-known upper and lower bounds on Ramsey numbers (see e.g.~\cite{conlon2020lower}) one should then choose $T\ge \exp_2\left(\Theta(M)\right)$. Moreover, the best lower bound on $R_Q(m)$, due to Lefmann~\cite{lefmann1987note}, asserts that $R_Q(m)\ge \exp_2\left(mQ/4 \right)$. Combining these bounds, we then conclude that $T$ should be of order at least 
\begin{equation}\label{eq:T-should-be-order-atleast}
T\ge \exp_2\Bigl(2^{\Theta(mQ)}\Bigr).
\end{equation} 
Now, an inspection of our proof of Theorem~\ref{thm:Main} yields that for certain union bounds, e.g. \eqref{eq:event-chaos-all-M}, to work; $T$ should be sub-exponential: $T=2^{o(n)}$. Combining this with \eqref{eq:T-should-be-order-atleast}; a necessary condition turns out to be
\begin{equation}\label{eq:a-necessary-cond-on-mQ}
mQ = O\Bigl(\log_2 n\Bigr).
\end{equation} 
Now, the discretization $Q$ should be sufficiently fine to ensure that the overlaps are eventually ``trapped" within the (forbidden) overlap region of length $\eta$ dictated by Theorem~\ref{thm:m-ogp-superconstant-m}. In particular, tracing our proof, it appears from~\eqref{events-steps-stable} that
\begin{equation}\label{eq:a-bound-on-Q-eta-sq}
Q = \Omega\left(\frac{1}{\eta^2}\right)
\end{equation} 
should hold. Furthermore, from the discussion on $m-$OGP; as well as the proof of Theorem~\ref{thm:m-ogp-superconstant-m}, it appears also that $mE_n$ should be $\Omega(n)$, that is
\begin{equation}\label{eq:a-bound-on-m-of-form-omega-n-e-n}
m = \Omega\left(\frac{n}{E_n}\right). 
\end{equation}
Now, we take the overlap value $\eta$ to be $g(n)/n$, where $g(n) = \omega(1)$; but it also satisfies other certain, natural, constraints. In particular, using \eqref{eq:nu-n-g-over-f} and \eqref{eq:beta:1-2nu-n}; for the parameters to make sense, $g(n)$ should be $o(E_n)$. Let
\begin{equation}\label{eq:g-n-E-n-eta-n}
E_n = ns(n), \quad g(n) = ns(n)z(n), \quad\text{and}\quad \eta = s(n)z(n) \quad\text{where}\quad s(n),z(n) = o_n(1). 
\end{equation}
Combining \eqref{eq:a-bound-on-Q-eta-sq}, \eqref{eq:a-bound-on-m-of-form-omega-n-e-n} and \eqref{eq:g-n-E-n-eta-n}; we therefore have
\begin{equation}\label{eq:we-must-therefore-have}
mQ = \Omega\left(\frac{n}{E_n \eta^2}\right) = \Omega\left(\frac{1}{s(n)^3z(n)^2}\right).
\end{equation}
Furthermore, to ensure Theorem~\ref{thm:m-ogp-superconstant-m} applies; the  ``exponent" of the first moment should not ``blow up". For this reason, using~\eqref{eq:this-will-later-be-super-useful}, \eqref{eq:a-very-loong-expression-for-the-firdt-momnet-m-ogp}, as well as the counting term~\eqref{eq:n-choose-n-1-beta-eta}, it must at least hold that
\[
mE_n = \Omega\left(\frac{mng(n)}{E_n} \right) \iff  \frac{E_n}{n} = s(n) = \Omega\left(z(n)\right).
\]
This, together with \eqref{eq:we-must-therefore-have} as well as the upper bound \eqref{eq:a-necessary-cond-on-mQ}, implies that
\[
s(n) = \Omega\left(\log^{-\frac15} n\right).
\]
Hence, 
\[
E_n = ns(n) = \Omega\left(n \log^{-\frac15} n\right),
\]
is essentially indeed the best possible. We gave ourselves an $\epsilon$ ``extra room" in Theorem~\ref{thm:Main} so as to avoid complicating relevant quantities any further. 

A very interesting question is whether one can by-pass the Ramsey argument altogether. This would help establishing the failure of (presumably more) stable algorithms for even higher energy levels, $E_n = \omega\left(\sqrt{n\log_2 n}\right)$, a regime where Theorem~\ref{thm:m-ogp-superconstant-m} is applicable.

\section{Open Problems and Future Work}\label{sec:conclusion}

Our work suggests interesting avenues for future research. While we have focused on the \npp in the present paper for simplicity, we believe that many of our results extend to the multi-dimensional case, \vbp \eqref{eq:vbp}, as well; perhaps at the cost of more detailed and computation-heavy proofs. This was noted already in Remark~\ref{remark:m-ogp-applies-to-vbp-aswell}.

Yet another very important direction pertains the \compgap of the \nppp. The $m-$OGP results that we established hold for energy levels $2^{-\Theta(n)}$ when $m=O(1)$; and for $2^{-E_n}$, $\omega\left(\sqrt{n\log_2 n}\right)\le E_n \le o(n)$, when $m=\omega_n(1)$. While we are able to partially explain the aforementioned \compgap to some extent, we are unable close it all the way down to the current computational threshold: the best known polynomial-time algorithm to this date achieves an exponent of only $\Theta\left(\log^2 n\right)$. A very interesting open question is whether this gap can be ``closed" altogether. That is, either devise a better (efficient) algorithm, improving upon the algorithm by Karmarkar and Karp~\cite{karmarkar1982differencing}; or establish the hardness by taking one of the alternative routes (mentioned in the introduction) tailored for proving \emph{average-case hardness}. In light of the fact that not much work has been done in the algorithmic front since the paper~\cite{karmarkar1982differencing}, it is plausible to hope that better efficient algorithms can indeed be found. In particular, a potential direction appears to be setting up an appropriate Markov Chain dynamics, and establishing rapid mixing. We leave this as an open problem for future work.

While we are able to rule out stable algorithms in the sense  of Definition~\ref{def:optimal-alg}, we are unable to prove that the algorithm by Karmarkar and Karp, in particular the LDM algorithm introduced earlier, is stable with appropriate parameters, although our simulation results suggest that it is. We leave this as yet another open problem. 


Another direction pertains the parameters of algorithms that we consider. In particular, one potential direction is to establish Theorem~\ref{thm:Main}
when the algorithm say has $o_n(1)$ probability of success. That is, $p_f$ in Definition~\ref{def:optimal-alg} is $1-o_n(1)$. 
We conjecture that the value $p_f = 1-n^{-O(1)}$ is within the reach.

\section{Proofs}\label{sec:proofs}
\subsection{Auxiliary Results}
Below, we record several auxiliary results that will guide our proofs. The first result is the standard asymptotic approximation for the factorial.
\begin{equation}\label{thm:stirling}
\log_2 n! =n\log_2 n -n\log_2 e+O\left(\log_2 n\right).
\end{equation}

The second is a very standard approximation for the binomial coefficients, whose proof we include herein for completeness.
\begin{lemma}\label{lemma:binomial-k-o(n)}
Let $n,k\in\mathbb{N}$, where $k=o(n)$. Then,
\[
\log_2 \binom{n}{k} = (1+o_n(1))k\log_2\frac{n}{k}.
\]
\end{lemma}
\begin{proof}
Note that for any $0\le i\le k-1$,
$
\frac{n-i}{k-i}\ge \frac{n}{k}$. 
Hence,
\[
\left(\frac{n}{k}\right)^k \le \prod_{0\le i\le k-1}\frac{n-i}{k-i}=\binom{n}{k}.
\]
Next,
\[
\binom{n}{k}\left(\frac{k}{n}\right)^k \le \sum_{0\le t\le n}\binom{n}{t}\left(\frac{k}{n}\right)^t =\left(1+\frac{k}{n}\right)^n.
\]
Since $\ln(1+x)\le x$, setting $x=k/n$ yields
\[
\left(1+\frac{k}{n}\right)^n \le e^k.
\]
Combining these, we obtain
\[
\left(\frac{n}{k}\right)^k \le \binom{n}{k}\le \left(\frac{en}{k}\right)^k. 
\]
Taking now the logarithms both sides, and keeping in mind that  $\log_2\frac{n}{k}=\omega_n(1)$; we arrive at
\[
k\log_2 \frac{n}{k}\le \log_2 \binom{n}{k} \le k\left(\log_2 \frac{n}{k}+\log_2 e\right)=k\log_2\frac{n}{k}(1+o_n(1)).
\]
Hence,
\[
\log_2 \binom{n}{k} = (1+o_n(1))k\log_2\frac{n}{k}
\]
as claimed.
\end{proof}

The third auxiliary result is a theorem from the matrix theory.
\begin{theorem}{{\bf (Wielandt-Hoffman)}}\label{thm:wielandt-hoffman}

Let $A,A+E\in\R^{n\times n}$ be two symmetric matrices with respective eigenvalues $\lambda_1(A)\ge \lambda_2(A)\ge\cdots\ge \lambda_n(A)$ and $\lambda_1(A+E)\ge \lambda_2(A+E)\ge \cdots\ge\lambda_n(A+E)$. Then,
\[
\sum_{1\le i\le n}\left(\lambda_i(A+E)-\lambda_i(A)\right)^2 \le \|E\|_F^2.
\]
\end{theorem}
For a reference, see e.g. \cite[Corollary~6.3.8]{horn2012matrix}; and see \cite{hoffman1953variation} for the original paper by   and Hoffman. 

\subsection{Proof of Theorem~\ref{thm:2-ogp}}\label{sec:pf-thm:2-ogp}
\begin{proof}
Let $\epsilon\in\left(\frac12,1\right]$. Let $\rho\in(0,1)$ to be tuned appropriately, and
\[
\mathcal{Z}(\rho) \triangleq \left\{\left(\sigma,\sigma'\right)\in \bincube\times \bincube:\Overlap\left(\sigma,\sigma'\right)\in\left[\rho,\frac{n-2}{n}\right]\right\}.
\]
Set
\begin{equation}\label{eq:2-ogp-rv-N}
N \triangleq \sum_{\left(\sigma,\sigma'\right)\in\mathcal{Z}(\rho)}\ind\left\{\frac{1}{\sqrt{n}}\left|\ip{\sigma}{X}\right|,\frac{1}{\sqrt{n}}\left|\ip{\sigma'}{X}\right| = O\left(2^{-n\epsilon}\right)\right\}
\end{equation}
We will establish that $\mathbb{E}[N]=\exp\left(-\Theta(n)\right)$. This, together with Markov's inequality, will then yield the desired conclusion: 
\[
\mathbb{P}\left(N\ge 1\right)\le\mathbb{E}[N]=\exp\left(-\Theta(n)\right).
\]

\paragraph{ Step I. Counting.} We first upper bound the cardinality $\left|\mathcal{Z}\left(\rho\right)\right|$. Note that there are $2^n$ choices for $\sigma\in\bincube$. Having chosen $\sigma$; $\sigma'$ can now be chosen in
\[
\sum_{1\le k\le \lceil n\frac{1-\rho}{2}\rceil}\binom{n}{k}
\]
different ways. This is due to the fact that if $k=d_H\left(\sigma,\sigma'\right)$ then $\Overlap\left(\sigma,\sigma'\right)=\left|1-2\frac{k}{n}\right|$. Using Stirling's approximation \eqref{thm:stirling}, and the fact the sum contains $n^{O(1)}=\exp_2\left(O\left(\log_2 n\right)\right)$ terms, we arrive at the upper bound
\begin{equation}\label{eq:card-bound-for-2-ogp-result}
\left|\mathcal{Z}\left(\rho\right)\right|\le \exp_2\left(n+nh\left(\frac{1-\rho}{2}\right)+O\left(\log_2 n\right)\right).
\end{equation}
Here $h(x) = -x\log_2 x-(1-x)\log_2(1-x)$ is the binomial entropy function logarithm base two. 

\paragraph{ Step II. Upper bound on probability.} Let $\left(\sigma,\sigma'\right)\in \mathcal{Z}(\rho)$ with $\Overlap\left(\sigma,\sigma'\right)=\bar{\rho}$. Set
\[
Y_\sigma \triangleq \frac{1}{\sqrt{n}}\ip{\sigma}{X} \quad\text{and}\quad Y_{\sigma'}\triangleq  \frac{1}{\sqrt{n}}\ip{\sigma'}{X}.
\]
Note that $Y_\sigma,Y_{\sigma'}\distr\mathcal{N}(0,1)$ with  correlation $\bar{\rho}$. Now, let $C>0$ be a constant; and denote by $\mathcal{R}_C$ the region
\[
\mathcal{R}_C\triangleq \left[-C2^{-n\epsilon},C2^{-n\epsilon}\right] \times \left[-C2^{-n\epsilon},C2^{-n\epsilon}\right].
\]
Denote also
\[
f(x,y)\triangleq \exp\left(-\frac{1}{2\left(1-\bar{\rho}^2\right)}\left(x^2-2\bar{\rho} xy+y^2\right)\right)
\]
As long as $\bar{\rho}\in(0,1)$, we have $f(x,y)\le 1$ for every $x,y$. Furthermore,
\begin{equation}\label{eq:2-ogp-cov-asymp}
    \sqrt{1-\left(1-\frac2n\right)^2} = \sqrt{\frac4n - \frac{4}{n^2}} = \frac{2}{\sqrt{n}}\left(1+o_n(1)\right).
\end{equation}
We then have,
\begin{align}
    \mathbb{P}\left(\left(Y_\sigma,Y_{\sigma'}\right)\in\mathcal{R}_C\right) &= \frac{1}{2\pi \sqrt{1-\bar{\rho}^2}}\int_{(x,y)\in \mathcal{R}_C}f\left(x,y\right)\; dx \; dy\\
    &\le \mathcal{C}\frac{1}{ \sqrt{1-\left(1-\frac2n\right)^2}}2^{-2n\epsilon} \\
    &=\mathcal{C}'\left(1+o_n(1)\right)2^{-2n\epsilon}\sqrt{n},\label{eq:2-ogp-prob-bd}
\end{align}
where $\mathcal{C},\mathcal{C}'>0$ are some absolute constants. Note that \eqref{eq:2-ogp-prob-bd} is uniform in $\bar{\rho}$: it holds true for {\bf every} $\left(\sigma,\sigma'\right)\in\mathcal{Z}\left(\rho\right)$.

\paragraph{ Step III. Computing the expectation.} We now compute $\mathbb{E}[N]$ of $N$ introduced in \eqref{eq:2-ogp-rv-N}. Using linearity of expectation, \eqref{eq:card-bound-for-2-ogp-result}, and \eqref{eq:2-ogp-prob-bd}, we arrive at
\begin{equation}\label{eq:2-ogp-E-N-up-bd}
\mathbb{E}[N]\le \exp_2\left(n+nh\left(\frac{1-\rho}{2}\right)-2n\epsilon +O\left(\log_2 n\right)\right).
\end{equation}
Since $\epsilon>\frac12$, there exists a $\rho>0$ such that
\[
1+h\left(\frac{1-\rho}{2}\right)-2\epsilon<0.
\]
For this choice of $\rho$, we indeed have per \eqref{eq:2-ogp-E-N-up-bd} that
\[
\mathbb{E}[N]=\exp\left(-\Theta(n)\right),
\]
concluding the proof.
\end{proof}
\subsection{Proof of Theorem~\ref{thm:main-eps-energy}} \label{sec:pf-eps-m-OGP}

\begin{proof}
For any $1\le i\le m$ and $\tau\in\mathcal{I}$;
recall $Y_i(\tau)\triangleq \sqrt{1-\tau^2}X_0+\tau X_i\in\R^n$; and
\[
H(\sigma^{(i)},Y_i(\tau))\triangleq \frac{1}{\sqrt{n}}|\ip{\sigma^{(i)}}{Y_i(\tau)}|.
\]
Define,
\[
S(\beta,\eta,m)\triangleq \left\{(\sigma^{(1)},\dots,\sigma^{(m)}):\sigma^{(i)}\in\{-1,1\}^n, \Overlap(\sigma^{(i)},\sigma^{(j)})\in[\beta-\eta,\beta],1\le i<j\le m\right\},
\]
and
\begin{equation}\label{eq:expect}
N(\beta,\eta,m,\epsilon,\mathcal{I})=\sum_{(\sigma^{(1)},\dots,\sigma^{(m)})\in S(\beta,\eta,m)}\ind\left\{\exists \tau_1,\dots,\tau_m\in\mathcal{I}:H(\sigma^{(i)},Y_i(\tau_i))=O\left(2^{-n\epsilon}\right),1\le i\le m\right\}.
\end{equation}
Observe that $N(\beta,\eta,m,\epsilon,\mathcal{I})=|\mathcal{S}(\beta,\eta,m,\epsilon,\mathcal{I})|$. In what follows, we will establish that for  an appropriate choice of parameters $\beta,\eta$, and $m\in\mathbb{Z}_+$, 
\[
\mathbb{E}[N(\beta,\eta,m,\epsilon,\mathcal{I})]=\exp_2(-\Theta(n)),\]
which will then yield,
\[
\mathbb{P}(\mathcal{S}(\beta,\eta,m,\epsilon,\mathcal{I})\ne \varnothing) = \mathbb{P}(N(\beta,\eta,m,\epsilon,\mathcal{I})\ge 1)\le \exp_2(-\Theta(n))
\]
through Markov's inequality, and thus the conclusion.

\paragraph{ Step I: Counting.} We start by upper bounding $|S(\beta,\eta,m)|$. There are $2^n$ choices for $\sigma^{(1)}$. Now, for any fixed $\sigma$, we claim there exists $2\binom{n}{n\frac{1-\rho}{2}}$ sign configurations $\sigma'\in\{-1,1\}^n$ for which $\Overlap(\sigma,\sigma')=\rho$. Indeed, let $k\triangleq \sum_{1\le i\le n}\ind\{\sigma_i\neq \sigma_i'\}$, the number of coordinates $\sigma$ and $\sigma'$ disagree. With this we have $\Overlap(\sigma,\sigma')=\left|\frac{n-2k}{n}\right|$, from which we obtain $k=n\frac{1\pm \rho}{2}$. Equipped with this observation, we now compute the number of choices for $\sigma^{(2)}$ as $
2\sum_{\beta-\eta\le \rho \le \beta:\rho n\in\mathbb{Z}}\binom{n}{n\frac{1-\rho}{2}}$. We then obtain
\begin{align}
    |S(\beta,\eta,m)|&\le 2^n \left(2\sum_{\beta-\eta\le \rho \le \beta:\rho n\in\mathbb{Z}}\binom{n}{n\frac{1-\rho}{2}}\right)^{m-1} \\
    &=2^n\left(2\sum_{\beta-\eta\le \rho \le \beta:\rho n\in\mathbb{Z}}\exp_2\left(nh\left(\frac{1-\rho}{2}\right)+O(\log_2 n)\right)\right)^{m-1}\label{eq:stirling-upper} \\
    &\le 2^n\left(\exp_2\left(nh\left(\frac{1-\beta+\eta}{2}\right)+O(\log_2 n)\right)\right)^{m-1}\label{eq:involves-O(n)-term}.
\end{align}
We now justify these lines. Recall that $\log_2 n! = n \log_2 n -n\log_2 e +O(\log_2 n)$ by the Stirling's approximation, \eqref{thm:stirling}. Using this, we obtain $\rho\in(0,1)$, $\binom{n}{\rho n}=\exp_2 (nh(\rho)+O(\log_2 n))$ where we recall that $h(x)=-x\log_2 x -(1-x)\log_2(1-x)$ is the binary entropy function logarithm base $2$. Thus \eqref{eq:stirling-upper} follows. \eqref{eq:involves-O(n)-term} is a consequence of the fact that the sum involves $O(n)$ terms. We conclude 
\begin{equation}\label{eq:cardinality-bd}
    |S(\beta,\eta,m)|\le \exp_2\left(n+n(m-1)h\left(\frac{1-\beta+\eta}{2}\right)+(m-1)O(\log_2 n)\right).
\end{equation}

\paragraph{ Step II: Probability calculation.} Fix $\tau_1,\dots,\tau_m\in\mathcal{I}$.  For any fixed $(\sigma^{(1)},\dots,\sigma^{(m)})\in S(\beta,\eta,m)$, we now investigate
\[
\mathbb{P}\left(H(\sigma^{(i)},Y_i(\tau_i))=O(2^{-n\epsilon}),1\le i\le m\right).
\]
To that end, let $Z_i = \frac{1}{\sqrt{n}}\ip{\sigma^{(i)}}{Y_i(\tau_i)}$, and let $\rho_{ij}=\frac1n\ip{\sigma^{(i)}}{\sigma^{(j)}}$. Note that for each $1\le i\le m$, $Z_i$ is standard normal, and moreover, the vector $(Z_1,\dots,Z_m)$ is a multivariate Gaussian with mean zero and some covariance matrix $\Sigma$. 

We now investigate this covariance matrix. To that end, let $\gamma_i\triangleq \sqrt{1-\tau_i^2}$, $1\le i\le m$. We first compute $\mathbb{E}[Y_i(\tau_i) Y_j(\tau_j)^T]\in\R^{n\times n}$. We have
\[
Y_i(\tau_i) Y_j(\tau_j)^T =\gamma_i\gamma_j X_0X_0^T + \gamma_i\tau_j X_0X_j^T +\gamma_j\tau_i X_iX_0^T + \tau_i\tau_j X_iX_j^T.
\]
Since $X_0,X_i,X_j$ are i.i.d., we thus obtain
\begin{equation}\label{eq:cov-auxiL2}
\mathbb{E}[Y_i(\tau_i) Y_j(\tau_j)^T]=\gamma_i\gamma_j I_n\in\R^{n\times n}.
\end{equation}
Equipped with this, we now have for any $1\le i<j\le m$,
\begin{align*}
    {\rm Cov}(Z_i,Z_j) &=\mathbb{E}\left[\frac{1}{\sqrt{n}}\ip{\sigma^{(i)}}{Y_i(\tau_i)}\frac{1}{\sqrt{n}}\ip{\sigma^{(j)}}{Y_j(\tau_j)}\right]\\
    &=\frac1n (\sigma^{(i)})^T\mathbb{E}[Y_i(\tau_i)Y_j(\tau_j)^T]\sigma^{(j)}\\
    &=\rho_{ij}\gamma_i\gamma_j.
\end{align*}
Namely, the covariance matrix $\Sigma\in\R^{m\times m}$ of $(Z_1,\dots,Z_m)$ is given by $\Sigma_{ii}=1$ for $1\le i\le m$, and $\Sigma_{ij}=\Sigma_{ji}=\rho_{ij}\gamma_i\gamma_j$ for $1\le i<j\le m$. 

Now, fix arbitrary constants $C_1,\dots,C_m>0$; and let $V\subset \R^m$ be the region defined by
\[
V = \left(-C_1 2^{-n\epsilon},C_1 2^{-n\epsilon}\right)\times \left(-C_2 2^{-n\epsilon},C_2 2^{-n\epsilon}\right)\times \cdots\times \left(-C_m 2^{-n\epsilon},C_m 2^{-n\epsilon}\right).
\]
Provided $\Sigma$ is invertible, which we verify independently, the probability of interest evaluates to
\[
\mathbb{P}((Z_1,\dots,Z_m)\in V) = \frac{1}{(2\pi)^{m/2} |\Sigma|^{1/2}}\int_V \exp\left(-\frac{x^T \Sigma^{-1}x}{2}\right)\; dx.
\]
As $
\exp\left(-\frac{x^T\Sigma^{-1}x}{2}\right)\le 1$,
we can crudely upper bound this by
\[
\frac{1}{(2\pi)^{m/2} |\Sigma|^{1/2}}{\rm Vol}(V)=\frac{2^{m/2}\prod_{1\le j\le m}C_j}{\pi^{m/2}} |\Sigma|^{-1/2} 2^{-n\epsilon m}.
\]
Observe now that $2^{m/2},\pi^{m/2}$, and $\prod_{1\le j\le m}C_j$ are all constant order $O(1)$ with respect to $n$. Suppose now that $\Sigma$ is such that the determinant of $\Sigma$ is bounded away from zero by an explicit constant controlled solely by $m,\beta,\eta$, regardless of $\mathcal{I}$ and regardless of $\tau_1,\dots,\tau_m\in\mathcal{I}$. If this is the case, then $|\Sigma|^{-1}$ is $O(1)$ with respect to $n$. 
This yields
\begin{equation}\label{eq:prob1}
\mathbb{P}\left((Z_1,\dots,Z_m)\in V\right)\le \exp_2\left(-n\epsilon m +O(1)\right).
\end{equation}
We now take now a union bound over all $\tau_1,\dots,\tau_m\in \mathcal{I}$ (note that there are at most $2^{mo(n)}=2^{o(n)}$ such terms), and arrive at
\begin{equation}\label{eq:prob}
\mathbb{P}\left(\exists \tau_1,\dots,\tau_m\in\mathcal{I}:H(\sigma^{(i)},Y_i(\tau_i))=O\left(2^{-n\epsilon}\right),1\le i\le m\right)=\exp_2(-n\epsilon m+o(n)).
\end{equation}
\paragraph{ Step III: Calculating the expectation  $\mathbb{E}[N(\beta,\eta,m,\epsilon)]$.} Provided
{\bf $\Sigma$ is invertible}, we can compute the expectation \eqref{eq:expect} by using \eqref{eq:cardinality-bd} and \eqref{eq:prob}:
\[
\mathbb{E}[N(\beta,\eta,m,\epsilon)] \le \exp_2\left(n+n(m-1)h\left(\frac{1-\beta+\eta}{2}\right)+o(n) -n\epsilon m\right).
\]
Hence, provided the parameters $\beta,\eta,m$ are chosen so that
\begin{equation}\label{eq:exponent-condition}
    1+(m-1)h\left(\frac{1-\beta+\eta}{2}\right)-\epsilon m<0,
\end{equation}
and $|\Sigma|$ is bounded away zero by an explicit constant independent of $\mathcal{I}$, and the choices $\tau_1,\dots,\tau_m$, we indeed obtain $\mathbb{E}[N(\beta,\eta,m,\epsilon)]=\exp_2(-\Theta(n))$, as desired. 

We choose $m>\frac2\epsilon$. With this, $1-\frac{\epsilon m}{2}<0$. Observe now that if $0<\eta<\beta<1$ are chosen so that $h\left(\frac{1-\beta+\eta}{2}\right) <\frac\epsilon 2$, the condition \eqref{eq:exponent-condition} is indeed satisfied. With this, it suffices for $0<\eta<\beta<1$ to satisfy
\begin{equation}\label{eq:cond1-on-beta-eta1}
    \beta-\eta>1-2h^{-1}(\epsilon/2),
\end{equation}
where $h^{-1}:[0,1]\to[0,1/2]$ is the inverse of the binary entropy function. 

\paragraph{ Step IV: Invertibility of $\Sigma$.} We next study the invertibility of covariance matrix $\Sigma$, which we recall $\Sigma_{ii}=1$ for $1\le i\le m$; and $\Sigma_{ij}=\Sigma_{ji}=\gamma_i\gamma_j\rho_{ij}$, $1\le i<j\le m$, for some $\tau_1,\dots,\tau_j\in\mathcal{I}$.

Let us now define an auxiliary matrix $\bar{\Sigma}\in\R^{m\times m}$ by $\bar{\Sigma}_{ii}=1$ for $1\le i\le m$ and $\bar{\Sigma}_{ij}=\bar{\Sigma}_{ji}=\rho_{ij}$ for any $1\le i<j\le m$. Namely, $\bar{\Sigma}$ is the covariance matrix $\Sigma$ when $\gamma_1=\cdots=\gamma_m=1$, namely when $\tau_1=\cdots=\tau_m=0$, and thus $\sigma^{(1)},\dots,\sigma^{(m)}$ are near-ground states with respect to the {\bf same} instance $X_0\in\R^n$ of the problem.

Note that $\rho_{ij}=\Overlap(\sigma^{(i)},\sigma^{(j)})=|\bar{\Sigma}_{ij}|$. Thus, \[
\Sigma_{ij}\in [-\beta,-\beta+\eta]\cup [\beta-\eta,\beta],
\]
for $1\le i<j\le m$. In particular, there exists $2^{\binom{m}{2}}$ possible ``signs" for the off-diagonal entries for the matrix $\bar{\Sigma}$. With this observation, we now prove an auxiliary lemma.
\begin{lemma}\label{claim:perturbation-lemma}
Let $m\in\mathbb{Z}_+$, $K\triangleq 2^{\binom{m}{2}}$. Construct a family $M_k(x)$, $1\le k\le K$, of $m\times m$ matrices with unit diagonal entries, where each off-diagonal entry is defined in terms of $x$, as follows. Fix any ``sign-configuration" $\gamma^{(k)}=(\gamma_{ij}^{(k)}:1\le i<j\le m)\in\{-1,1\}^{K}$, $1\le k\le K$. let $M_k(x)\in\R^{m\times m}$ be the matrix defined by $(M_k(x))_{ii}=1$ for $1\le i\le m$, and $(M_k(x))_{ij}=(M_k(x))_{ji}=\gamma_{ij}^{(k)}x$. Then the following holds:
\begin{itemize}
    \item[(a)] Define $\varphi_k(x)\triangleq \sigma_{\min}(M_k(x))$, $1\le k\le K$. Then, for any $k$, there exists an $\epsilon_k>0$ such that $\varphi_k(x)>0$ for all $x\in(1-\epsilon_k,1)$.
    \item[(b)] Fix any $x\in(1-\min_{k\in[K]} \epsilon_k,1)$. Then $(M_k(x)+E)$ is invertible for every $1\le k\le K$, provided
\[
\|E\|_2 < \min_{1\le k\le K}\varphi_k(x).
\]
\end{itemize}
\end{lemma}
\begin{proof}{(of Lemma \ref{claim:perturbation-lemma})}
\begin{itemize}
    \item[(a)] Let $D_k(x)={\rm det}(M_k(x))$. Note that $D_k(0)=1$, thus $D_k\neq 0$ identically. Now observe that $D_k$ is a polynomial in $x$, of degree $m$. Thus there indeed exists an $\epsilon_k>0$ such that $D_k(x)\neq 0$ for $x\in(1-\epsilon_k,1)$. This yields $\varphi_k(x)>0$ whenever $x\in(1-\epsilon_k,1)$ as well. 
    \item[(b)] Fix any $M\in\R^{m\times m}$ with ${\rm rank}(M)=m$. Let $E\in\R^{m\times m}$ satisfy  ${\rm rank}(M+E)<m$. We claim $\|E\|_2\ge \sigma_{\min}(M)$. To see this, note that if $M+E$ is rank-deficient, then there exists a $v$ with $\|v\|_2=1$ such that $(M+E)v=0$. This yields $Ev=-Mv$, thus
    \[
    \|E\|_2\ge \|Ev\|_2 =\|Mv\|_2\ge \sigma_{\min}(M).
    \]
\end{itemize}
\end{proof}
We now return to the proof, where in the remainder we will make use of the  quantities defined in Lemma \ref{claim:perturbation-lemma}. We express $\bar{\Sigma}=\widehat{\Sigma}+E$. Here, $\widehat{\Sigma}\in\R^{m\times m}$ with unit diagonal entries, and $\widehat{\Sigma}_{ij}=\beta$ if $\ip{\sigma^{(i)}}{\sigma^{(j)}}>0$, and $\widehat{\Sigma}_{ij}=-\beta$ otherwise. The matrix $E\in\R^{m\times m}$ is such that $E_{ii}=0$ for $1\le i\le m$; and $|E_{ij}|\le \eta$ for $1\le i<j\le m$. 
Note  that,
\[
\|E\|_F^2 = \sum_{1\le i\le m}\sum_{1\le j\le m}E_{ij}^2 \le m^2 \eta^2.
\]
Since $\|E\|_2\le \|E\|_F$, we then conclude $\|E\|_2\le m\eta$. We now choose $\beta\in(1-\min_{1\le k\le K}\epsilon_k,1)$ where  the constants $\epsilon_k$ are defined in Lemma \ref{claim:perturbation-lemma}, and set 
\[
\eta(\beta) = \frac{\min_{1\le k\le K}\varphi_k(\beta)}{Nm},
\]
where $N$ is a (large) positive integer, to be tuned. Using Lemma \ref{claim:perturbation-lemma}, we have $\eta(\beta)>0$ for every $\beta\in(1-\min_k \epsilon_k,1)$, and any $N\in\mathbb{Z}_+$. Furthermore, Lemma \ref{claim:perturbation-lemma}{\rm (b)} also yields that under this choice of parameters, $\bar{\Sigma}$ is always invertible. $\bar{\Sigma}$ is, by construction, a covariance matrix thus has non-negative eigenvalues $\lambda_1(\bar{\Sigma})\ge \cdots \ge \lambda_m(\bar{\Sigma})>0$ with
\[
m={\rm trace}(\bar{\Sigma}) \ge m\lambda_m(\bar{\Sigma}).
\]
Thus we have \begin{equation}\label{eq:min-lambda-bar-Sigma-le1}
    \lambda_m(\bar{\Sigma})\le 1.
\end{equation}

Fix now $\tau_1,\dots,\tau_m\in\mathcal{I}\subset[0,1]$, and recall that $\gamma_i=\sqrt{1-\tau_i^2}$, $1\le i\le m$. We now express the covariance matrix $\Sigma$ in terms of $\bar{\Sigma}$, which depends only on $m,\beta$, and $\eta$. To that end, let $A={\rm diag}(\gamma_1,\dots,\gamma_m)\in\R^{m\times m}$ be a diagonal matrix. Observe that, 
\begin{equation}\label{eq:sum-two-psd}
\Sigma = A\bar{\Sigma}A +(I-A^2).
\end{equation}
Observe that as $\bar{\Sigma}$ is positive semidefinite, so do $A\bar{\Sigma}A$. Furthermore, as $1-\gamma_i^2\ge 0$ for $1\le i\le m$, the matrix $I-A^2$ is positive semidefinite as well. We now study the smallest eigenvalue $\lambda_m(\Sigma)$. 
\begin{lemma}\label{lemma:det-low-bd}
For any choices of $\tau_1,\dots,\tau_m\in\mathcal{I}$, it is the case that $\lambda_m(\Sigma)\ge \lambda_m (\bar{\Sigma})$. Hence,
\[
|\Sigma|\ge (\lambda_m (\bar{\Sigma}))^m>0,
\]
which is independent of the indices $\tau_1,\dots,\tau_m$. 
\end{lemma}
\begin{proof}
Recall the Courant-Fischer-Weyl variational characterization of the smallest singular value $\lambda_m(\Sigma)$ of a Hermitian matrix $\Sigma\in\R^{m\times m}$ \cite{horn2012matrix}:
\begin{equation}\label{eq:bar-sigma-lower-bd2}
\lambda_m(\Sigma)= \inf_{v:\|v\|_2=1}v^T\Sigma v.
\end{equation}
Then for any $v$ with $\|v\|_2=1$,
\begin{align*}
v^T\Sigma v &= v^T(I-A^2)v +v^TA\bar{\Sigma}Av \\
&\ge \sum_{1\le i\le m}(1-\gamma_i^2)v_i^2 + \lambda_m(\bar{\Sigma})\|Av\|_2^2\\
&=\sum_{1\le i\le m}(1-\gamma_i^2+\lambda_m(\bar{\Sigma})\gamma_i^2)v_i^2\\
&\ge \lambda_m(\bar{\Sigma})\|v\|_2^2 \\&= \lambda_m(\bar{\Sigma}),
\end{align*}
where the first equality uses \eqref{eq:sum-two-psd}, the first inequality uses \eqref{eq:bar-sigma-lower-bd2}, and the last inequality uses $\lambda_m(\bar{\Sigma})\le 1$ as established in \eqref{eq:min-lambda-bar-Sigma-le1}. Taking the infimum over all unit norm $v$, we conclude
\[
\lambda_m(\Sigma)\ge \lambda_m(\bar{\Sigma}).
\]
 Finally
\[
|\Sigma|=\prod_{1\le j\le m}\lambda_j(\Sigma)\ge \lambda_m(\Sigma)^m \ge \lambda_m(\bar{\Sigma})^m,
\]
as desired.
\end{proof}
By the Lemma~\ref{lemma:det-low-bd}, we have that $|\Sigma|$ is bounded away from zero by an explicit constant controlled solely by $m,\beta,\eta$, which in particular is independent of $\mathcal{I}$. Thus the union bound leading to \eqref{eq:prob} is indeed valid.

We finally show how \eqref{eq:cond1-on-beta-eta1} is fulfilled, which is to ensure
\[
\psi_N(\beta)\triangleq \beta - \eta(\beta)=\beta - \frac{\min_{1\le k\le K}\varphi_k(\beta)}{Nm}>1-h^{-1}\left(\frac{\epsilon}{2}\right).
\]
Notice  $1-h^{-1}(\epsilon/2)$ is strictly smaller than $1$. Observing now that $\psi_N(\beta)\to 1$ as $N\to +\infty$, and $\beta\to 1$, one can indeed find such $\beta$ and $\eta$. It suffices that $\beta$ satisfies
\[
1>\beta >\max\left\{1-\min_{1\le k\le K}\epsilon_k,1-\frac12 h^{-1}\left(\frac\epsilon2\right)\right\}.
\]
Having selected this value of $\beta>0$, prescribe now $\eta\triangleq \eta(\beta)$, by choosing $N\in\mathbb{Z}_+$ sufficiently large so that
\[
\eta = \frac{\min_{1\le k\le K}\varphi_k(\beta)}{Nm} <\frac12 h^{-1}\left(\frac{\epsilon}{2}\right).
\]
This concludes the proof of Theorem~\ref{thm:main-eps-energy}.
\end{proof}

\subsection{Proof of Theorem~\ref{thm:ogp-absent}}\label{sec:pf-absent-OGP}
The proof of Theorem~\ref{thm:ogp-absent} is based on the so-called \emph{second moment method}, but in addition uses several other ideas. We provide a short outline below for convenience.
\subsubsection*{Outline of the Proof of Theorem~\ref{thm:ogp-absent}}
\begin{itemize}
    \item Fix an $m\in\mathbb{N}$, $\rho\in(0,1)$, and a function $f:\mathbb{N}\to\R^+$ with $f(n)\in o(n)$. We first show that with high probability over $X\distr\mathcal{N}(0,I_n)$, there exists an $m-$tuple $\left(\sigma^{(i)}:1\le i\le m\right)$ of spin configurations $\sigma^{(i)}\in\bincube$ such that i) for $1\le i\le m$, $n^{-1/2}|\ip{\sigma^{(i)}}{X}| \le 2^{-f(n)}$; and ii) for $1\le i<j\le m$, $\rho-\bar{\rho} \le n^{-1}\ip{\sigma^{(i)}}{\sigma^{(j)}} \le \rho+\bar{\rho}$, provided $\bar{\rho}$ is sufficiently small.
    \item To start with, it is not even clear if for every $\bar{\rho}$ sufficiently small; there exists---deterministically---$\sigma^{(i)}\in\bincube$, $1\le i\le m$, such that  $\rho-\bar{\rho} \le n^{-1}\ip{\sigma^{(i)}}{\sigma^{(j)}} \le \rho+\bar{\rho}$ for $1\le i<j\le m$. We establish this using the so-called \emph{probabilistic method}~\cite{alon2016probabilistic}: we assign the coordinates $\sigma^{(i)}(j)\in\{-1,1\}$, $1\le i\le m$ and $1\le j\le n$ randomly according to the Rademacher distribution with an appropriate parameter\footnote{Here, we interpret the Rademacher distribution with parameter $p$ as the distribution supported on $\{-1,1\}$, which takes the value $+1$ with probability $p$.}; and show that with positive probability, such a configuration exists.
    \item We then let the random variable $M$ to count the number of $m-$tuples $\left(\sigma^{(i)}:1\le i\le m\right)$ of spin configurations which satisfy the desired properties. Our goal is to establish $\mathbb{P}(M\ge 1)=1-o_n(1)$. For this goal, we use the so-called \emph{second moment method} which uses the \emph{Paley-Zygmund inequality}: for a non-negative random variable $M$  taking integer values,
    \[
    \mathbb{P}(M\ge 1) \ge \frac{(\E{M}^2)}{\E{M^2}}.
    \]
    Namely, if the second moment $\E{M^2}$ is asymptotically $\E{M}^2\left(1+o_n(1)\right)$, in other words when ${\rm Var}(M)=o\left(\E{M}^2\right)$, we have that $\mathbb{P}(M\ge 1)=1-o_n(1)$, as desired.
    \item As is rather common with the applications of the second moment method, the computation of the second moment is challenging: it involves an expectation of a sum running over pairs of $m-$tuples of spin configurations. To compute this sum, we employ an overcounting idea.   \item To that end, fix an $\epsilon>0$ small; and let $I_\epsilon$ be the set of all integers in the set $[0,n(1-\epsilon)/2]\cup [n(1+\epsilon)/2,1]$. We now overestimate $\E{M^2}$ by dividing the sum into two components. Specifically, for two $m-$tuples of spin configurations $\mathcal{T} = \left(\sigma^{(i)}:1\le i\le m\right)$ and $\overline{\mathcal{T}} = \left(\overline{\sigma}^{(i)}:1\le i\le m\right)$, we distinguish two cases. The first case pertains the pairs $\left(\mathcal{T},\overline{\mathcal{T}}\right)$ for which there exists an  $i,j$ such that $d_H\left(\sigma^{(i)},\overline{\sigma}^{(j)}\right)\in I_\epsilon$. The second case pertains the pairs $\left(\mathcal{T},\overline{\mathcal{T}}\right)$ for which it is the case that for every $i,j$; $n(1-\epsilon)/2 <  d_H\left(\sigma^{(i)},\overline{\sigma}^{(j)}\right)<n(1+\epsilon)/2$. Namely, the second case essentially corresponds to the pairs of $m-$tuples that are nearly ``uncorrelated". The term $\epsilon$ introduced above essentially controls the ``residual correlation".
    \item We then find that due to cardinality constraints (via a certain asymptotics pertaining the binomial coefficients), the number of pairs of first kind is small, and the probability term can be neglected. See the proof for details.
    \item We then observe that the number of pairs of $m-$tuples of second kind dominates the second moment  term. For those pairs, however, the computation of their joint probability is tractable due to
    the fact that they are nearly uncorrelated.
    \item We then take a union bound over a certain choice of grid, for the goal of obtaining the event which involves a condition over {\bf all} $\beta\in[0,1]$. 
    \item Finally, sending $n\to\infty$ and $\epsilon\to 0$ carefully; we obtain our desired conclusion.
\end{itemize}
We now provide the complete proof.
\subsubsection*{Proof of Theorem~\ref{thm:ogp-absent}}
\begin{proof}

In what follows, denote by $S(m,\rho,\bar{\rho},E)$ to be the set of all $m-$tuples $\left(\sigma^{(i)}:1\le i\le m\right)$ of spin configurations $\sigma^{(i)} \in \bincube$ such that
\begin{itemize}
\item For every $1\le i<j\le m$, $\rho-\bar{\rho}\le \Overlap\left(\sigma^{(i)},\sigma^{(j)}\right)\le \rho+\bar{\rho}$. 
\item For every $1\le i\le m$, $n^{-1/2}\left|\ip{\sigma^{(i)}}{X}\right|\le E$. 
\end{itemize} 
Namely, $S(m,\rho,\bar{\rho},E)$ is a shorthand for the set $\mathcal{S}(m,\rho,\bar{\rho},\log_2 E,\{0\})$ introduced in Definition~\ref{def:overlap-set} with the modification as in Theorem~\ref{thm:ogp-absent}.

Let $m\in\mathbb{N}$, $\gamma\in(0,\frac12)$, and $\delta\in(0,\gamma)$. Define the set
\begin{equation}\label{eq:S-m-gamma-delta-set}
S(m,\gamma,\delta)\triangleq \left\{\left(\sigma^{(i)}:1\le i\le m\right):\sigma^{(i)}\in\bincube,\frac1n d_H\left(\sigma^{(i)},\sigma^{(i')}\right)\in[\gamma-\delta,\gamma+\delta],1\le i<i'\le m\right\}.
\end{equation}
Call the triple $(m,\gamma,\delta)$ {\bf admisible} if there exists an $N\triangleq N(m,\gamma,\delta)\in\mathbb{N}$ such that for every  $n\ge N$, $S(m,\gamma,\delta)\ne\varnothing$. We first prove that for any fixed $m\in\mathbb{N}$, $\gamma$, and $\delta$ sufficiently small;  $S(m,\gamma,\delta)\ne \varnothing$ for all sufficiently large $n$.

Note that this step is necessary: in order to ensure the the existence of $m-$tuples with desired ``energy levels" as required by the Theorem, one needs to ensure first that such $m-$tuples of spin configurations with pairwise constrained overlaps do exist; and this is quite non-trivial for $m>2$. We will later translate the condition on Hamming distances into a condition on their pairwise (normalized) overlaps.

\paragraph{ $S(m,\gamma,\delta)\ne\varnothing$ for $n$ sufficiently large.} 
We choose the spin configurations $\sigma^{(i)}\in\bincube$ randomly. 
Specifically, let $\sigma^{(i)}=\left(\sigma^{(i)}(j):1\le j\le n\right)\in\bincube$ be i.i.d. across $1\le i\le m$ and $1\le j\le n$ with
\[
\mathbb{P}\left(\sigma^{(i)}(j)=1\right) = \eta^*,
\]
where $\eta^*\in(0,1)$ is chosen so that
\[
\eta^*\left(1-\eta^*\right)=\frac12 \gamma. 
\]

Define now a sequence $\mathcal{E}_{i,i'}$ of events $1\le i<i'\le m$,
\[
\mathcal{E}_{i,i'}\triangleq \left\{\gamma-\delta \le \frac1n d_H\left(\sigma^{(i)},\sigma^{(i')}\right)\le \gamma+\delta\right\}.
\]
It suffices to establish
\[
\mathbb{P}\left(\bigcap_{1\le i<i'\le m}\mathcal{E}_{i,i'}\right)>0 \Leftrightarrow \mathbb{P}\left(\bigcup_{1\le i<i'\le m}\mathcal{E}_{i,i'}^c\right)<1.
\]
We next study $\mathbb{P}\left(\mathcal{E}_{1,2}^c\right)$. Define $Z_j \triangleq \ind\left\{\sigma^{(1)}(j)\ne \sigma^{(2)}(j)\right\}$, $1\le j\le n$. Note that $Z_j$ are i.i.d. Bernoulli variables with mean $\mathbb{E}[Z_j] = 2\eta^*\left(1-\eta^*\right)=\gamma$. 
Using now standard concentration results for the sum of i.i.d. Bernoulli variables~\cite{vershynin2010introduction}, we have that for any $\delta>0$ 
\[
\mathbb{P}\left(\mathcal{E}_{1,2}^c\right) = \mathbb{P}\left(\left|\frac1n \sum_{1\le j\le n}Z_j - \gamma\right|>\delta\right) \le \exp\left(-Cn\delta^2\right),
\]
for an absolute constant $C>0$. Now, the events $\mathcal{E}_{i,i'}$ are clearly equiprobable across $1\le i<i'\le m$. Applying a union bound,
\[
\mathbb{P}\left(\bigcup_{1\le i<i'\le m}\mathcal{E}_{i,i'}^c\right)\le \binom{m}{2} \exp\left(-Cn\delta^2\right).
\]
Since $m$ is constant, the claim follows. 

Fix an arbitrary $\rho\in\left(0,1\right)$; a ``proxy" for $\beta$ appearing in the statement of Theorem~\ref{thm:ogp-absent}. Suppose $\bar{\rho}$ is a sufficiently small parameter; a ``proxy" for $\eta$. Set
\begin{equation}\label{eq:gamma-and-delta}
    \gamma \triangleq \frac{1-\rho}{2}\in(0,\frac12)\quad\text{and}\quad \delta \triangleq \frac{\bar{\rho}}{2}.
\end{equation}
Observe that 
\[
\rho -\bar{\rho} \le \frac1n\ip{\sigma}{\sigma'}\le \rho+\bar{\rho} \iff \gamma - \delta \le \frac1n d_H\left(\sigma,\sigma'\right)\le \gamma+\delta.
\]

In what follows, we define certain sets and random variables; which depend on $n$ but is dropped in the notation. 
Recall the set $S(m,\gamma,\delta)$ from \eqref{eq:S-m-gamma-delta-set}. As we established, for every $m\in\mathbb{N}$, and $\gamma$, $S(m,\gamma,\delta)\ne\varnothing$ for all $\delta$ sufficiently small and all $n$ large. Define $L_\sigma$ to be the cardinality of set
\[
S_\sigma = \left\{\left(\sigma^{(i)}:1\le i\le m\right)\in S(m,\gamma,\delta):\sigma^{(1)}=\sigma\right\}\subset S(m,\gamma,\delta).
\]
Note that $L_\sigma$ is independent of $\sigma$. So we instead use the notation $L$ for 
\begin{equation}\label{eq:number-L}
    L\triangleq \left|\left\{\left(\sigma^{(i)}:1\le i\le m\right)\in S(m,\gamma,\delta):\sigma^{(1)}=\sigma\right\}\right|.
\end{equation}
We then have
\begin{equation}\label{eq:cardinality-of-s-k-gamma-delta}
\left|S(m,\gamma,\delta)\right| = 2^n L.
\end{equation}
Fix $f:\mathbb{N}\to\R^+$, the ``energy exponent" with sub-linear growth, $f(n)\in o(n)$; and let $E=2^{-f(n)}$. Consider 
\begin{equation}\label{eq:M-m-gamma-delta-E}
M\triangleq M(m,\gamma,\delta,E) = \sum_{\left(\sigma^{(i)}:1\le i\le m\right)\in\displaystyle S\left(m,\gamma,\delta\right)} \ind \left\{\left|Y_i\right|<2^{-f(n)}, 1\le i\le m\right\},
\end{equation}
where 
\begin{equation}\label{eq:Y-i-for-second-mom}
Y_i \triangleq \frac{1}{\sqrt{n}}\ip{\sigma^{(i)}}{X},\quad 1\le i\le m,
\end{equation}
and, $X\distr \mathcal{N}(0,I_n)$. Then $Y_i\distr \mathcal{N}(0,1)$, $1\le i\le m$, though not independent.

Namely, for $\gamma=\frac{1-\rho}{2}$, $\delta=\frac{\bar{\rho}}{2}$, and $E=2^{-f(n)}$, $M(m,\gamma,\delta,E)$ is a lower bound on the cardinality of the set $S(m,\rho,\bar{\rho},E)$ that we study in Theorem~\ref{thm:ogp-absent}.  In what follows, we study $\mathbb{P}(M\ge 1)$ and give a lower bound on it for an appropriate choice of parameters.
\paragraph{ Second moment method.} We recall the Paley-Zygmund inequality: if $M\ge 0$ is a non-negative integer valued random variable, then
\begin{equation}\label{thm:paley-zygmund}
\mathbb{P}(M\ge 1)=\mathbb{P}(M>0)\ge \frac{(\E{M})^2}{\E{M^2}}.
\end{equation}
For a short proof, see \cite[Exercise~2.4]{boucheron2013concentration}. 
In particular, to show $\pr(M \ge 1)=1-o_n(1)$, it suffices to establish \[
\E{M^2}=\E{M}^2(1+o_n(1)).
\]
\paragraph{ First moment computation.} Fix any $\left(\sigma^{(i)}:1\le i\le m\right)\in S(m,\gamma,\delta)$ \eqref{eq:S-m-gamma-delta-set}; and recall $Y_i$, $1\le i\le m$, from \eqref{eq:Y-i-for-second-mom}. To compute the joint probability, we first recover the structure of the covariance matrix $\Sigma\in\R^{m\times m}$. $\Sigma_{ii}=1$ for $1\le i\le m$; and
\[
\mathbb{E}[Y_iY_j]=
\frac1n\ip{\sigma^{(i)}}{\sigma^{(j)}}.
\]
Since $\left(\sigma^{(i)}:1\le i\le m\right)\in S(m,\gamma,\delta)$, it follows that for any $1\le i<j\le m$, $\Sigma_{ij}=\Sigma_{ji}\in[\rho-\bar{\rho},\rho+\bar{\rho}]$. In particular, 
\[
\Sigma = (1-\rho)I_m + \rho 11^T+E,
\]
where $E\in\R^{m\times m}$ is a perturbation matrix with $E_{ii}=0$ and $|E_{ij}|=|E_{ji}|\le \bar{\rho}$ for $1\le i<j\le m$. The spectrum of $(1-\rho)I_m+\rho 11^T$ consists of the eigenvalue $1+\rho(m-1)$ with multiplicity one; and the eigenvalue $1-\rho$ with multiplicity $m-1$. Clearly $\|E\|_2 \le \|E\|_F\le m\bar{\rho}$. Using now Wielandt-Hoffman Theorem (Theorem~\ref{thm:wielandt-hoffman}), we have that for $\bar{\rho}\ll \frac{1-\rho}{m}$, the matrix $\Sigma$ is invertible. In what follows, assume $\bar{\rho}$ is in this regime.

With this, we compute that for energy level $E=2^{-f(n)}$ (with exponent $f(n)\in \omega_n(1)\cap o(n)$), 
\[
\mathbb{P}\left(\left|Y_i\right|<2^{-f(n)},1\le i\le m\right) =\frac{1}{(2\pi)^{\frac{m}{2}}\left|\Sigma\right|^{\frac12}}\int_{{\bf y}\triangleq (y_1,\dots,y_m)\in[-E,E]^m}\exp\left(-\frac12 {\bf y}^T \Sigma^{-1}{\bf y}\right)\; d{\bf y}.
\]
Now, observe that for ${\bf y}\in [-E,E]^m$, $\exp\left(-\frac12 {\bf y}^T\Sigma^{-1}{\bf y}\right)=1+o_n(1)$, provided $\Sigma^{-1}$ is invertible (which we ensured). Under this condition, 
\[
\mathbb{P}\left(\left|Y_i\right|<2^{-f(n)},1\le i\le m\right) = \frac{2^m}{(2\pi)^{\frac{m}{2}} |\Sigma|^{\frac12}}E^{m}\left(1+o_n(1)\right).
\]
Equipped with this, we now give two expressions for the first moment. First, using \eqref{eq:M-m-gamma-delta-E} and the linearity of expectations, we obtain
\begin{equation}\label{eq:first-mom-M-exp-1}
\mathbb{E}[M]=\sum_{\left(\sigma^{(i)}:1\le i\le m\right)\in\displaystyle S\left(m,\gamma,\delta\right)} \frac{2^m}{(2\pi)^{\frac{m}{2}} |\Sigma|^{\frac12}}E^{m}\left(1+o_n(1)\right).
\end{equation}
Here, the tuple $\left(\sigma^{(i)}:1\le i\le m\right)$ induces an ``overlap pattern", which, in turn, induces the covariance matrix $\Sigma$. 

For the second, note that using Wieland-Hoffman inequality,  it is the case that for $\bar{\rho}\ll \frac{1-\rho}{m}$, there exists constants $C_1<C_2$---depending only on $m,\rho,\bar{\rho}$ and are independent of $n$---such that
\[
C_1<|\Sigma|<C_2.
\]
Namely, $|\Sigma|=O_n(1)$ across all $m-$tuples $\left(\sigma^{(i)}:1\le i\le m\right)\in S(m,\gamma,\delta)$. With this, we have 
\begin{equation}\label{eq:first-mom-M-exp-2}
\mathbb{E}[M]\ge 2^n L \frac{2^m}{(2\pi)^{\frac{m}{2}}C_2}E^m\left(1+o_n(1)\right).
\end{equation}
Above, we utilized the cardinality bound~\eqref{eq:cardinality-of-s-k-gamma-delta}.

\paragraph{ Second moment computation.} The computation for the second moment is more delicate, and involves a sum over pairs of $m-$tuples of spin configurations. For notational purposes, let $\mathcal{T}\triangleq \left(\sigma^{(i)}:1\le i\le m\right)$ and $\overline{\mathcal{T}}\triangleq \left(\overline{\sigma}^{(i)}:1\le i\le m\right)$. We have
\begin{equation}\label{eq:second-moment-of-M}
    \E{M^2}=\sum_{\mathcal{T},\overline{\mathcal{T}}\in S(m,\gamma,\delta)}\mathbb{P}\left(\left|Y_i\right|\le E,\left|\overline{Y_i}\right|\le E,1\le i\le m\right).
\end{equation}
Here the following variables are standard normal
\begin{equation}\label{eq:Y-and-bar-Y}
Y_i\triangleq \frac{1}{\sqrt{n}}\ip{\sigma^{(i)}}{X} \quad\text{and}\quad \overline{Y_i} \triangleq \frac{1}{\sqrt{n}}\ip{\overline{\sigma}^{(i)}}{X}, \quad \quad 1\le i\le m.
\end{equation}
Now, fix an arbitrary $\epsilon>0$, and define the set
\begin{equation}\label{eq:I-of-eps}
I_\epsilon = \mathbb{Z}\cap 
\left(\left[0,\frac{n}{2}(1-\epsilon)\right]\cup\left[\frac{n}{2}(1+\epsilon),n\right]\right).
\end{equation}
For the pairs $(\mathcal{T},\overline{\mathcal{T}})\in S(m,\gamma,\delta)\times S(m,\gamma,\delta)$ of spin configurations define the following family of $m^2$ sets
\begin{equation}\label{eq:s-eps-sets}
    S^{(ij)}(\epsilon)\triangleq \left\{\left(\mathcal{T},\overline{\mathcal{T}}\right)\in S(m,\gamma,\delta)\times S(m,\gamma,\delta):d_H\left(\sigma^{(i)},\overline{\sigma}^{(j)}\right)\in I_\epsilon\right\}
\end{equation}
for $1\le i,j\le m$. Let
\begin{equation}\label{eq:set-bar-S-2nd-moment}
\overline{S}\triangleq \left(S(m,\gamma,\delta)\times S(m,\gamma,\delta)\right)\setminus \left(\bigcup_{1\le i, j\le m}S^{(ij)}(\epsilon)\right).
\end{equation}

Note that the sets $S^{(ij)}(\epsilon)$ potentially intersect for different pairs $(i,j)$. 
This is the essence of the overcounting we utilize in the remainder, with the key idea being  that the overcounting can only increase the second moment. 

We next establish an upper bound on the cardinality of $S^{(ij)}(\epsilon)$. Recalling the quantity $L$ from \eqref{eq:number-L}, 
we have
\[
\left|S^{(ij)}(\epsilon)\right| = 2^n L^2 \left(\sum_{k\in I_\epsilon}\binom{n}{k}\right).
\]
The rationale for this is as follows. The first coordinate of $\mathcal{T}$ is chosen in $2^n$ different ways; and the remainder are filled in $L$ different ways. Having fixed this $m-$tuple; now using the constraint $d_H\left(\sigma^{(i)},\bar{\sigma}^{(j)}\right)\in I_\epsilon$, the object $\bar{\sigma}^{(j)}$ can be chosen in $\sum_{k\in I_\epsilon}\binom{n}{k}$ different ways; and finally having fixed $\overline{\sigma}^{(j)}$, the rest of the coordinates of the $m-$tuple $\overline{\mathcal{T}}$ can  now be filled in $L$ different ways. 

Applying the Stirling's formula \eqref{thm:stirling} and using $|I_\epsilon|=n^{O(1)}$ 
    \[
    \sum_{k\in I_\epsilon}\binom{n}{k}\le n^{O(1)} \binom{n}{n\frac{1-\epsilon}{2}} = \exp_2\left(nh_b\left(\frac{1-\epsilon}{2}\right)+O(\log_2 n)\right).
    \]
Thus, 
\begin{equation}\label{eq:cardinality-of-s-ij-of-eps}
    \left|S^{(ij)}(\epsilon)\right| \le 2^n  L^2 \exp_2\left(nh_b\left(\frac{1-\epsilon}{2}\right)+O\left(\log_2 n\right)\right), \quad \text{for}\quad 1\le i\ne j\le m.
\end{equation}
\paragraph{ Overcounting argument.} Now, for any pair $(\mathcal{T},\overline{\mathcal{T}})\in S(m,\gamma,\delta)\times S(m,\gamma,\delta)$, let
\[
p\left(\mathcal{T},\overline{\mathcal{T}}\right)\triangleq \mathbb{P}\left(\left|Y_i\right|\le E,\left|\overline{Y_i}\right|\le E,1\le  i\le m\right).
\]
We now compute the second moment. In terms of the sets introduced in \eqref{eq:s-eps-sets} \eqref{eq:set-bar-S-2nd-moment}, we have
\begin{align*}
    \E{M^2} & = \sum_{(\mathcal{T},\overline{\mathcal{T}})\in S(m,\gamma,\delta)\times S(m,\gamma,\delta)}p\left(\mathcal{T},\overline{\mathcal{T}}\right) \\
    &\le \sum_{1\le i,j\le m}\sum_{(\mathcal{T},\overline{\mathcal{T}})\in S^{(ij)}(\epsilon)} p\left(\mathcal{T},\overline{\mathcal{T}}\right) + \sum_{(\mathcal{T},\overline{\mathcal{T}})\in \overline{S}} p\left(\mathcal{T},\overline{\mathcal{T}}\right).
\end{align*}
Consequently,
\begin{equation}\label{eq:second-moment-overcount-bd}
\E{M^2}\le \underbrace{m^2 2^n L^2 \exp_2\left(nh_b\left(\frac{1-\epsilon}{2}\right) + O\left(\log_2 n\right)\right)}_{\triangleq A_\epsilon} + \underbrace{\sum_{(\mathcal{T},\overline{\mathcal{T}})\in \overline{S}} p\left(\mathcal{T},\overline{\mathcal{T}}\right)}_{\triangleq B_\epsilon}.
\end{equation}


\paragraph{ Study of the $A_\epsilon$ term.} Using the crude lower bound \eqref{eq:first-mom-M-exp-2} on the first moment, we arrive at
\begin{align*}
\frac{A_\epsilon}{\E{M}^2} &\le \frac{m^2 2^n L^2\exp_2\left(nh_b\left(\frac{1-\epsilon}{2}\right)+O\left(\log_2 n\right) \right)}{2^{2n}L^2 2^{2m}\left(2\pi\right)^{-m}C_2^{-2} E^{2m}\left(1+o_n(1)\right)}\\
&=\exp_2\left(-n+nh_b\left(\frac{1-\epsilon}{2}\right)-2m\log_2 E +O\left(\log_2 n\right)+O(1)\right)\\
&=\exp_2\left(-n\left(1-h_b\left(\frac{1-\epsilon}{2}\right)\right)-2mf(n)+O(\log_2 n)\right) \\
&=\exp_2\left(-n\left(1-h_b\left(\frac{1-\epsilon}{2}\right)\right)+o(n)\right).
\end{align*}
Above, we used the fact that $E=2^{-f(n)}$ for some exponent $f(n)\in o(n)$; and the fact $\epsilon>0$ hence $h_b\left(\frac{1-\epsilon}{2}\right)<1$. Consequently,
\begin{equation}\label{A-eps-upper-bd-oley}
    \frac{A_\epsilon}{\E{M}^2} \le  \exp_2\left(-\Theta(n)\right).
\end{equation}
\paragraph{ Study of the $B_\epsilon$ term.} This term is more involved, and it is the one that leads to the dominant contribution to the second moment of $M$. 

Fix a pair $(\mathcal{T},\overline{\mathcal{T}})\in S(m,\gamma,\delta)\times S(m,\gamma,\delta)$, and recall the associated standard normal variables $Y_i,\overline{Y_i}$, $1\le i\le n$ per \eqref{eq:Y-and-bar-Y}. Our goal is to  study the probability
\[
p\left(\mathcal{T},\overline{\mathcal{T}}\right) = \mathbb{P}\left(\left|Y_i\right|,\left|\overline{Y_i}\right|\le \epsilon,1\le i\le m\right).
\]
To  that end, fix $1\le i,j\le m$. We study the covariance  between $Y_i$ and $\overline{Y_j}$. Fixing a pair $(\mathcal{T},\overline{\mathcal{T}})\in \overline{S}$, we have
\[
d_H\left(\sigma^{(i)},\overline{\sigma}^{(j)}\right) \in \left[\frac{n}{2}(1-\epsilon),\frac{n}{2}(1+\epsilon)\right],\quad \text{for all}\quad i,j.
\]
Then,
\[
\frac1n\ip{\sigma^{(i)}}{\overline{\sigma}^{(j)}}\in [-\epsilon,\epsilon].
\]
Let $\Sigma_\epsilon\in \R^{2m\times 2m}$ be the covariance matrix for the random vector $(Y_1,\dots,Y_m,\overline{Y_1},\dots,\overline{Y_{m}})$. Observe that it has the following ``block" structure:
\[
\Sigma_\epsilon = \displaystyle\begin{pmatrix}
 \Sigma_{\mathcal{T}}
  & \rvline & E \\
\hline
  E& \rvline &
 \Sigma_{\overline{\mathcal{T}}}
\end{pmatrix} \in\R^{2m\times 2m}.
\]
Here, $\Sigma_{\mathcal{T}}\in\R^{m\times m}$ is the covariance matrix corresponding to the random vector $(Y_1,\dots,Y_m)$; $\Sigma_{\overline{\mathcal{T}}}\in\R^{m\times m}$ is the covariance matrix corresponding to the random vector $(\overline{Y_1},\dots,\overline{Y_m})$; and $E\in\R^{m\times m}$ is given by $E_{ij} =\E{Y_i\overline{Y_j}}$. 
We have that for $1\le i\le m$,
\[
\left(\Sigma_{\mathcal{T}}\right)_{ii}= 1 =
\left(\Sigma_{\overline{\mathcal{T}}}\right)_{ii},
\]
and for $1\le i<j\le m$,
\[
\left(\Sigma_{\mathcal{T}}\right)_{ij},\quad \left(\Sigma_{\overline{\mathcal{T}}}\right)_{ij} \quad \in [\rho-\bar{\rho},\rho+\bar{\rho}].
\]
Moreover, for $1\le i<j\le m$,
\[
\left|E_{ij}\right|=\left|E_{ji}\right|\le \epsilon.
\]
We now invert the $2\times 2$ block matrix $\Sigma_\epsilon$, while keeping in mind that the block, $\Sigma_{\mathcal{T}}$, is invertible. Observe that
\[
 \displaystyle\begin{pmatrix}
 \Sigma_{\mathcal{T}}
  & \rvline & E \\
\hline
  E& \rvline &
 \Sigma_{\overline{\mathcal{T}}}
\end{pmatrix}   \displaystyle\begin{pmatrix}
I
  & \rvline & -\Sigma_{\mathcal{T}}^{-1}E\\
\hline
  O& \rvline &
 I
\end{pmatrix}   = \displaystyle\begin{pmatrix}
 \Sigma_{\mathcal{T}}
  & \rvline & O\\
\hline
  E& \rvline &
\Sigma_{\overline{\mathcal{T}}}-E\Sigma_{\mathcal{T}}^{-1}E
\end{pmatrix}.
\]
With this decomposition and the fact that the determinant of a ``block triangular" matrix is the product of the determinants of blocks constituting the diagonal, we arrive at 
\begin{align*}
    \left|\Sigma_\epsilon\right|& = \left|\Sigma_{\mathcal{T}}\right|\cdot \left|\Sigma_{\overline{\mathcal{T}}}-E\Sigma_{\mathcal{T}}^{-1}E\right|\\
    &=\left|\Sigma_{\mathcal{T}}\right|\cdot\left|\Sigma_{\overline{\mathcal{T}}}\right|\cdot \left|I-\Sigma_{\overline{\mathcal{T}}}^{-1}E\Sigma_{\mathcal{T}}^{-1}E\right|,
\end{align*}
where we have used the fact that $\Sigma_{\widehat{\mathcal{T}}}$ is invertible as well, to pull the term outside.

Note that for $\epsilon$ sufficiently small; the determinant $\left|I-\Sigma_{\overline{\mathcal{T}}}^{-1}E\Sigma_{\mathcal{T}}^{-1}E\right|$ is non-zero over all choices of $\Sigma_{\mathcal{T}}$ and $\Sigma_{\overline{\mathcal{T}}}$. Namely, provided $\epsilon$ is small; $\Sigma_\epsilon$ is invertible uniformly across all $\Sigma_{\mathcal{T}}$ and $\Sigma_{\overline{\mathcal{T}}}$. In the remainder, assume $\epsilon>0$ though sufficiently small. 
We now define the object
\[
\overline{\varphi}(\epsilon)\triangleq \overline{\varphi}\left(\Sigma_{\mathcal{T}},\Sigma_{\overline{\mathcal{T}}},E\right) = \left|I-\Sigma_{\overline{\mathcal{T}}}^{-1}E\Sigma_{\mathcal{T}}^{-1}E\right|.
\]
Note that $\overline{\varphi}(\cdot)$ is a polynomial in the entries $E_{ij}$, $1\le i,j\le m$; as well as in the entries of matrices $\Sigma_{\mathcal{T}}$ and $\Sigma_{\overline{\mathcal{T}}}$. Furthermore, 
$\overline{\varphi}\to 1$ as $\epsilon \to 0$. 
With this we write
\begin{equation}\label{eq:determinant-of-sigma-of-eps}
  \left|\Sigma_\epsilon\right| = \left|\Sigma_{\mathcal{T}}\right|\left|\Sigma_{\overline{\mathcal{T}}}\right| \overline{\varphi}(\epsilon).  
\end{equation}
We now compute
\begin{align*}
    p\left(\mathcal{T},\overline{\mathcal{T}}\right) &= \left(2\pi\right)^{-m}\left|\Sigma_\epsilon\right|^{-\frac12} \int_{{\bf y}=(y_1,\dots,y_m,\overline{y_1},\dots,\overline{y_m})\in[-E,E]^{2m}}\exp\left(-\frac12 {\bf y}^T\Sigma_\epsilon^{-1} {\bf y}\right)\; d{\bf y} \\
    &=\left(2\pi\right)^{-m}\left|\Sigma_\epsilon\right|^{-\frac12}2^{2m}E^{2m}\left(1+o_n(1)\right) \\
    &=\left(\overline{\varphi}(\epsilon)\right)^{-\frac12}\left(1+o_n(1)\right)\left(\left(2\pi\right)^{-\frac{m}{2}}\left|\Sigma_{\mathcal{T}}\right|^{-\frac12}\left(2E\right)^m\right)\left(\left(2\pi\right)^{-\frac{m}{2}}\left|\Sigma_{\overline{\mathcal{T}}}\right|^{-\frac12}\left(2E\right)^m\right).
\end{align*}
Here, the first line uses the definition of $p\left(\mathcal{T},\overline{\mathcal{T}}\right)$ together with the formulae for the multivariate normal density; the second line  uses the fact when ${\bf y}\in [-E,E]^{2m}$ then $\exp\left(-\frac12 {\bf y}^T \Sigma_{\epsilon}^{-1}{\bf y}\right)=1+o_n(1)$ provided $\Sigma_{\epsilon}$ is invertible (which we ensured); and the third line uses \eqref{eq:determinant-of-sigma-of-eps}. 
Thus,
\begin{equation}\label{eq:b-epsilon-massaged}
B_\epsilon = \sum_{(\mathcal{T},\overline{\mathcal{T}})\in\overline{S}}p\left(\mathcal{T},\overline{\mathcal{T}}\right) = \overline{\varphi}(\epsilon)\left(1+o_n(1)\right)\sum_{(\mathcal{T},\overline{\mathcal{T}})\in\overline{S}}\left(\left(2\pi\right)^{-\frac{m}{2}}\left|\Sigma_{\mathcal{T}}\right|^{-\frac12}\left(2E\right)^m\right)\left(\left(2\pi\right)^{-\frac{m}{2}}\left|\Sigma_{\overline{\mathcal{T}}}\right|^{-\frac12}\left(2E\right)^m\right).
\end{equation}
We now square the expression  \eqref{eq:first-mom-M-exp-1}, keep only the terms corresponding to $\overline{S}$; and lower bound the square of the first moment
\begin{equation}\label{eq:first-mom-sq-massaged}
\E{M}^2 \ge \left(1+o_n(1)\right)\sum_{(\mathcal{T},\overline{\mathcal{T}})\in \overline{S}} \left(\left(2\pi\right)^{-\frac{m}{2}}\left|\Sigma_{\mathcal{T}}\right|^{-\frac12}\left(2E\right)^m\right)\left(\left(2\pi\right)^{-\frac{m}{2}}\left|\Sigma_{\overline{\mathcal{T}}}\right|^{-\frac12}\left(2E\right)^m\right).
\end{equation}
Combining \eqref{eq:b-epsilon-massaged} and \eqref{eq:first-mom-sq-massaged}, we arrive at
\begin{equation}\label{eq:b-eps-over-first-mom-sq}
\frac{B_\epsilon}{\E{M}^2} \le \left(1+o_n(1)\right)\left(\overline{\varphi}(\epsilon)\right)^{-\frac12}.
\end{equation}
\paragraph{ Applying Paley-Zygmund Inequality.} Applying the Paley-Zygmund inequality \eqref{thm:paley-zygmund},

\begin{equation}\label{eq:paley-baba-oley}
    \mathbb{P}\left(M \ge 1\right)\ge \frac{\E{M}^2}{\E{M^2}} \ge \frac{\E{M}^2}{A_\epsilon+B_\epsilon} = \frac{1}{ \frac{A_\epsilon}{\E{M}^2}+\frac{B_\epsilon}{\E{M}^2} }\ge \frac{1}{\exp\left(-\Theta(n)\right)+\left(1+o_n(1)\right)\overline{\varphi}(\epsilon)^{-1/2}}.
\end{equation}
Here, the second inequality uses the overcounting upper bound \eqref{eq:second-moment-overcount-bd}; and the third inequality uses the upper bounds \eqref{A-eps-upper-bd-oley} and \eqref{eq:b-eps-over-first-mom-sq}. 

\paragraph{ Combining everything.} The reasoning above remains valid if (a) $\rho>\bar{\rho}$ and (b) $\bar{\rho}\ll \frac{1-\rho}{m}$. Now, choose \[
\bar{\rho} = \frac{\eta}{1000m}.
\] 
Here, the choice of the constant $1000$ is arbitrary. Next, let $\ell$ be the largest positive integer such that $2\ell\eta<1-\eta$. Consider the ``grid" $\rho_k = 2k \eta$, $1\le k\le \ell$, and intervals $I_k = [(2k-1)\eta,(2k+1)\eta]=[\rho_k-\eta,\rho_k+\eta]$ centered at $\rho_k$. Since 
\[
\left[\rho_k - \bar{\rho},\rho_k+\bar{\rho}\right]\subset \left[\rho_k-\eta,\rho_k+\eta\right]
\]
it follows by using~\eqref{eq:paley-baba-oley} that 
\begin{equation}\label{eq:pz-over-grid}
\mathbb{P}\left(S\left(m,\rho_k,\eta,E\right)\ne\varnothing\right) \ge \mathbb{P}\left(S\left(m,\rho_k,\bar{\rho},E\right)\ne\varnothing\right) \ge \frac{1}{\exp\left(-\Theta(n)\right)+\left(1+o_n(1)\right)\overline{\varphi}_k(\epsilon)^{-1/2}},
\end{equation}
where $\overline{\varphi}_k\left(\cdot\right)$ is a continuous function with the property that $\overline{\varphi}_k(\epsilon)\to 1$ as $\epsilon\to 0$. Taking a union bound over $1\le k\le \ell$, we arrive at
\begin{equation}\label{eq:limite-burdan-gec}
\mathbb{P}\left(\underbrace{\bigcap_{1\le k\le \ell} \left\{S\left(m,\rho_k,\eta,E\right)\ne\varnothing\right\}}_{\triangleq \mathcal{E}_{\rm aux}}\right)\ge 1-\ell \frac{\exp\left(-\Theta(n)\right)+\left(1+o_n(1)\right)\overline{\varphi}(\epsilon)^{-1/2}-1}{\exp\left(-\Theta(n)\right)+\left(1+o_n(1)\right)\overline{\varphi}(\epsilon)^{-1/2}},
\end{equation}
where $\overline{\varphi}(\cdot) = \min_{1\le k\le \ell}\overline{\varphi}_k(\cdot)$. In particular, since $\ell$ is finite, it follows $\overline{\varphi}(\epsilon)\to 1$ as $\epsilon \to 0$. 

We now carefully send $n$ and $\epsilon$ to their corresponding limits. Note that the asymptotic expressions (in $n$) given above are valid so long as $\epsilon>0$---see, e.g. \eqref{A-eps-upper-bd-oley}. Thus, we must send $n\to\infty$ first, while keeping $\epsilon>0$ fixed. We clearly have
\[
1\ge \limsup_{n\to\infty}\mathbb{P}( \mathcal{E}_{\rm aux}).
\]
Furthermore, while keeping $\epsilon>0$ and sending $n\to\infty$ in \eqref{eq:limite-burdan-gec}, we obtain
\[
\liminf_{n\to\infty}\mathbb{P}( \mathcal{E}_{\rm aux}) \ge 1- \ell \cdot \frac{\overline{\varphi}(\epsilon)^{-1/2}-1}{\overline{\varphi}(\epsilon)^{-1/2}}.
\]
Note that the sequence $\left\{\mathbb{P}( \mathcal{E}_{\rm aux})\right\}_{n\ge 1}$ (note that $\mathcal{E}_{\rm aux}$ implicitly depends on $n$) is not a function of $\epsilon$---and the lower bound holds true for every $\epsilon$ sufficiently close to zero. Moreover, $\ell$ is a constant. For this reason, we can now safely send $\epsilon\to 0$ to obtain
\[
\liminf_{n\to\infty}\mathbb{P}( \mathcal{E}_{\rm aux})\ge 1.
\]
Hence
\[
1\ge \limsup_{n\to\infty}\mathbb{P}\left( \mathcal{E}_{\rm aux}\right)\ge \liminf_{n\to\infty}\mathbb{P}\left( \mathcal{E}_{\rm aux}\right)\ge 1
\]
implying
\[
\lim_{n\to\infty}\mathbb{P}\left( \mathcal{E}_{\rm aux}\right)=1.
\]
Finally, observe that on high probability event $\mathcal{E}_{\rm aux}$, it is the case that for each of $[\eta,3\eta],[3\eta,5\eta],\dots$, there exists an $m-$tuple of spin configurations (with appropriate energy)  whose pairwise overlaps are contained in the chosen interval. Since for each $\beta\in[0,1]$; $[\beta-3\eta,\beta+3\eta]$ contains a full interval $[(2k-1)\eta,(2k+1)\eta]$; we conclude that 
\[
\mathbb{P}\left(\forall \beta\in[0,1]:\mathcal{S}\left(m,\beta,3\eta,2^{-f(n)}\right)\ne\varnothing\right)\ge \mathbb{P}\left(\mathcal{E}_{\rm aux}\right) = 1-o_n(1).
\]
The  above reasoning remains true for every $\eta>0$. Taking $\frac{\eta}{3}$ in place of $\eta$ yields the desired conclusion.

\end{proof}
\subsection{Proof of Theorem~\ref{thm:m-ogp-superconstant-m}}\label{sec:m-ogp-superconstant-m}
\subsubsection*{Case 1: $\omega\left(\sqrt{n\log_2 n}\right)\le E_n\le o(n)$.}
\begin{proof}
Let $g(n)$ be an arbitrary function with growth
\[
\omega(1)\le g(n) \le o\left(\frac{E_n^2}{n\log_2 n}\right).
\]
We take $m,\beta$, and $\eta$ per~\eqref{eq:beta-eta-m-n-main-supconstant}, that is
\[
m=\frac{2n}{E_n}, \quad \beta = 1-\frac{2g(n)}{E_n},\quad\text{and}\quad \eta = \frac{g(n)}{2n}.
\]
Define next several auxiliary parameters. First, set $\phi(n)$ by the expression
 \begin{equation}\label{eq:E_n-phi-n-sqrt-nlogn}
E_n=\phi(n)\sqrt{n\log n}.
\end{equation}
Since $\omega\left(\sqrt{n\log_2 n}\right)\le E_n\le o(n)$, it holds that
\begin{equation}\label{eq:phi(n)-up-low-bd}
\omega_n(1)\le \phi(n)\le o\left(\sqrt{\frac{n}{\log n}}\right).
\end{equation}
Moreover, in terms of $\phi(\cdot)$, the growth condition on $g$ translates as
\begin{equation}\label{eq:g(n)-up-low-bd}
\omega_n(1)\le g(n)\le o\left(\phi(n)^2\right).
\end{equation}
Introduce another parameter $\nu_n$ via
\begin{equation}\label{eq:nu-n-g-over-f}
    \nu_n = \frac{g(n)}{E_n} = \frac{g(n)}{\phi(n)\sqrt{n\log n}}.
\end{equation}
Thus, in terms of $g(n),\phi(n)$, and $\nu_n$; the parameters $m,\beta,\eta$ chosen as above satisfy the following relations:
\begin{equation}\label{eq:m-chosen-this-way}
    m=2\frac{n}{E_n} = \frac{2\sqrt{n}}{\phi(n)\sqrt{\log n}},
\end{equation}
\begin{equation}\label{eq:eta-nu-n-over-m}
\eta  =\frac{g(n)}{2n} = \frac{g(n)}{\phi(n)\sqrt{n\log n}}\cdot \frac{\phi(n)\sqrt{\log n}}{2\sqrt{n}}= \frac{\nu_n}{m};
\end{equation}
and
\begin{equation}\label{eq:beta:1-2nu-n}
    \beta = 1-\frac{2g(n)}{E_n} =  1-2\nu_n = 1-2\frac{g(n)}{\phi(n)\sqrt{n\log n}}.
\end{equation}
In particular, it holds that 
\begin{equation}\label{eq:eta-1-beta-over2m}
    \eta=\frac{1-\beta}{2m}.
\end{equation} 
The expressions~\eqref{eq:m-chosen-this-way}-\eqref{eq:eta-1-beta-over2m} will be convenient for handling certain expressions appearing below.

We will establish $m-$OGP for the interval $[\beta-\eta,\beta]$, where $m,\beta,\eta$ are chosen as above. As a sanity check, note that the interval $[\beta-\eta,\beta]$ has length $\eta$, and for our result to be non-vacuous, it should be the case that the overlap region is not void, that is
\[
\left|(n\beta-n\eta,n\beta)\cap \mathbb{Z}\right|\ge 1.
\] Indeed
\[
n\eta = \frac{n\nu_n}{m}= \frac12 g(n)=\omega_n(1),
\]
thus the region is not void.

Recall now
\[
Y_i(\tau) = \sqrt{1-\tau^2}X_0 + \tau X_i\in\R^n,\quad\text{for}\quad 1\le i\le m\quad\text{and}\quad \tau\in\mathcal{I}.
\]
In order to apply first moment method and Markov's inequality, we essentially need two bounds: 1) a bound on the cardinality of the $m-$tuples $(\sigma^{(i)}:1\le i\le m)$ of spin configurations whose pairwise overlaps are constrained to $[\beta-\eta,\beta]$, and 2) a bound on a certain (joint) probability.

To that end, define 
\[
S(\beta,\eta,m)\triangleq \left\{(\sigma^{(1)},\dots,\sigma^{(m)}):\sigma^{(i)}\in\{-1,1\}^n, \frac1n\ip{\sigma^{(i)}}{\sigma^{(j)}}\in[\beta-\eta,\beta],1\le i<j\le m\right\}
\]
and
\[
N(\beta,\eta,m,E_n,\mathcal{I}) =\sum_{(\sigma^{(i)}:1\le i\le m)\in S(\beta,\eta,m)}\ind  \left\{\exists \tau_1,\dots,\tau_m\in\mathcal{I}:\frac{1}{\sqrt{n}}|\ip{\sigma^{(i)}}{Y_i(\tau_i)}| \le 2^{-E_n},1\le i\le m\right\}.
\]
Observe that, with these notation, we have 
\[
N(\beta,\eta,m,E_n,\mathcal{I}) = |\mathcal{S}(\beta,\eta,m,E_n,\mathcal{I})|.
\]
Thus, by Markov's inequality
\[
\mathbb{P}\left(\mathcal{S}(\beta,\eta,m,E_n,\mathcal{I})\ne\varnothing\right) = \mathbb{P}\left(N(\beta,\eta,m,E_n,\mathcal{I})\ge 1\right)\le \mathbb{E}[N(\beta,\eta,m,E_n,\mathcal{I})].
\]
We will establish that with the parameters chosen as above, $\mathbb{E}[N(\beta,\eta,m,E_n,\mathcal{I})] = \exp(-\Theta(n))$, which will conclude the proof. 

{\bf Step 1. Cardinality upper bound.} We now upper bound the number of $m-$tuples $(\sigma^{(i)}:1\le i\le m)$ of spin configurations with (pairwise) overlaps constrained to $[\beta-\eta,\beta]$, that is, we upper bound the cardinality $|S(\beta,\eta,m)|$. For this we rely on Lemma~\ref{lemma:binomial-k-o(n)}.

Now, for $\sigma^{(1)}$, there are $2^n$ choices. Furthermore, for any fixed $\rho\in[\beta-\eta,\beta]$, there exists
\[
\binom{n}{n\frac{1-\rho}{2}}
\]
spin configurations $\sigma'$ for which $\frac1n \ip{\sigma}{\sigma'}=\rho$. With this, the number of choices for $\sigma^{(2)}$ evaluates
\[
    \sum_{\rho:\beta-\eta\le \rho\le \beta,n\rho \in\mathbb{Z}} \binom{n}{n\frac{1-\rho}{2}}.
\]
With this, the number of all such $m-$tuples $(\sigma^{(i)}:1\le i\le m)$ with $n^{-1}\ip{\sigma^{(i)}}{\sigma^{(j)}}\in[\beta-\eta,\beta]$, $1\le i<j\le m$, is at most
\begin{equation}\label{eq:superconst-ogp-card-up-bd}
    2^n \left(  \sum_{\rho:\beta-\eta\le \rho\le \beta,n\rho \in\mathbb{Z}} \binom{n}{n\frac{1-\rho}{2}}\right)^{m-1}.
\end{equation}
Observe that with our choice of parameters where $1-\beta=o_n(1)$ and $\eta=o_n(1)$,
\begin{equation}\label{eq:bin-coeff-max-up-bd}
\displaystyle\max_{\substack{\rho:\rho\in[\beta-\eta,\beta]\\\rho n\in\mathbb{Z}}} \binom{n}{n\frac{1-\rho}{2}}= \binom{n}{n\frac{1-\beta+\eta}{2}}.
\end{equation}
Recalling \eqref{eq:eta-1-beta-over2m} and the fact $m=\omega_n(1)$, and therefore $m^{-1}=o_n(1)$, we have
\[
1-\beta+\eta = (1-\beta)\left(1+\frac{1}{2m}\right)=(1-\beta)(1+o_n(1)).
\]
Consequently, using \eqref{eq:beta:1-2nu-n}, we conclude
\begin{align}
n\frac{1-\beta+\eta}{2} &=\frac{n}{2}(1-\beta)(1+o_n(1)) \\
&=\frac{n}{2} \frac{2g(n)}{\phi(n)\sqrt{n\log_2 n}}(1+o_n(1))\\
&=\frac{g(n)\sqrt{n}}{\phi(n)\sqrt{\log_2 n}}(1+o_n(1))\label{eq:g-n-sart-n-phi-n-sqrt-log2-on}\\
&=\frac{ng(n)}{E_n}(1+o_n(1)).\label{eq:binomial-auxiliary}
\end{align}
Next, observe that
\[
g(n) = o\left(\phi(n)^2\right) = o\left(\frac{E_n^2}{n\log_2 n}\right) = o\left(E_n\right),
\]
as $E_n = o\left(n\right)$ which is trivially $o\left(n\log_2 n\right)$. Thus, it follows from \eqref{eq:binomial-auxiliary} that $n\frac{1-\beta+\eta}{2} = o(n)$. Thus, Lemma~\ref{lemma:binomial-k-o(n)} applies.
As a sanity check, we also verify $n\frac{1-\beta+\eta}{2} =\omega(1)$, so that the counting bound is not vacuous: using \eqref{eq:phi(n)-up-low-bd} and \eqref{eq:g-n-sart-n-phi-n-sqrt-log2-on}, we have $\phi(n)^{-1} = \omega\left(\sqrt{\frac{\log_2 n}{n}}\right)$. Thus
\[
n\frac{1-\beta+\eta}{2} = \omega(g(n))=\omega(1).
\]
We now proceed to control  the term
\[
\binom{n}{n\frac{1-\beta+\eta}{2}}.
\]
As we have verified, $n\frac{1-\beta+\eta}{2} = o(n)$. Thus we are indeed in the setting of Lemma~\ref{lemma:binomial-k-o(n)}.

Observe that using \eqref{eq:binomial-auxiliary}, 
\[
\log_2 \frac{n}{n\frac{1-\beta+\eta}{2}} =  \log_2 \frac{E_n}{g(n)} = O\left(\log_2 n\right),
\]
since $E_n=o(n)$. 
We now apply Lemma~\ref{lemma:binomial-k-o(n)} to conclude that
\begin{equation}\label{eq:n-choose-n-1-beta-eta}
\binom{n}{n\frac{1-\beta+\eta}{2}} = \exp_2\left((1+o_n(1))\frac{ng(n)}{E_n}O\left( \log_2 n\right)\right) = \exp_2\left(O\left(\frac{ng(n)}{E_n} \log_2 n\right)\right).
\end{equation}
Note also that by \eqref{eq:eta-nu-n-over-m}
\begin{align*}
\left|[n\beta-n\eta,n\beta]\cap\mathbb{Z}\right|&=O(n\eta)\\
&=O(g(n)).
\end{align*}
Consequently, using \eqref{eq:bin-coeff-max-up-bd}, \eqref{eq:n-choose-n-1-beta-eta}, the fact $E_n=\phi(n)\sqrt{n\log_2 n}$ and the cardinality bound above in this order
\[
\sum_{\rho:\beta-\eta\le \rho\le \beta,\rho n\in\mathbb{Z}}\binom{n}{n\frac{1-\rho}{2}} \le \exp_2\left(O\left(\frac{g(n)\sqrt{n\log_2 n}}{\phi(n)}\right)+O(\log_2 g(n))\right).
\]
Now, since $\frac1\phi  = \omega\left(\sqrt{\frac{\log_2 n}{n}}\right)$, we have
\[
\frac{g(n)\sqrt{n\log_2 n}}{\phi(n)}=\omega(g(n)\log_2 n);
\]
and therefore the term, $O(\log_2 g(n))$ appearing in the bound above, is lower order. Thus we conclude
\begin{equation}\label{eq:toward-counting}
    \sum_{\rho:\beta-\eta\le \rho\le \beta:\rho n\in\mathbb{Z}} \binom{n}{n\frac{1-\rho}{2}} \le \exp_2\left(O\left(\frac{g(n)\sqrt{n\log_2 n}}{\phi(n)}\right)\right).
\end{equation}
We now return back to the earlier bound on the cardinality of the $m-$tuples with overlaps in $[\beta-\eta,\beta]$ as per \eqref{eq:superconst-ogp-card-up-bd}. Since $m=\omega_n(1)$, it holds $m-1=m(1+o_n(1))$. With this, we obtain 
\begin{align*}
2^n \left( \sum_{\rho:\beta-\eta\le \rho\le \beta:\rho n\in\mathbb{Z}} \binom{n}{n\frac{1-\rho}{2}}\right)^{m-1}  &\le \exp_2\left(n+(m-1)O\left(\frac{g(n)\sqrt{n\log_2 n}}{\phi(n)}\right)\right)\\
&\le \exp_2\left(n+O\left(\frac{mg(n)\sqrt{n\log_2 n}}{\phi(n)}\right)\right)\\
&=\exp_2\left(n+O\left(\frac{ng(n)}{\phi(n)^2}\right)\right).
\end{align*}
Consequently,
\begin{equation}\label{eq:cardinality-to-be-used-inMarkov}
|S(\beta,\eta,m)| \le \exp_2\left(n+O\left(\frac{ng(n)}{\phi(n)^2}\right)\right)\le \exp_2(n+o(n)),
\end{equation}
since $g(n) = o\left(\phi(n)^2\right)$. 

{\bf Step 2. Upper bounding the probability.} For the energy exponent $E_n$ defined earlier, suppose that $\mathcal{R}$ is the region
\[
\mathcal{R}=\left[-2^{-E_n},2^{-E_n}\right]\times \left[-2^{-E_n},2^{-E_n}\right]\times \cdots \times \left[-2^{-E_n},2^{-E_n}\right]\subset\R^m.
\]

Fix $\tau_1,\dots,\tau_m\in\mathcal{I}$, and fix any $m-$tuple, $(\sigma^{(i)}:1\le i\le m)\in S(\beta,\eta,m)$. Recall $Y_i(\tau_i)$, $1\le i\le m$ from Definition~\ref{def:overlap-set} and $Z_i = \frac{1}{\sqrt{n}}\ip{\sigma^{(i)}}{Y_i(\tau_i)}\distr\mathcal{N}(0,1)$, $1\le i\le m$. 

 Let $\Sigma$ denotes the covariance matrix of the (centered) vector $(Z_i:1\le i\le m)
 \in\R^m$.

The probability we want to upper bound is the following:
\[
\mathbb{P}\left((Z_i:1\le i\le m)\in \mathcal{R}\right)=(2\pi)^{-\frac{m}{2}}|\Sigma|^{-\frac12}\int_{\mathcal{R}\subset \R^m} \exp\left(-\frac12 x^T \Sigma^{-1}x\right)\;dx.
\]
Since provided $\left|\Sigma\right|\ne 0$, $\exp\left(-\frac12 x^T \Sigma^{-1}x\right)\le 1$ for any $x\in\R^m$; we upper bound the probability with
\begin{equation}\label{eq:crude-prob-up-bd}
\mathbb{P}\left((Z_i:1\le i\le m)\in\mathcal{R}\right) \le (2\pi)^{-\frac{m}{2}}|\Sigma|^{-\frac12} {\rm Vol}(\mathcal{R})  = 2^{\frac{m}{2}} \pi^{-\frac{m}{2}}|\Sigma|^{-\frac12}2^{-mE_n}.
\end{equation}
\paragraph{ Studying the covariance matrix, $\Sigma$.} The lines below are almost identical to Step II in the proof of Theorem~\ref{thm:main-eps-energy}; and kept for convenience. 

To control the probability in \eqref{eq:crude-prob-up-bd}, we study the covariance matrix $\Sigma$. In particular, our goal is to lower bound $|\Sigma|$ away from zero, uniformly for all choices of $(\sigma^{(i)}:1\le i\le m)\in S(\beta,\eta,m)$, and for every $\tau_1,\dots,\tau_m\in\mathcal{I}$. 

Using the exact same route as in Step II of the proof of Theorem~\ref{thm:main-eps-energy}, we arrive at
the conclusion that $\Sigma\in\R^{m\times m}$ has the following structure:
\[
\Sigma_{ii}=1,\quad\text{for}\quad 1\le i\le m;\quad\text{and}\quad  \Sigma_{ij}=\Sigma_{ji}=\gamma_i\gamma_j\bar{\Sigma}_{ij},\quad \text{for}\quad 1\le i<j\le m,
\]
where $\bar{\Sigma}_{ij}=\bar{\Sigma}_{ji} =\rho_{ij}=\frac1n\ip{\sigma^{(i)}}{\sigma^{(j)}}$, $1\le i<j\le m$. Here, $\gamma_i = \sqrt{1-\tau_i^2}$, $1\le i\le m$.

Namely, $\bar{\Sigma}\in\R^{m\times m}$ is an auxiliary matrix introduced for studying $\Sigma\in\R^{m\times m}$, and has the structure:
\[
\bar{\Sigma}_{ii}=1,\quad \text{for}\quad 1\le i\le m; \quad\text{and}\quad \bar{\Sigma}_{ij}=\bar{\Sigma}_{ji} = \rho_{ij},\quad \text{for}\quad 1\le i<j\le m.
\]
Now, let $A={\rm diag}(\gamma_1,\dots,\gamma_m)\in\R^{m\times m}$ be a diagonal matrix. It follows that
\begin{equation}\label{eq:original-cov-sigma}
\Sigma=A\bar{\Sigma}A + (I-A^2). 
\end{equation}
We next study $\bar{\Sigma}$. To that end, define the matrix $\widehat{\Sigma}\in\R^{m\times m}$, where $\widehat{\Sigma}_{ii}=1$ for $1\le i<j\le m$; and $\widehat{\Sigma}_{ij}=\widehat{\Sigma}_{ji}=\beta$ for $1\le i<j\le m$. Observe that
\[
\widehat{\Sigma} = (1-\beta)I_m + \beta 11^T.
\]
Now, the spectrum of the matrix $11^T\in \R^{m\times m}$ consists of the eigenvalue $m$ with multiplicity one; and eigenvalue $0$ with multiplicity $m-1$. Furthermore, since $\widehat{\Sigma}$ is obtained by applying a rank-1 perturbation to a multiple of identity matrix, its spectrum consists of the eigenvalue $1-\beta +\beta m$, that is, $1+\beta(m-1)$ with multiplicity one; and $1-\beta$ with multiplicity $m-1$. Since $1+\beta(m-1)$ and $1-\beta$ are both positive, this (symmetric) matrix is also positive definite. 

With this notation, we now express $\bar{\Sigma}$ of interest as
\[
\bar{\Sigma}=\widehat{\Sigma}+E,
\]
where the (symmetric) perturbation matrix $E\in\R^{m\times m}$ satisfies $E_{ii}=0$ for $1\le i\le m$, and $|E_{ij}|=|E_{ji}|\le \eta$ for $1\le i<j\le m$. We will bound the spectrum of $\bar{\Sigma}$ away from zero, using Wielandt-Hoffman inequality, Theorem~\ref{thm:wielandt-hoffman}. 
To that end, let $\lambda_1=1+(m-1)\beta\ge \lambda_2=\cdots=\lambda_m=1-\beta>0$ denotes the eigenvalues of $\widehat{\Sigma}$; and let $\mu_1\ge \mu_2\ge \cdots \ge \mu_m$ denotes the eigenvalues of $\bar{\Sigma}=\widehat{\Sigma}+E$. Then, Theorem~\ref{thm:wielandt-hoffman} yields
\[
\sum_{1\le j\le m}\left(\mu_j-\lambda_j\right)^2 \le \|E\|_F^2\le m(m-1)\eta^2<(m\eta)^2= \left(\frac{1-\beta}{2}\right)^2,
\]
where we use the facts $E_{ii}=0$, $|E_{ij}|\le \eta$ for the second inequality; and \eqref{eq:eta-1-beta-over2m} for the last  equality. With this, 
\[
\frac{1-\beta}{2}>|\mu_m-\lambda_m| \Rightarrow \mu_m>\frac{1-\beta}{2}=\nu_n,
\]
using \eqref{eq:beta:1-2nu-n}. Note that, this bound is uniform across all $(\sigma^{(i)}:1\le i\le m)\in S(\beta,\eta,m)$: no matter which $m-$tuple $(\sigma^{(i)}:1\le i\le m)\in S(\beta,\eta,m)$ is chosen, the determinant of the (induced) covariance matrix $\Sigma$ satisfies
$|\bar{\Sigma}|>\nu_n^m$. Consequently,
$
|\bar{\Sigma}|^{-\frac12}<\nu_n^{-\frac{m}{2}}$. Having controlled the determinant of $\bar{\Sigma}$, we now return back to the original covariance matrix $\Sigma$ as per \eqref{eq:original-cov-sigma}.
Note that, under the aforementioned choice of parameters, $\bar{\Sigma}$ is invertible. Furthermore, $\bar{\Sigma}$ is, by construction, a covariance matrix thus has non-negative eigenvalues $\lambda_1(\bar{\Sigma})\ge \cdots \ge \lambda_m(\bar{\Sigma})>0$ with
\[
m={\rm trace}(\bar{\Sigma}) \ge m\lambda_m(\bar{\Sigma}).
\]
Thus we have \begin{equation}\label{eq:min-lambda-bar-Sigma-le2}
    \lambda_m(\bar{\Sigma})\le 1.
\end{equation}
Now, observe that as $\bar{\Sigma}$ is positive semidefinite, so is $A\bar{\Sigma}A$, appearing in the equation \eqref{eq:original-cov-sigma}. Furthermore, as $1-\gamma_i^2\ge 0$ for $1\le i\le m$, the matrix $I-A^2$ is positive semidefinite as well. We now study the smallest eigenvalue $\lambda_m(\Sigma)$. 
\begin{lemma}
For any choices of $\tau_1,\dots,\tau_m\in\mathcal{I}$, it is the case that $\lambda_m(\Sigma)\ge \lambda_m (\bar{\Sigma})$. Hence,
\[
|\Sigma|\ge (\lambda_m (\bar{\Sigma}))^m>\nu_n^m>0,
\]
which is independent of the indices $\tau_1,\dots,\tau_m$. 
\end{lemma}
\begin{proof}
Recall the Courant-Fischer-Weyl variational characterization of the smallest singular value $\lambda_m(\Sigma)$ of a Hermitian matrix $\Sigma\in\R^{m\times m}$ \cite{horn2012matrix}:
\[
\lambda_m(\Sigma)= \inf_{v:\|v\|_2=1}v^T\Sigma v.
\]
Notice furthermore that for the matrix $\bar{\Sigma}\in\R^{m\times m}$, and any $v'\in\R^m$, we also have
\begin{equation}\label{eq:bar-sigma-lower-bd}
(v')^T \bar{\Sigma}v' \ge \lambda_m(\bar{\Sigma})\|v'\|_2^2.
\end{equation}
Now let $v\in\R^m$ have unit $\ell_2$ norm ($\|v\|_2=1$). We have,
\begin{align*}
v^T\Sigma v &= v^T(I-A^2)v +v^TA\bar{\Sigma}Av \\
&\ge \sum_{1\le i\le m}(1-\gamma_i^2)v_i^2 + \lambda_m(\bar{\Sigma})\|Av\|_2^2\\
&=\sum_{1\le i\le m}(1-\gamma_i^2+\lambda_m(\bar{\Sigma})\gamma_i^2)v_i^2\\
&\ge \lambda_m(\bar{\Sigma})\|v\|_2^2 \\&= \lambda_m(\bar{\Sigma}),
\end{align*}
where the first equality uses \eqref{eq:original-cov-sigma}, the first inequality uses \eqref{eq:bar-sigma-lower-bd}, and the last inequality uses $\lambda_m(\bar{\Sigma})\le 1$ as established in \eqref{eq:min-lambda-bar-Sigma-le2}. Since for any $v\in\R^m$ with unit norm we have $v^T\Sigma v \ge \lambda_m(\bar{\Sigma})$, we thus take the  infimum over all unit norm $v$ and conclude
\[
\lambda_m(\Sigma)\ge \lambda_m(\bar{\Sigma}).
\]
 Finally
\[
|\Sigma|=\prod_{1\le j\le m}\lambda_j(\Sigma)\ge \lambda_m(\Sigma)^m \ge \lambda_m(\bar{\Sigma})^m>\nu_n^m,
\]
as desired.
\end{proof}
As a consequence of this lemma, we obtain that
\[
|\Sigma|^{-\frac12}<\nu_n^{-\frac{m}{2}}
\]
uniformly for all $(\sigma^{(i)}:1\le i\le m)\in S(\beta,\eta,m)$, and every $\tau_1,\dots,\tau_m\in\mathcal{I}$; meaning that the upper bound depends only on the choice of $m$, and $\nu_n$ induced by the overlap value $\beta$ ($=1-2\nu_n$). 

Equipped with this, we now return to the probability upper bound \eqref{eq:crude-prob-up-bd}:
\begin{align*}
    \mathbb{P}\left((Z_i:1\le i\le m)\in\mathcal{R}\right) &\le 2^{\frac{m}{2}}\pi^{-\frac{m}{2}}|\Sigma|^{-\frac12}2^{-mE_n}\\
    &\le \exp_2\left(\frac{m}{2}-\frac{m}{2}\log_2\pi+\frac{m}{2}\log_2 \frac{1}{\nu_n}-mE_n\right).
\end{align*}
In particular, taking union bound over all choices of $\tau_1,\dots,\tau_m\in \mathcal{I}$, and recalling there are $n^{O(m)}=\exp_2(O(m\log_2 n))$ such choices, we have
\begin{align}
&\mathbb{E}\left[\ind  \left\{\exists \tau_1,\dots,\tau_m\in\mathcal{I}:\frac{1}{\sqrt{n}}|\ip{\sigma^{(i)}}{Y_i(\tau_i)}| \le 2^{-E_n},1\le i\le m\right\}\right]\\
&=\exp_2\left(O(m\log_2 n)+\frac{m}{2}-\frac{m}{2}\log_2\pi+\frac{m}{2}\log_2 \frac{1}{\nu_n}-mE_n\right)\label{eq:this-will-later-be-super-useful}.
\end{align}

We are now ready to upper bound the expectation.

{\bf Step 3. Upper bounding the expectation.}
Using the linearity of expectation, and the fact $O\left(ng(n)/\phi(n)^2\right)=o(n)$ following from \eqref{eq:g(n)-up-low-bd} we have
\begin{equation}\label{eq:a-very-loong-expression-for-the-firdt-momnet-m-ogp}
\mathbb{E}[N(\beta,\eta,m,E_n,\mathcal{I})]\le \exp_2\left(n+o(n)+O(m\log_2 n)+\frac{m}{2}-\frac{m}{2}\log_2\pi+\frac{m}{2}\log_2 \frac{1}{\nu_n}-mE_n\right).
\end{equation}
where we used the probability/expectation bound above, and the cardinality bound per \eqref{eq:cardinality-to-be-used-inMarkov}. Keep in mind that $mE_n=2n$ per \eqref{eq:m-chosen-this-way}. Thus
$n-mE_n = -n$.
Now, since $E_n=\omega_n(1)$, we simultaneously have
\[
\frac{m}{2},\frac{m}{2}\log_2 \pi =o(mE_n)=o(n).
\]
Since $E_n = \omega(\sqrt{n\log_2 n})$, we also have
\[
O(m\log_2 n) = o(mE_n)=o(n).
\]
Finally, we study the $\displaystyle\frac{m}{2}\log_2 \frac{1}{\nu_n}$ term. Recalling $\nu_n$ from \eqref{eq:nu-n-g-over-f}, we obtain
\begin{align*}
\log_2 \frac{1}{\nu_n}&=\left(\frac12\log_2 n+\frac12 \log_2\log_2 n +\log_2 \phi(n)-\log_2 g(n)\right)\\
&=O(\log_2 n),
\end{align*}
using \eqref{eq:phi(n)-up-low-bd} and \eqref{eq:g(n)-up-low-bd}. Applying now \eqref{eq:m-chosen-this-way} and the fact $\phi(n)=\omega_n(1)$, we obtain
\[
m\log_2 \frac{1}{\nu_n} = O\left(\frac{\sqrt{n\log_2 n}}{\phi(n)}\right)=o\left(\sqrt{n\log_2 n}\right)=o(n).
\]
Consequently,
\[
\mathbb{E}[N(\beta,\eta,m,E_n,\mathcal{I})]\le \exp_2\left(-n+o(n)\right)=\exp_2(-\Theta(n)).
\]


Finally, applying Markov's inequality, we conclude
\[
\mathbb{P}\left(\mathcal{S}(\beta,\eta,m,E_n)\ne\varnothing\right)\le \exp\left(-\Theta(n)\right),
\]
as claimed. This concludes the proof when
\[
\omega\left(\sqrt{n\log_2 n}\right)\le E_n\le o(n).
\]
\end{proof}
\subsubsection*{Case 2: The Special Case, $E_n = \omega\left(n\cdot \log^{-1/5+\epsilon} n\right)$.}
\begin{proof}
Let $E_n =  \omega\left(n\cdot \log^{-1/5+\epsilon} n\right)$ for an $\epsilon\in(0,\frac15)$. 
We set 
\[
g(n) \triangleq n \cdot \left(\frac{E_n}{n}\right)^{2+\frac{\epsilon}{8}}. 
\]
Take $m,\beta$, and $\eta$ per~\eqref{eq:beta-eta-m-n-main-supconstant}, that is
\[
m=\frac{2n}{E_n}, \quad \beta = 1-\frac{2g(n)}{E_n},\quad\text{and}\quad \eta = \frac{g(n)}{2n}.
\]
We next introduce several auxiliary quantities. Set $s(n) \triangleq E_n/n$, and $z(n) = s(n)^{1+\frac{\epsilon}{8}}$. In terms of $s(n)$ and $z(n)$; the parameters $g(n),m,\beta,\eta$ chosen as above satisfy now the relations:
\begin{equation}\label{eq-alt:g(n)-and-z(n)}
    g(n) = n\cdot s(n)\cdot z(n) \quad\text{where}\quad z(n)=s(n)^{1+\frac{\epsilon}{8}},
\end{equation}
\begin{equation}\label{eq-alt:m(n)-2over-s(n)}
    m =\frac{2n}{E_n} = \frac{2}{s(n)},
\end{equation}
and  
\begin{equation}\label{eq-alt:eta(n)-s(n)-z(n)}
   \beta = 1-2\frac{g(n)}{E_n} = 1-2z(n)\quad\text{and}\quad  \eta  = \frac{g(n)}{2n} = \frac{s(n)z(n)}{2}. 
\end{equation}
We also have
\begin{equation}\label{eq:asymp-exp-for-s-and-z}
\omega\left(\log^{-1/5+\epsilon}n \right)\le s(n) \le o(1).
\end{equation}
Moreover, analogous to previous case, define $\nu_n = g(n)/E_n$, which now becomes
\begin{equation}\label{eq:nu-n-required-for-ramsey}
\nu_n = \frac{g(n)}{E_n} = z(n) = s(n)^{1+\frac{\epsilon}{8}}.
\end{equation}
This is clearly $o_n(1)$ due to~\eqref{eq:asymp-exp-for-s-and-z}. Furthermore, the length of the interval $[n\beta-n\eta,n\beta]$ is $\Theta(n\eta)$ which is $\Theta(g(n)) =\omega_n(1)$. Hence, the interval is not vacuous. Moreover, in terms of this parameter, the expressions $\beta=1-2\nu_n$, $\eta=\nu_n/m$ and $\eta = \frac{1-\beta}{2m}$ are still valid.

Most of the steps of the proof remains the same as the previous case. Below, we only point out the necessary changes.
\paragraph{ Step 1. Cardinality upper Bound.} The expressions~\eqref{eq:superconst-ogp-card-up-bd} and~\eqref{eq:bin-coeff-max-up-bd} for the counting term remain the same. We now analyze the $n\frac{1-\beta+\eta}{2}$ term. Using~\eqref{eq-alt:eta(n)-s(n)-z(n)}, we have
\begin{align}
n\frac{1-\beta+\eta}{2} & = \frac{n}{2}\left(1-(1-2z(n)) +\frac{s(n)z(n)}{2}\right)\\
& = \frac{n}{2}\left(2z(n) + \frac{s(n)z(n)}{2}\right) \\
& = nz(n) \left(1+\frac{s(n)}{4}\right)\\
& = nz(n)\left(1+o_n(1)\right)\label{eq:we-gonna-use-this-later},
\end{align}  
where the last step uses~\eqref{eq:asymp-exp-for-s-and-z}. Since $z(n) = s(n)^{1+\frac{\epsilon}{8}} = o(1)$ as well, we obtain $n\frac{1-\beta+\eta}{2}$ to be $o(n)$. Hence, Lemma~\ref{lemma:binomial-k-o(n)} applies. Applying it with,
\[
k=n\frac{1-\beta+\eta}{2} = nz(n)(1+o_n(1)) = ns(n)^{1+\frac{\epsilon}{8}}(1+o_n(1)),
\]
the expression~\eqref{eq:n-choose-n-1-beta-eta} modifies to
\begin{align}
\binom{n}{n\frac{1-\beta+\eta}{2}} &= \exp_2\left(\left(1+o_n(1)\right)ns(n)^{1+\frac{\epsilon}{8}}\log \frac{1}{s(n)^{1+\frac{\epsilon}{8}}}\right) \\
& = \exp_2\left(O\left(ns(n)^{1+\frac{\epsilon}{8}}\log \frac{1}{s(n)}\right)\right)\label{eq:ramsey-exponent-bin-coeff}.
\end{align} 
Next,  $\Bigl |[n\beta-n\eta,n\beta]\cap \mathbb{Z}\Bigr| = O(n\eta) = O(g(n))$, and thus this term contributes to $O\left(\log_2 g(n)\right)$ in the exponent. Using $g(n) = ns(n)^{2+\frac{\epsilon}{8}}$ per~\eqref{eq-alt:g(n)-and-z(n)} as well as $s(n) \ge \omega\left(\log^{-\frac15+\epsilon} n \right)$ per~\eqref{eq:asymp-exp-for-s-and-z}, we find
\[
n\log^{-O(1)} n \le g(n)\le n \implies \log_2 g(n) = \Theta\left(\log_2 n\right). 
\]
Using $1/s(n) = \omega_n(1)$, we have
\begin{equation}\label{eq:ramsey-exp-card-intermediate-bd}
ns(n)^{1+\frac{\epsilon}{8}}\log \frac{1}{s(n)}  = \omega\left(ns(n)^{1+\frac{\epsilon}{8}}\right) = \omega\Bigl(\log_2 n\Bigr) = \omega\Bigl(\log_2 g(n) \Bigr).
\end{equation}
Consequently,
\begin{align}
\sum_{\rho:\beta-\eta\le \rho\le \beta,\rho n\in\mathbb{Z}}\binom{n}{n\frac{1-\rho}{2}} &\le O(g(n)) \cdot \binom{n}{n\frac{1-\beta+\eta}{2}}\label{eq:ramsey-up-bd-step33} \\
& \le \exp_2\left(O\left(ns(n)^{1+\frac{\epsilon}{8}}\log \frac{1}{s(n)}\right)\right)\label{eq:ramsey-up-bd-step3},
\end{align}
where the second step uses~\eqref{eq:ramsey-exponent-bin-coeff} and~\eqref{eq:ramsey-exp-card-intermediate-bd}. Next,
\begin{equation}\label{eq:temp-temp-temp-ramsey}
2^n \left( \sum_{\rho:\beta-\eta\le \rho\le \beta:\rho n\in\mathbb{Z}} \binom{n}{n\frac{1-\rho}{2}}\right)^{m-1} = \exp_2\left(n+(m-1) \log_2\left(\sum_{\rho:\beta-\eta\le \rho\le \beta,\rho n\in\mathbb{Z}}\binom{n}{n\frac{1-\rho}{2}} \right)\right).
\end{equation}
Applying~\eqref{eq:ramsey-up-bd-step3} in~\eqref{eq:temp-temp-temp-ramsey}, we obtain the following modification for the cardinality bound appearing in~\eqref{eq:cardinality-to-be-used-inMarkov} :
\begin{equation}\label{eq:modified-cardinality-bound}
\left|S(\beta,\eta,m) \right| \le  \exp_2\left(n+O\left(mns(n)^{1+\frac{\epsilon}{8}}\log \frac{1}{s(n)}\right)\right).
\end{equation}
\paragraph{ Step 2. Upper bounding the probability.} The entire analysis of the probabilistic term remains intact. In particular, \eqref{eq:this-will-later-be-super-useful} remains valid. (The $\nu_n$ term appearing in~\eqref{eq:this-will-later-be-super-useful} is now given by~\eqref{eq:nu-n-required-for-ramsey}.)

We are now ready to upper bound the first moment.

\paragraph{ Step 3. Upper bounding the expectation.}
Recalling~\eqref{eq:asymp-exp-for-s-and-z} and~\eqref{eq:nu-n-required-for-ramsey}, we obtain
\[
\log\frac{1}{\nu_n} = \Theta\left(\log\frac{1}{s(n)} \right) = O\left(\log \log n\right) =o\left(E_n\right). 
\]
Consequently, the term $\frac{m}{2}\log_2 \frac{1}{\nu_n}$ appearing in~\eqref{eq:this-will-later-be-super-useful} is $o(mE_n)$. The remaining terms, $O(m\log_2 n)$, $\frac{m}{2}$, and $\frac{m}{2}\log_2 \pi$, are still $o(mE_n)$ as in the first case, since $E_n = \omega(1)$. Hence,~\eqref{eq:this-will-later-be-super-useful} becomes
\begin{equation}\label{eq:modified-pb-bd}
\exp_2\left(O(m\log_2 n)+\frac{m}{2}-\frac{m}{2}\log_2\pi+\frac{m}{2}\log_2 \frac{1}{\nu_n}-mE_n\right) = \exp_2\left(-mE_n+o\left(mE_n\right)\right).
\end{equation} 
After incorporating the modified cardinality bound~\eqref{eq:modified-cardinality-bound} into the probability bound~\eqref{eq:modified-pb-bd}, the expression~\eqref{eq:a-very-loong-expression-for-the-firdt-momnet-m-ogp} for the first moment now becomes \begin{equation}\label{eq2:a-very-loong-expression-for-the-firdt-momnet-m-ogp}
\mathbb{E}[N(\beta,\eta,m,E_n,\mathcal{I})]\le \exp_2\left(n+O\left(mns(n)^{1+\frac{\epsilon}{8}}\log \frac{1}{s(n)}\right)-mE_n+o(mE_n)\right).
\end{equation}
Recalling $m=2n/E_n$, this bound is
\begin{equation}\label{eq:first-mom-bd-ramsey-finale}
\exp_2\left(-n+O\left(mns(n)^{1+\frac{\epsilon}{8}}\log \frac{1}{s(n)}\right) + o(n)\right).
\end{equation}
To finish the proof, that is to establish $\mathbb{E}[N(\beta,\eta,m,E_n,\mathcal{I})]\le \exp(-\Theta(n))$, it suffices to verify
\[
O\left(mns(n)^{1+\frac{\epsilon}{8}}\log \frac{1}{s(n)}\right) = o(n). 
\]
Recall by~\eqref{eq-alt:m(n)-2over-s(n)} that $m=2s(n)^{-1}$. Hence
\[
mns(n)^{1+\frac{\epsilon}{8}}\log \frac{1}{s(n)}  = n\left(2s(n)^{\frac{\epsilon}{8}} \log \frac{1}{s(n)} \right) =o(n),
\]
where we used the fact that $s(n) = o_n(1)$ per~\eqref{eq-alt:eta(n)-s(n)-z(n)} and thus
\[
2s(n)^{\frac{\epsilon}{8}} \log \frac{1}{s(n)} = o_n(1), \qquad \forall \epsilon>0.
\]
Hence, the expression in~\eqref{eq:first-mom-bd-ramsey-finale} is indeed $\exp_2\left(-\Theta(n)\right)$. This concludes the proof when
\[
\omega\left(n \cdot \log^{-\frac15+\epsilon} n\right)\le E_n\le o(n).
\]

\end{proof}

\subsection{Proof of Theorem~\ref{thm:local-opt}}\label{sec:pf-thm-local-opt}
We first have
\[
N_\epsilon = \sum_{\sigma\in\{\pm 1\}^n} \ind\left\{n^{-\frac12}|\ip{\sigma}{X}| = O(2^{-n\epsilon}),\sigma \text{ is locally optimal}\right\}.
\]
Hence,
\[
\E{N_\epsilon} = 2^n \mathbb{P}\left(n^{-\frac12}|\ip{\sigma}{X}| = O(2^{-n\epsilon}),\sigma \text{  is locally optimal}\right).
\]
To start with, notice $n^{-\frac12}\ip{\sigma}{X} \distr \mathcal{N}(0,1)$, thus
\begin{align*}
    \mathbb{P}\left(n^{-\frac12}|\ip{\sigma}{X}| = O(2^{-n\epsilon}),\sigma \text{  is locally optimal}\right)&\le \mathbb{P}\left(n^{-\frac12}|\ip{\sigma}{X}| = O(2^{-n\epsilon})\right) \\
    &\le C2^{-n\epsilon},
\end{align*}
where $C>0$ is some absolute constant. Hence,
\begin{equation}\label{eq:limsup}
\E{N_\epsilon}\le C2^{n(1-\epsilon)}\Rightarrow
\limsup_{n\to\infty}\frac1n \log_2 \E{N_\epsilon}\le 1-\epsilon.
\end{equation}
We now investigate the lower bound. To that end, let $Y_i\triangleq \sigma_i X_i$. Note that $Y_i$, $1\le i\le n$, is a collection of i.i.d. standard normal random variables. Now, local optimality of $\sigma$ per Definition~\ref{def:loc-opt}, namely, $\ip{\sigma^{(i)}}{X}^2 \ge \ip{\sigma}{X}^2$ for $1\le i\le n$ is equivalent to
\[
\sum_{j:1\le j\le n,j\ne i}Y_iY_j \le 0
\]
for $1\le i\le n$. With this, we arrive at
\[
\left\{n^{-\frac12}|\ip{\sigma}{X}| = O(2^{-n\epsilon}),\sigma\text{ is locally optimal}\right\} = \bigcap_{1\le i\le n}\left\{\sum_{1\le j\le n:j\ne i}Y_iY_j \le 0\right\}\cap \left\{\left|\frac{1}{\sqrt{n}}\sum_{1\le j\le n}Y_j \right|\le  2^{-n\epsilon}\right\},
\]
where we ignored the constant hidden under $O(2^{-n\epsilon})$ for convenience.

Observe now the following union of events, that are disjoint up to a measure zero set:
\[
\left\{\left|\frac{1}{\sqrt{n}}\sum_{1\le j\le n}Y_j \right|\le  2^{-n\epsilon}\right\}=\left\{-2^{-n\epsilon}\le \frac{1}{\sqrt{n}}\sum_{1\le j\le n}Y_j \le 0\right\}\cup \left\{0\le \frac{1}{\sqrt{n}}\sum_{1\le j\le n}Y_j \le  2^{-n\epsilon}\right\}.
\]
This brings us
\begin{align*}
    &\left\{n^{-\frac12}|\ip{\sigma}{X}| = O(2^{-n\epsilon}),\sigma\text{ is locally optimal}\right\}\\ &=\underbrace{\left(\bigcap_{1\le i\le n}\left\{\sum_{1\le j\le n:j\ne i}Y_iY_j \le 0\right\}\cap\left\{ \frac{1}{\sqrt{n}}\sum_{1\le j\le n}Y_j \in[-2^{-n\epsilon},0]\right\}\right)}_{\triangleq \mathcal{E}_1}\\
    &\bigcup \underbrace{\left(\bigcap_{1\le i\le n}\left\{\sum_{1\le j\le n:j\ne i}Y_iY_j \le 0\right\}\cap\left\{ \frac{1}{\sqrt{n}}\sum_{1\le j\le n}Y_j \in[0,2^{-n\epsilon}]\right\}\right)}_{\triangleq \mathcal{E}_2}.
\end{align*}
Note that  $(Y_1,\dots,Y_n)\distr (-Y_1,\dots,-Y_n)$. From here $\mathbb{P}(\mathcal{E}_1)=\mathbb{P}(\mathcal{E}_2)$. Furthermore, the events $\mathcal{E}_1$ and $\mathcal{E}_2$ are disjoint, up to a set of measure zero; thus $\mathbb{P}(\mathcal{E}_1\cup\mathcal{E}_2) = 2\mathbb{P}(\mathcal{E}_2)$. 

We now compute $\mathbb{P}(\mathcal{E}_2)$. For convenience set $S\triangleq \sum_{1\le j\le n}Y_j$. Then the condition $\sum_{1\le j\le n,j\ne i}Y_iY_j\le 0$ is equivalent to $Y_i(Y_i-S)\ge 0$. Namely,
\[
\bigcap_{1\le i\le n}\left\{\sum_{1\le j\le n:j\ne i}Y_iY_j\le 0\right\}\cap \left\{n^{-\frac12}S\in[0,2^{-n\epsilon}]\right\} = \bigcap_{1\le i\le n}\left\{Y_i\notin[0,S]\right\}\cap \left\{n^{-\frac12}S\in[0,2^{-n\epsilon}]\right\}.
\]
Define now the auxiliary variables
\[
\bar{Y_i}\triangleq Y_i-\frac1n S,\quad 1\le i\le n.
\]
Clearly, $(\bar{Y}_1,\dots,\bar{Y}_n)$ and $S$ are jointly normal. Notice, furthermore, that for any fixed $1\le i\le n$, $\E{\bar{Y_i}S} = 0=\E{\bar{Y_i}}=\E{S}$. This yields $(\bar{Y}_1,\dots,\bar{Y}_n)$ and $S$ are independent. Furthermore, $n^{-\frac12}S\distr\mathcal{N}(0,1)$.

With these, we obtain
\begin{align}
\mathbb{P}(\mathcal{E}_2) &= \mathbb{P}\left(\bigcap_{1\le i\le n}\left\{\bar{Y}_i\notin \left[-\frac1n S,\frac{n-1}{n}S\right]\right\}\cap \left\{n^{-\frac12}S\in[0,2^{-n\epsilon}]\right\}\right)\label{eq:event-e2}\\
&=\int_{z\in[0,2^{-n\epsilon}]}\mathbb{P}\left(\bigcap_{1\le i\le n}\left\{\bar{Y}_i\notin \left[-\frac1n S,\frac{n-1}{n}S\right]\right\}\given[\Big]n^{-\frac12}S =z\right)\varphi(z)\; dz\label{eq:law-of-total-prob}\\
&=\int_{z\in[0,2^{-n\epsilon}]}\mathbb{P}\left(\bigcap_{1\le i\le n}\left\{\bar{Y}_i\notin \left[-\frac{1}{\sqrt{n}} z,\frac{n-1}{\sqrt{n}}z\right]\right\}\right)\varphi(z)\; dz\label{eq:used-independence}\\
&=\int_{z\in[0,2^{-n\epsilon}]}\left(1-\mathbb{P}\left(\bigcup_{1\le i\le n}\left\{\bar{Y}_i\in \left[-\frac{1}{\sqrt{n}} z,\frac{n-1}{\sqrt{n}}z\right]\right\}\right)\right)\varphi(z)\; dz \label{eq:took-complement}\\
&\ge \int_{z\in[0,2^{-n\epsilon}]}\left(1-n\mathbb{P}\left(\bar{Y}_1\in \left[-\frac{1}{\sqrt{n}} z,\frac{n-1}{\sqrt{n}}z\right]\right)\right)\varphi(z)\; dz\label{eq:used-union-bd}\\
&\ge \int_{z\in[0,2^{-n\epsilon}]}\left(1-n\mathbb{P}\left(\mathcal{N}(0,1)\in \left[-\frac{1}{\sqrt{n-1}} 2^{-n\epsilon},\sqrt{n-1}\cdot 2^{-n\epsilon}\right]\right)\right)\varphi(z)\; dz\label{eq:adjusted-distr-of-bar-Y_1}\\
&\ge \int_{z\in[0,2^{-n\epsilon}]}\left(1-\frac{n^2}{\sqrt{n-1}}2^{-n\epsilon}\right)\varphi(z)\; dz\label{eq:trivially-upper-bd-normal}\\
&=\left(1-\frac{n^2}{\sqrt{n-1}}2^{-n\epsilon}\right)(2\pi)^{-\frac12}2^{-n\epsilon}(1+o_n(1))\label{eq:this-is-the-final-line}\\
&=(2\pi)^{-\frac12}(1+o_n(1))2^{-n\epsilon}\label{eq:really-final}.
\end{align}
We now justify each of these lines, where $\varphi(z)\triangleq (2\pi)^{-\frac12}\exp(-z^2/2)$ being the  standard normal density. \eqref{eq:event-e2} is the definition of $\mathcal{E}_2$; \eqref{eq:law-of-total-prob} follows from the law of total probability; \eqref{eq:used-independence} uses the fact that the random vector $(\bar{Y}_i:1\le i\le n)$ and $S$ are independent; \eqref{eq:took-complement} uses De Morgan's law; \eqref{eq:used-union-bd} uses union bound; \eqref{eq:adjusted-distr-of-bar-Y_1} uses the fact that $\bar{Y}_1 \distr \mathcal{N}(0,\frac{n-1}{n})$ where $\mathcal{N}(0,1)$ is standard normal; \eqref{eq:trivially-upper-bd-normal} uses the trivial upper bound $\mathbb{P}(\mathcal{N}(0,1)\in I) \le |I|$ for any interval $I$;  \eqref{eq:this-is-the-final-line} uses the fact that in the interval $[0,2^{-n\epsilon}]$, $\varphi(z) =(2\pi)^{-\frac12}(1+o_n(1))$; and finally $\eqref{eq:really-final}$ uses the fact $1-n^2(n-1)^{-\frac12}2^{-n\epsilon}=1+o_n(1)$.
 
Therefore, using \eqref{eq:really-final} we arrive at
\[
\mathbb{P}\left(n^{-\frac12}|\ip{\sigma}{X}| = O(2^{-n\epsilon}),\sigma\text{ is locally optimal}\right)=\mathbb{P}(\mathcal{E}_1\cup\mathcal{E}_2) = 2\mathbb{P}(\mathcal{E}_2) \ge 2(2\pi)^{-\frac12}(1+o_n(1))2^{-n\epsilon}.
\]
With this, we conclude,
\[
\mathbb{E}[N_\epsilon]=2^n\mathbb{P}\left(n^{-\frac12}|\ip{\sigma}{X}| = O(2^{-n\epsilon}),\sigma\text{ is locally optimal}\right) \ge 2(2\pi)^{-\frac12}(1+o_n(1))2^{n(1-\epsilon)}.
\]
Thus,
\begin{equation}\label{eq:liminf}
    \liminf_{n\to\infty}\frac1n \log_2 \E{N_\epsilon} \ge 1-\epsilon.
\end{equation}
Finally, we arrive at the desired conclusion by combining \eqref{eq:limsup} and \eqref{eq:liminf}:
\[
\lim_{n\to\infty}\frac1n \log_2 \E{N_\epsilon}=1-\epsilon.
\]
\subsection{Proof of Theorem~\ref{thm:Main}}\label{sec:pf-thm-Main}
The proof of Theorem~\ref{sec:pf-thm-Main} uses several interesting ideas. In order to present them in a coherent way, we first provide an informal outline sketching the proof.
\subsubsection*{Outline of the Proof of Theorem~\ref{thm:Main}}
Fix $E_n$ (with prescribed growth condition) corresponding to the exponent of energy level $2^{-E_n}$ we want to rule out. We use the $m-$OGP property established in Theorem~\ref{thm:m-ogp-superconstant-m}. Specifically, let $m\in\mathbb{N}$ and  $1>\beta>\eta>0$ be the parameters prescribed by Theorem~\ref{thm:m-ogp-superconstant-m} for this choice of $E_n$.
\begin{itemize}
    \item We first reduce the proof to the case of deterministic algorithms. That is, instead of considering $\A:\R^n\times \Omega\to\bincube$, we find a $\omega^*\in\Omega$, set $\A(\cdot)\triangleq\A(\cdot,\omega^*)$; and consider instead this deterministic choice $\A:\R^n\to\bincube$ in the remainder.
    \item We then study a certain high-probability event. This event will establish that for any $m-$tuple $\left(\sigma^{(i)}:1\le i\le m\right)$ of spin configurations that are near-optimal with respect to \emph{independent instances} $X_i\distr \mathcal{N}(0,I_n)$, $1\le i\le m$, there is a pair $1\le i<j\le m$ such that $\OBar\left(\sigma^{(i)},\sigma^{(j)}\right)$ is contained in an interval of form $[0,1-\eta']$, which is below $[\beta-\eta,\beta]$ interval prescribed by the OGP result, Theorem~\ref{thm:m-ogp-superconstant-m}.
    \item We then set $X_0\distr \mathcal{N}(0,I_n)$ and generate $T$ ``replicas" $X_i\in\R^n$, i.i.d. random vectors each with distribution $\mathcal{N}(0,I_n)$. We then divide $[0,1]$ into $Q$ equal pieces via $0=\tau_0<\tau_1<\cdots<\tau_Q=1$; and interpolate, for each $1\le i\le T$, between $X_0$ and $X_i$ in the  following way:
    \[
    Y_i(\tau_k) =\sqrt{1-\tau_k^2}X_0+\tau_k X_i,\quad 1\le i\le T,\, 0\le k\le Q.
    \]
    The numbers $T$ and $Q$ will be tuned appropriately. 
    \item We next establish that the pairwise ``overlaps" are ``stable" along each interpolation  trajectory: for $1\le i<j\le T$ and $0\le k\le Q-1$, we show
    \[
    \left|\OBar^{(ij)}\left(\tau_k\right)-\OBar^{(ij)}\left(\tau_{k+1}\right)\right|\quad \text{is small}.
    \]
    \item We then use the guarantee on the probability of the success of the algorithm to arrive at a guarantee that the algorithm will produce, for each interpolation trajectory $1\le i\le T$ and time instance $0\le k\le Q$ (that is for $Y_i(\tau_k)\in\R^n$), a solution that is near ground-state: the solution $\A\left(Y_i(\tau_k)\right)\in\bincube$ generated achieves an objective value $2^{-E_n}$ for \eqref{eq:NPP-main}. That is,
    \[
    \frac{1}{\sqrt{n}}\left|\ip{Y_i(\tau_k)}{\A\left(Y_i(\tau_k)\right)}\right|=2^{-E_n}.
    \]
    \item We then take a  union bound over all subsets $S=\{i_1,\dots,i_m\}$ of $[T]$ of cardinality $|S|=m$ to extend the previous high-probability event (the event pertaining the spin configurations that are near ground with respect to independent instances)  when the indices come from the set $S$. Now, $\tau=0$ in the beginning of the interpolation. Thus, it is the case that for every $1\le i<j\le T$,
    \[
    Y_i(\tau) = Y_j(\tau)\implies \A(Y_i(\tau))=\A(Y_i(\tau)),
    \]
    when $\tau=0$. On the other hand, due to the previous property applied to this subset $\mathcal{S}$, there exists indices $1\le k<\ell \le m$ such that the overlap between $i_k$ and $i_\ell$ is eventually below $\beta-\eta$. Since the overlaps are stable, that is, they do not change \emph{abruptly}, this implies that there is a time $\tau$ such that the overlap between $\A(Y_{i_k}(\tau))$ and $\A(Y_{i_\ell}(\tau))$ is contained in $(\beta-\eta,\beta)$. 
    \item Equipped with this, we then construct a certain graph $\Gr=(V,E)$. Specifically, we let $|V|=T$ where each vertex corresponds to a replica (i.e., an interpolation trajectory); and for $1\le i<j\le m$, we let $(i,j)\in E$ if there is a time $\tau$ such that the overlap between $Y_i(\tau)$ and $Y_j(\tau)$ is contained in $(\beta-\eta,\beta)$. Moreover, each edge $(i,j)$ is \emph{colored} with one of $Q$ different colors: color the edge $(i,j)\in E$ with color $1\le t\le Q$ if $\tau_t$ is the first time such that the overlap between $\A\left(Y_i(\tau_t)\right)$ and $\A\left(Y_j(\tau_t)\right)$ is contained in $(\beta-\eta,\beta)$. With this, the graph has following properties. For every subset $S\subset V$ of $|S|=m$ vertices, there exists $1\le i_S<j_S\le T$ such that $(i_S,j_S)\in E$. Moreover, each edge of this graph is colored with one of $Q$ colors. We then establish, using tools from the extremal graph theory  and the Ramsey theory, that $\Gr$ contains a monochromatic $m-$clique provided that $T$ is sufficiently large.
    \item Finally, if $\Gr$ contains a monochromatic $m-$clique, this means there exists indices $1\le i_1<i_2<\cdots<i_m\le T$ and a time $\tau'\in\{\tau_1,\dots,\tau_Q\}$ such that the overlap between $\A\left(Y_{i_k}\left(\tau'\right)\right)$ and $\A\left(Y_{i_\ell}\left(\tau'\right)\right)$ is contained in $(\beta-\eta,\beta)$ for any $1\le k<\ell \le m$. Setting 
    \[
    \sigma^{(k)}\triangleq \A\left(Y_{i_k}\left(\tau'\right)\right), \quad 1\le k\le m,
    \]
    we then  deduce the $m-$tuple $\left(\sigma^{(k)}:1\le k\le m\right)$ of near ground-state spin configurations $\sigma^{(k)}\in\bincube$ violates the $m-$OGP established in Theorem~\ref{thm:m-ogp-superconstant-m}. This will conclude the proof.
\end{itemize}
Before we formally start proving Theorem~\ref{thm:Main}, we state several auxiliary results.
\subsubsection*{Auxiliary Results from Ramsey Theory and Extremal Graph Theory}
Our first auxiliary result pertains the so-called two-color Ramsey numbers.
\begin{theorem}\label{thm:ramsey-1}
Let $k,\ell\ge 2$ be integers; and $R(k,\ell)$ denotes the smallest $n\in\mathbb{N}$ such that any red/blue (edge) coloring of $K_n$ contains either a red $K_k$ or a blue $K_\ell$. Then
\begin{equation}\label{eq:standard-erdos-szekeres}
R(k,\ell) \le \binom{k+\ell-2}{k-1} = \binom{k+\ell-2}{\ell-1}.
\end{equation}
\end{theorem} 
\begin{proof}
To that end, we show $R(k,\ell)$ exists for any $k,\ell\in\mathbb{N}$; and moreover for $k,\ell\ge 2$, it holds that
\begin{equation}\label{eq:ES-recursion}
R(k,\ell)\le R(k,\ell-1)+R(k-1,\ell).
\end{equation}
The elegant argument below is due to Erd{\"o}s and Szekeres~\cite{erdos1935combinatorial} and is reproduced herein for completeness. The argument  is via induction on $k+\ell$. The base case is clear. Suppose for every $i,j$ with $i+j\le n-1$ the numbers $R(i,j)$ exist. Now, we consider $R(k,\ell)$ for $k+\ell =n$, $k,\ell \ge 2$. By inductive hypothesis, both $R(k-1,\ell)$ and $R(k,\ell-1)$ exists. Now, let $m\triangleq R(k-1,\ell)+R(k,\ell-1)$, and consider any red/blue (edge) coloring of $K_m$. For any vertex $v\in K_m$, either (a) $v$ is adjacent to at least $R(k-1,\ell)$ vertices through a red edge; or (b) $v$ is adjacent to at least  $R(k,\ell-1)$ vertices through a blue edge. Assume case (a). By inductive hypothesis, any $R(k-1,\ell)$ such neighbors of $v$ contains either a red $K_{k-1}$ or a blue $K_\ell$. Adding $v$, the resulting graph indeed has either a red $K_\ell$ or a blue $K_\ell$. The case (b) is handled similarly. This establishes \eqref{eq:ES-recursion}.

\eqref{eq:standard-erdos-szekeres} now follows from \eqref{eq:ES-recursion} again by induction on $k+\ell$. The base cases are verified easily. Assume $k,\ell\ge 3$. Then by inductive hypothesis
\[
R(k-1,\ell) \le \binom{k+\ell-3}{k-2} \quad\text{and}\quad R(k,\ell-1)\le \binom{k+\ell-3}{k-1}.
\] 
Thus,
\[
R(k,\ell) \le R(k-1,\ell)+R(k,\ell-1) \le \binom{k+\ell-3}{k-2} + \binom{k+\ell-3}{k-1} = \binom{k+\ell-2}{k-1}.
\]
\end{proof}
The second result pertains the so-called multicolor Ramsey numbers.
\begin{theorem}\label{thm:ramsey}
Let $q,m\in\mathbb{N}$. Denote by $R_q(m)$ the smallest $n\in\mathbb{N}$ for which any $q-$coloring of the edges of $K_n$ necessarily contains a monochromatic $K_m$. Then
\begin{equation}\label{eq:q-color-ramzi-up-bd}
R_q(m)\le q^{qm}.
\end{equation} 
\end{theorem} 
Theorem~\ref{thm:ramsey} can be shown using a minor modification of the neighborhood-chasing argument given by Erd{\"o}s and Szekeres~\cite{erdos1935combinatorial}. See~\cite[Page~6]{conlon2015recent} for more information. 

We next define a certain graph property.
\begin{definition}\label{def:m-admissible-graph}
Fix a positive integer $M\in\mathbb{N}$. A graph  $\mathbb{G}=(V,E)$ is called \emph{$M-$admissible} if for any $S\subset V$, $|S|=M$; there exists distinct $i,j\in S$ such that $(i,j)\in E$. 
\end{definition}
Namely, $\mathbb{G}$ is $M-$admissible if $\alpha(\mathbb{G})\le M-1$, where $\alpha(\mathbb{G})$ is the independence number of $\mathbb{G}$. 

We now state and prove our second auxiliary result, an extremal graph theory result.
\begin{proposition}\label{thm:clique-exist}
Let $M\in\mathbb{N}$. Any $M-$admissible graph $\Gr=(V,E)$ with 
\[
|V|\ge \binom{2M-2}{M-1}
\]
contains an $M-$clique. 
\end{proposition}
\subsubsection*{Proof of Proposition~\ref{thm:clique-exist}}
\begin{proof}
Let $\Gr$ be an $M-$admissible graph on $|V|\ge \binom{2M-2}{M-1}$ vertices. Theorem~\ref{thm:ramsey-1} then yields that $|V|\ge R(M,M)$. Now, for any $i,j\in V$; we say $(i,j)$ is colored ``red" if $(i,j)\in E$; and $(i,j)$ is colored ``blue" otherwise. Due to the Ramsey property, $\Gr$ contains either a red $K_M$ or a blue $K_M$; that is, $\Gr$ contains either a clique of size $M$ or an independent set of size $M$. But since $\Gr$ is $M-$admissible, $\alpha(\G)\le M-1$. Thus the latter is not the case. Hence $\Gr$ contains a $K_M$. 
\end{proof}
We are now ready to start formally proving Theorem~\ref{thm:Main}.
\subsubsection*{Proof of Theorem~\ref{thm:Main}}
\begin{proof}
In what follows, recall the notation that for $\sigma,\sigma'\in\bincube$;
\[
\OBar\left(\sigma,\sigma'\right) = \frac1n\ip{\sigma}{\sigma'} = \frac1n \sum_{1\le i\le n}\sigma_i \sigma'_i.
\]
Recall also that all floor/ceiling signs are omitted for the sake of a clear presentation.

Let $L>0$ be fixed (which is constant in $n$); and $\exp_2\left(-E_n\right)$ be the target energy level whose ``exponent" $E_n$ satisfies, for some $\epsilon\in(0,\frac15)$,
\begin{equation}\label{eq:E-n-satisfies-this}
   \omega\left(n\cdot \log^{-\frac15+\epsilon} n\right)\le E_n \le o(n). 
\end{equation}
In what follows, we choose
\begin{equation} \label{eq:chooise-c1c2-main-pf}
c_1 = \frac{1}{6400} \quad\text{and}\quad c_2 =8\cdot 480^2,
\end{equation}
and establish that there exists no randomized algorithm $\mathcal{A}:\R^n\times \Omega\to \bincube$ that is $\left(2^{-E_n}, f,L,\rho',p_f,p_{\rm st}\right)$-optimal for every sufficiently large $n$, where the parameter $f$ is specified in Theorem~\ref{thm:Main}, as
\begin{equation}\label{eq:main-proof-f-param}
f = \frac{1}{6400} \cdot n \cdot \left(\frac{E_n}{n}\right)^{4+\frac{\epsilon}{4}};
\end{equation}
and the parameters $\rho',p_f,p_{\rm st}$ are given by~\eqref{eq:main-choice-of-param} with $c_2$ chosen as in~\eqref{eq:chooise-c1c2-main-pf} (we supress the dependence of $\rho',p_f$ and $p_{\rm st}$ on $n$ and $c_2$ for convenience). The proof is by a contradiction argument. 
\paragraph{ Choice of Auxiliary Parameters.}

We choose parameters $m,\beta,\eta$ (all functions of $n$) as in the second part of Theorem~\ref{thm:m-ogp-superconstant-m}, which we recall for convenience:
\begin{equation}\label{eq:m(n)-2over-s(n)}
    m \triangleq m(n)=\frac{2n}{E_n}, 
\end{equation}
\begin{equation}\label{eq:g(n)-and-z(n)}
    g(n) = n\cdot \left(\frac{E_n}{n}\right)^{2+\frac{\epsilon}{8}};
\end{equation}
and  
\begin{equation}\label{eq:eta(n)-s(n)-z(n)}
   \beta\triangleq \beta(n) = 1-2\frac{g(n)}{E_n}
   \quad\text{and}\quad  \eta \triangleq \eta(n) = \frac{g(n)}{2n}. 
\end{equation}
We now establish certain convenient expression for the parameters $f,\rho',p_f,p_{\rm st}$ in terms of the quantities $m,\beta,\eta$ above. For $f$ chosen per~\eqref{eq:main-proof-f-param}, define
\[
C_1 \triangleq \frac{f}{n} = \frac{1}{6400} \left(\frac{E_n}{n}\right)^{4+\frac{\epsilon}{4}} .
\]
Using~\eqref{eq:g(n)-and-z(n)} and~\eqref{eq:eta(n)-s(n)-z(n)}, it follows that
\begin{equation}\label{eq:C-1-chosen-this-way}
    C_1 = \frac{g(n)^2}{6400\cdot n^2} = \frac{\eta^2}{1600}.
\end{equation}
Define next
\begin{equation}\label{eq:Q-chosen-this-way}
    Q=\frac{480^2\cdot 2L}{\eta^2}. 
\end{equation}
Using~\eqref{eq:g(n)-and-z(n)}, \eqref{eq:eta(n)-s(n)-z(n)}, and \eqref{eq:Q-chosen-this-way}, it follows that
\begin{equation}\label{eq:choose-Q-thisway}
Q=\frac{2\cdot 480^2 \cdot L}{\eta^2} = \frac{2\cdot 480^2 \cdot L}{(g(n)/2n)^2} = \left(8\cdot 480^2\cdot L\right)\cdot \left(\frac{n}{E_n}\right)^{4+\frac{\epsilon}{4}}.
\end{equation} 
In particular, with $c_2 =8\cdot 480^2$ as above, the parameter $T(c_2)$ defined per~\eqref{eq:main-choice-of-T} becomes
\begin{equation}\label{eq:T-chosen-this-way}
   T(c_2) \triangleq T=\exp_2\Bigl(2^{4mQ\log_2 Q}\Bigr)
\end{equation}
for $m$ chosen in \eqref{eq:m(n)-2over-s(n)}  and $Q$ chosen in \eqref{eq:Q-chosen-this-way}. 

Moreover, for $p_f$ and $p_{\rm st}$ chosen per~\eqref{eq:main-choice-of-param}, it holds that
\begin{equation}\label{eq:p_f-chosen-this-way}
    p_f =\frac{1}{4T(Q+1)}
\end{equation}
and 
\begin{equation}\label{eq:p-rm-st}
    p_{\rm st} = \frac{1}{9TQ^2}.
\end{equation}
Define next the function
\[
\Psi(x) \triangleq  \sqrt{\left(1-\frac{x^2}{Q^2}\right)\left(1-\frac{(x+1)^2}{Q^2}\right)} + \frac{x(x+1)}{Q^2}, \quad 0\le x\le Q-1.
\]
We show $\Psi(\cdot)$ is decreasing on $[0,Q-1]$. For this, it suffices to verify $\Psi'(x)\le 0$ for $0\le x\le Q-1$. We have
\[
\Psi'(x) = \frac{1}{Q^2} \left(2x+1 - \frac{x \left(Q^2-(x+1)^2\right)+(x+1)\left(Q^2-x^2\right)}{\sqrt{\left(Q^2-x^2\right)\left(Q^2-(x+1)^2\right)}} \right)
\]
Now set $u\triangleq Q^2-(x+1)^2$, $\lambda\triangleq \frac{x}{2x+1}$; and $\bar{\lambda}=1-\lambda = \frac{x+1}{2x+1}$ (while suppressing $x$ dependence). Clearly $u\ge 0$ as $0\le x\le Q-1$ and $\bar{\lambda}>1/2$. It then boils down verifying
\[
\Psi'(x)\le 0 
\Leftrightarrow  \sqrt{u(u+2x+1)} \le \lambda u+\bar{\lambda}(u+2x+1).
\]
Applying the weighted AM-GM inequality, we find
\[
\lambda u+\bar{\lambda}(u+2x+1) \ge u^\lambda \cdot (u+2x+1)^{\bar{\lambda}}.
\]
Hence, it suffices to verify
\[
 u^\lambda \cdot (u+2x+1)^{\bar{\lambda}}\ge  \sqrt{u(u+2x+1)} \Leftrightarrow (u+2x+1)^{\frac{1}{4x+2}} > u^{\frac{1}{4x+2}},
\]
which is immediate as $u+2x+1>u$. Having established that $\Psi(\cdot)$ is decreasing on $[0,Q-1]$, thus $\min_{0\le k\le Q-1}\Psi(k) = \Psi(Q-1)$; it holds that
\[
\min_{0\le k\le Q-1}\Psi(k) =  \min_{0\le k\le Q-1} \sqrt{\left(1-\frac{k^2}{Q^2}\right)\left(1-\frac{(k+1)^2}{Q^2}\right)} + \frac{k(k+1)}{Q^2} = 1-\frac1Q.
\]
In particular, $\rho'$ chosen as in~\eqref{eq:main-choice-of-param} admits
\begin{equation}\label{eq:rho-prime}
\rho' =1 -\frac{1}{8\cdot 480^2\cdot L} \left(\frac{E_n}{n} \right)^{4+\frac{\epsilon}{4}} =  1 - \frac1Q = \min_{0\le k\le Q-1}\Psi(k),
\end{equation}
where $Q$ is the parameter studied in~\eqref{eq:Q-chosen-this-way},~\eqref{eq:choose-Q-thisway}. 

{\bf To prove.} In what follows, our goal is to establish that there exists no randomized algorithm  $\mathcal{A}:\R^n\times \Omega\to \bincube$ that is $\left(2^{-E_n}, C_1n, L,\rho',p_f,p_{\rm st}\right)$-optimal for every sufficiently large $n$, where $C_1,\rho',p_f,p_{\rm st}$ admit the convenient expressions~\eqref{eq:C-1-chosen-this-way},~\eqref{eq:rho-prime},~\eqref{eq:p_f-chosen-this-way}, and~\eqref{eq:p-rm-st}; respectively.

\paragraph{ Reduction to deterministic algorithms.} We first reduce the proof  to the case $\A$ is deterministic. Let $\mathbb{P}_X\otimes \mathbb{P}_\omega$ denotes the joint law of $(X,\omega)$. Here, $\omega$ is the randomness of $\A$. Define now the event
\[
\mathcal{E}_s(\omega) \triangleq  \left\{\frac{1}{\sqrt{n}}\Bigl|\ip{X}{\A(X;\omega)}\Bigr| \le 2^{-E_n}\right\}.
\]
Observe that
\[
\mathbb{P}_{X,\omega}\left(\frac{1}{\sqrt{n}}\Bigl |\ip{X}{\A(X,\omega)}\Bigr|>2^{-E_n}\right) = \mathbb{E}_\omega \Bigl[\mathbb{P}_X\Bigl(\mathcal{E}_s\left(\omega\right)^c\Bigr)\Bigr].
\]
We now perceive $\mathbb{P}_X\left(\mathcal{E}_s\left(\omega\right)^c\right)$ as a random variable whose source of randomness is $\omega$ (as the randomness over $X$ is ``integrated" over $\mathbb{P}_X$). Using Markov's inequality
\[
\mathbb{P}_\omega \Bigl(\mathbb{P}_X\Bigl(\mathcal{E}_s\left(\omega\right)^c\Bigr)\ge 2p_f\Bigr)\le \frac{\mathbb{E}_\omega \left[\mathbb{P}_X\left(\mathcal{E}_s\left(\omega\right)^c\right)\right]}{2p_f} \le \frac12.
\]
Set
\[
\Omega_1\triangleq \Bigl\{\omega\in\Omega: \mathbb{P}_X\left(\mathcal{E}_s\left(\omega\right)^c\right)<2p_f\Bigr\}  \implies \mathbb{P}_\omega\left(\Omega_1\right)\ge \frac12.
\]
Now, divide the interval $[0,1]$ into $Q$ subintervals $0=\tau_0<\tau_1<\cdots<\tau_Q=1$, each of size $Q^{-1}$ for $Q$ introduced in \eqref{eq:Q-chosen-this-way}. Next, set
\begin{equation}\label{eq:rho-k}
\rho_k \triangleq \sqrt{\left(1-\tau_k^2\right)\left(1-\tau_{k+1}^2\right)}+\tau_k\tau_{k+1},\quad 0\le k\le Q-1.
\end{equation}
For $\rho'$ introduced in \eqref{eq:rho-prime}, we have that $\rho_k \in[\rho',1]$, $0\le k\le Q-1$. Define next
\[
\mathcal{E}_2\left(\omega\right)\triangleq \Bigl\{d_H\Bigl(\A(X,\omega),\A(Y,\omega)\Bigr)\le C_1n+L\|X-Y\|_2^2\Bigr\}.
\]
Define also  the  sequence $A_{k,\omega}$ of random variables
\[
A_{k,\omega} \triangleq \mathbb{P}_{(X,Y):X\sim_{\rho_k}Y}\Bigl(\mathcal{E}_2(\omega)^c\Bigr),\quad 0\le k\le Q-1,\quad\text{and}\quad \omega\in\Omega.
\]
The source of randomness in each $A_{k,\omega}$ is due to $\mathbb{P}_\omega$. Observe now that for any fixed $0\le k\le Q-1$, using Markov's inequality similar to above,
\begin{align*}
\mathbb{P}_\omega \left(A_{k,\omega}\ge 3Qp_{\rm st}\right)&\le \frac{1}{3Qp_{\rm st}}\mathbb{E}_\omega \Bigl[A_{k,\omega}\Bigr] \\
&= \frac{1}{3Qp_{\rm st}}\mathbb{P}_{(X,Y,\omega):X\sim_{\rho_k}Y}\Bigl(d_H\Bigl(\A(X,\omega),\A(Y,\omega)\Bigr)> C_1n+L\|X-Y\|_2^2\Bigr)\\
&\le \frac{1}{3Q}.
\end{align*}
Taking now a union bound over $0\le k\le Q-1$,
\[
\mathbb{P}_\omega\left(\bigcup_{0\le k\le Q-1}\Bigl\{A_{k,\omega}\ge 3Qp_{\rm st}\Bigr\}\right)\le \frac13.
\]
Hence, 
\[
\Omega_2\triangleq \Bigl\{\omega\in\Omega:A_{k,\omega}<3Qp_{\rm st},0\le k\le Q-1\Bigr\} \implies \mathbb{P}_\omega(\Omega_2)\ge \frac23.
\]
Since $\mathbb{P}_\omega(\Omega_1)+\mathbb{P}_\omega(\Omega_2)\ge \frac12+\frac23>1$, it follows that $\Omega_1\cap \Omega_2\ne\varnothing$. Consequently, there exists an $\omega^*\in\Omega$, such that
\begin{equation}\label{eq:algo-reduce-to-deterministic}
    \mathbb{P}_X\left(\frac{1}{\sqrt{n}}\Bigl|\ip{X}{\mathcal{A}(X,\omega^*)}\Bigr|\le 2^{-E_n}\right)\ge 1-2p_f.
\end{equation}
and
\begin{equation}\label{eq:algo-reduce-to-deterministic-stability}
  \mathbb{P}_{(X,Y):X\sim_{\rho_k}Y}\Bigl(d_H\Bigl(\A(X,\omega^*),\A(Y,\omega^*)\Bigr)\le C_1n+L\|X-Y\|_2^2\Bigr) \ge 1-3Qp_{\rm st},\quad \text{for}\quad 0\le k\le Q-1.
\end{equation}
In the remainder, we fix this choice of $\omega^*\in\Omega$, and interpret $\A(\cdot)\triangleq \A(\cdot,\omega^*)$ as a deterministic (that is, no ``coin flip" $\omega$) map acting between $\R^n$ and $\bincube$. 
\paragraph{ An auxiliary high-probability event.} We now study a certain auxiliary high-probability event. This event 
pertains to the spin configurations that are near-optimal with respect to {\em independent} instances.

Let $\mathcal{M}$ be an index set with cardinality $m$, $X_i\distr\mathcal{N}(0,I_n)$, $i\in \mathcal{M}$, be i.i.d.\,Let $S_{\mathcal{M}}$ be a shorthand for the set 
\[
\mathcal{S}_{\mathcal{M}} \triangleq \mathcal{S}\left(1,\frac{3g(n)}{E_n},m,E_n,\{1\}\right)
\] 
(in the sense of Definition~\ref{def:overlap-set}) of $m-$tuples $\left(\sigma^{(i)}:i\in \mathcal{M}\right)$ of spin configurations $\sigma^{(i)}\in\bincube$ with a modification that the $\Overlap\left(\cdot,\cdot\right)$ in Definition~\ref{def:overlap-set} is replaced with $\OBar\left(\cdot,\cdot\right)$, the normalized inner product, as studied in Theorem~\ref{thm:m-ogp-superconstant-m}. Here, we keep the $m$ parameter as in \eqref{eq-alt:m(n)-2over-s(n)}; but modify the $\eta$ parameter into $3g(n)/E_n$. Note that with these choices, $\eta=\frac{1-\beta}{2m}$ no longer holds, but as we expand below this does not cause any problems.

Namely, $S_\mathcal{M}$ is the set of spin configurations that i) have  a large inner product and ii) are near ground-state with respect to \emph{independent} instances $X_i\in\R^n$. We claim 
\begin{lemma}\label{lem:kaotik}
 \begin{equation}\label{eq:chaos}
\mathbb{P}(\mathcal{E}_{\mathcal{M}}) \le \exp(-\Theta(n)),
\end{equation}
where
\begin{equation}\label{eq:event-chaos}
  \mathcal{E}_{\mathcal{M}}\triangleq \left\{S_{\mathcal{M}}\ne \varnothing\right\}=\left\{\left|S_{\mathcal{M}}\right|\ge 1\right\}.
\end{equation}
\end{lemma}

Namely on $\mathcal{E}_{\mathcal{M}}$, it is the case that for any $\left(\sigma^{(i)}:i\in\mathcal{M}\right)$ that are near-optimal, there exists $i<j$, $i,j\in\mathcal{M}$, such that 
\begin{equation}\label{eq:Obar-chaos}
\OBar\left(\sigma^{(i)},\sigma^{(j)}\right)\in\left[0,1-\frac{3g(n)}{E_n}\right].
\end{equation}

\begin{proof}[Proof of Lemma~\ref{lem:kaotik}]
The proof  of this claim is nearly identical to (and in fact easier than) that of Theorem~\ref{thm:m-ogp-superconstant-m}, Case 2. Thus we only point out the necessary modifications. 

The term $g(n)$ and $E_n$, as functions of $s(n)$ and $z(n)$, remain the same as~\eqref{eq-alt:g(n)-and-z(n)}. That is,
\[
g(n) = n\cdot s(n)\cdot z(n), \quad E_n = n\cdot s(n),\quad\text{where}\quad z(n) = s(n)^{1+\frac{\epsilon}{8}}.
\]
The expression~\eqref{eq-alt:eta(n)-s(n)-z(n)} regarding parameters $\beta,\eta$ now modifies to
\[
\beta =1 \quad \text{and}\quad \eta = \frac{3g(n)}{E_n}.
\]

 Note that, in this case the covariance matrix $\Sigma$ is always identity due to the independence of $X_i$, $1\le i\le m$. Moreover, with $\beta=1$; the counting term $n\frac{1-\beta+\eta}{2}$ studied in \eqref{eq:we-gonna-use-this-later} is now 
\[
n\frac{1-\beta+\eta}{2} =\frac{3ng(n)}{2E_n} = \frac32 nz(n).
\]
This is clearly $o(n)$ since $z(n) = o(1)$. Thus, Lemma~\ref{lemma:binomial-k-o(n)} is applicable, and the counting bound, \eqref{eq:ramsey-exponent-bin-coeff}, now becomes
\begin{align*}
\binom{n}{n\frac{1-\beta+\eta}{2}}  &=\exp_2\left(\left(1+o_n(1)\right)\frac{3nz(n)}{2} \log_2 \frac{2n}{3nz(n)}\right) \\
&= \exp_2\left(O\left(ns(n)^{1+\frac{\epsilon}{8}}\log_2 \frac{1}{s(n)}\right)\right),
\end{align*}
where we used the fact $z(n)=s(n)^{1+\frac{\epsilon}{8}}$. Note now that 
\[
\Bigl |[n\beta-n\eta,n\beta]\cap \mathbb{Z}\Bigr |=O\left(n\eta\right)=O\left(\frac{ng(n)}{E_n}\right) = O(nz(n)) = o(n),
\] 
using $z(n)=o(1)$. Hence, 
\[
\log \Bigl |[n\beta-n\eta,n\beta]\cap \mathbb{Z}\Bigr | = O(\log n) =o\left(ns(n)^{1+\frac{\epsilon}{8}}\log_2 \frac{1}{s(n)}\right)
\]
Thus,~\eqref{eq:ramsey-up-bd-step3} remains the same. Hence, the cardinality upper bound~\eqref{eq:modified-cardinality-bound} is still of form
\[
\exp_2\left(n+O\left(mns(n)^{1+\frac{\epsilon}{8}}\log_2 \frac{1}{s(n)}\right)\right).
\]

Now, since the covariance matrix $\Sigma$ is identity,  there is no contribution of a term of form $\frac{m}{2}\log_2 \frac{1}{\nu_n}$ (the determinant contribution) to~\eqref{eq:modified-pb-bd}; and the ``dominant" contribution of the probability term~\eqref{eq:modified-pb-bd} to the exponent of the first moment is $-mE_n$. 

Hence,~\eqref{eq:modified-pb-bd},~\eqref{eq2:a-very-loong-expression-for-the-firdt-momnet-m-ogp}; and~\eqref{eq:first-mom-bd-ramsey-finale} all remain the same. Thus, Theorem~\ref{thm:m-ogp-superconstant-m} indeed still remains valid.

\end{proof}
\paragraph{ Construction of interpolation paths.} Our proof will use the so-called ``interpolation method". To that end, let $X_i\in\R^n$, $0\le i\le T$, be i.i.d. random vectors (dubbed as \emph{replicas}), each having distribution $\mathcal{N}(0,I_n)$, where $T$ is specified in \eqref{eq:T-chosen-this-way}.

Recall now $Y_i(\tau)$, $\tau\in[0,1]$ and $1\le i\le T$, from Definition~\ref{def:overlap-set}.
Notice that for any $\tau\in[0,1]$ and any $1\le i\le T$, $Y_i(\tau)\distr \mathcal{N}(0,I_n)$. At $\tau=0$, it is the case that $Y_i(\tau)=Y_j(\tau)=X_0$ for $1\le i<j\le T$. Thus, for $\tau=0$,
\[
\A\left(Y_i(\tau)\right)=\A\left(Y_j(\tau)\right),\quad 1\le i<j\le T.
\]
At $\tau=1$, on the other hand, $Y_i(\tau)$, $1\le i\le T$, is a collection of $T$ i.i.d. random vectors, each with distribution $\mathcal{N}(0,I_n)$. 

Divide the interval $[0,1]$ into $Q$ subintervals $0=\tau_0<\tau_1<\cdots<\tau_{Q}=1$, each of size $1/Q$, where $Q$ is specified in \eqref{eq:Q-chosen-this-way}. Define next the pairwise overlaps
\begin{equation}\label{eq:pairwise-overlap}
\OBar^{(ij)}\left(\tau_k\right)\triangleq \frac1n \ip{\A\left(Y_i\left(\tau_k\right)\right)}{\A\left(Y_j\left(\tau_k\right)\right)}
\end{equation}
for $1\le i<j\le T$ and $0\le k\le Q$. 

\paragraph{ Stability of successive steps.} We now establish, using the stability of $\A$, that for $1\le i<j\le T$ and $0\le k\le Q-1$, 
\[
\left|\OBar^{(ij)}\left(\tau_k\right)-\OBar^{(ij)}\left(\tau_{k+1}\right)\right|
\]
is small. More concretely we establish
\begin{lemma}\label{lemma:steps-stable}
\begin{equation}\label{prob-steps-stable}
    \mathbb{P}(\mathcal{E}_3) \ge 1-(T+1)\exp\Bigl(-\Theta(n)\Bigr)-3TQ^2 p_{\rm st}
\end{equation}
where
\begin{equation}\label{events-steps-stable}
    \mathcal{E}_3\triangleq \bigcap_{1\le i<j\le T}\bigcap_{0\le k\le Q}\left\{\left|\OBar^{(ij)}(\tau_k)-\OBar^{(ij)}(\tau_{k+1})\right|\le 4\sqrt{C_1}+\frac{48\sqrt{2L}}{\sqrt{Q}}\right\}.
\end{equation}
\end{lemma}
Later, we study the asymptotics of $T$ and show that the bound in \eqref{prob-steps-stable} is not vacuous.
\begin{proof}

We first establish 
that for every $1\le i\le T$, $\|X_i\|_2\le 6\sqrt{n}$ w.h.p. Let $X_i=(X_i(j):1\le j\le n)$, where $X_i(j)$, $1\le j\le n$ are i.i.d. standard normal. Appealing to Bernstein's inequality as in the proof of \cite[Theorem~3.1.1]{vershynin2018high}, we have that for every $t\ge 0$,
\[
\mathbb{P}\left(\left|\frac1n\sum_{1\le j\le n}X_i(j)^2-1\right|\ge t\right)\le \exp(-cn\min\{t^2,t\}).
\]
Using now a union bound over $1\le i\le T$, we conclude
\[
\mathbb{P}\left(\|X_i\|\le  6\sqrt{n},0\le i\le T\right)\ge 1-(T+1)\exp(-\Theta(n)). 
\]
Here the choice of $6$ is arbitrary, any constant larger than $1$ works.

Fix any $1\le i\le T$. We now upper bound $\|Y_i\left(\tau_k\right)-Y_i\left(\tau_{k+1}\right)\|$, where $Y_i(\cdot)$ is defined in Definition~\ref{def:overlap-set}. Note that
\begin{align*}
\left\|Y_i\left(\tau_k\right)-Y_i\left(\tau_{k+1}\right)\right\|&\le \left|\sqrt{1-\tau_k^2}-\sqrt{1-\tau_{k+1}^2}\right|\|X_0\| +|\tau_k-\tau_{k+1}|\|X_i\| \\
&\le \left|\sqrt{1-\tau_k^2}-\sqrt{1-\tau_{k+1}^2}\right|\|X_0\| +Q^{-1}\|X_i\|
\end{align*}
using triangle inequality, and the fact $|\tau_k-\tau_{k+1}|\le Q^{-1}$. Next, observe that using $\tau_k,\tau_{k+1}\in[0,1]$,
\[
\sqrt{\left|\tau_k^2-\tau_{k+1}^2\right|} = \sqrt{\left|\tau_k-\tau_{k+1}\right|}\sqrt{\tau_k+\tau_{k+1}}\le \sqrt{2\left|\tau_k-\tau_{k+1}\right|}.
\]
We now show
\[
\frac{\sqrt{\left|\tau_k^2-\tau_{k+1}^2\right|}}{\sqrt{1-\tau_k^2}+\sqrt{1-\tau_{k+1}^2}}\le 1.
\]
Squaring, and using $\tau_{k+1}>\tau_k$, this is equivalent to having
\begin{align*}
\tau_{k+1}^2 - \tau_k^2 \le 1-\tau_k^2 + 1-\tau_{k+1}^2 +2\sqrt{\left(1-\tau_k^2\right)\left(1-\tau_{k+1}^2\right)} \iff \tau_{k+1}^2 \le 1+\sqrt{\left(1-\tau_k^2\right)\left(1-\tau_{k+1}^2\right)},
\end{align*}
which holds as $0\le \tau_i \le 1$ for all $i$.
Equipped with the previous bounds, we thus have
\begin{align*}
    \left|\sqrt{1-\tau_k^2}-\sqrt{1-\tau_{k+1}^2}\right| &= \frac{|\tau_k^2-\tau_{k+1}^2|}{\sqrt{1-\tau_k^2}+\sqrt{1-\tau_{k+1}^2}} \\
    &\le \sqrt{2|\tau_k-\tau_{k+1}|}\underbrace{\frac{\sqrt{\left|\tau_k^2-\tau_{k+1}^2\right|}}{\sqrt{1-\tau_k^2}+\sqrt{1-\tau_{k+1}^2}}}_{\le 1}\\
    &\le \sqrt{2}|\tau_k-\tau_{k+1}|^{\frac12} \\&\le \sqrt{2}Q^{-\frac12},
\end{align*}
where the last line uses $|\tau_k-\tau_{k+1}|\le Q^{-1}$. In particular, on the high probability event\footnote{As we verify soon, $T=2^{o(n)}$, hence this is indeed a high probability event.}
\[
\mathcal{E}_{\rm norm}\triangleq \left\{\|X_i\|_2\le 6\sqrt{n},0\le i\le T\right\}
\]
it holds that
\begin{equation}\label{eq:Y-i-tau-bd}
    \Bigl\|Y_i(\tau_k)-Y_i(\tau_{k+1})\Bigr\| \le 6\sqrt{2}n^{\frac12}Q^{-\frac12} + 6Q^{-1}n^{\frac12}\le 12\sqrt{2}n^{\frac12}Q^{-\frac12}.
\end{equation}
Next, we observe that $Y_i\left(\tau_k\right),Y_i\left(\tau_{k+1}\right)$ are both distributed $\mathcal{N}(0,I_n)$ with correlation
\[
\mathbb{E}\left[Y_i\left(\tau_k\right)Y_i\left(\tau_{k+1}\right)^T\right] = \rho_k I
\]
where $\rho_k$, $0\le k\le Q-1$, is per~\eqref{eq:rho-k}.
Define now the event
\[
\mathcal{E}_{\rm stability} =\bigcap_{0\le i\le T}\bigcap_{0\le k\le Q-1}\Bigl\{d_H\Bigl(\mathcal{A}\left(Y_i\left(\tau_k\right)\right),\mathcal{A}\left(Y_i\left(\tau_{k+1}\right)\right)\Bigr)\le C_1n+L\left\|Y_i\left(\tau_k\right)-Y_i\left(\tau_{k+1}\right)\right\|_2^2\Bigr\}
\]
which is the event that the algorithm is stable for each interpolation trajectory $1\le i\le T$ along time indices $0\le k\le Q-1$. Using \eqref{eq:algo-reduce-to-deterministic-stability} together with a union bound over $1\le i\le T$ and $0\le k\le Q-1$, it holds that with probability at least $1-3TQ^2p_{\rm st}$,
\[
\mathbb{P}\left(\mathcal{E}_{\rm stability}\right)\ge 1-3TQ^2p_{\rm st}.
\]
Consequently, taking a union bound, this time for $\mathcal{E}_{\rm norm}\cap \mathcal{E}_{\rm stability}$, we arrive at 
\[
\mathbb{P}\left(\mathcal{E}_{\rm norm}\cap \mathcal{E}_{\rm stability}\right) = 1-\left(T+1\right)\exp\left(-\Theta(n)\right)-3TQ^2 p_{\rm st}.
\]
We now compute $\left|\OBar^{(ij)}\left(\tau_k\right)-\OBar^{(ij)}\left(\tau_{k+1}\right)\right|$ on the event $\mathcal{E}_{\rm norm}\cap \mathcal{E}_{\rm stability}$, while treating the stability condition deterministic due to the conditioning. 

For notational convenience, let $\A_i(k)\triangleq \A\left(Y_i(\tau_k)\right)$, $1\le i\le T$ and $0\le k\le Q$. We first observe that
\[
\|\A_i(k)-\A_j(k)\| =2\sqrt{d_H\left(\A_i(k),\A_j(k)\right)}.
\]
Using the stability condition, 
\[
d_H\left(\A_i(k),\A_i(k+1)\right)\le C_1n+L\|Y_i(k)-Y_i(k+1)\|_2^2,\quad 1\le i\le T,\quad 0\le k\le Q-1
\]
together with the trivial inequality $\sqrt{u+v}\le \sqrt{u}+\sqrt{v}$, valid for all $u,v\ge 0$, we have
\begin{equation}\label{eq:hamming-to-l2}
\|\A_i(k)-\A_i(k+1)\|_2\le 2\sqrt{C_1}\sqrt{n} + 2\sqrt{L}\left\|Y_i(k)-Y_i(k+1)\right\|_2.
\end{equation}
Next,
\begin{align}
     \left|\OBar^{(ij)}(\tau_k)-\OBar^{(ij)}(\tau_{k+1})\right|      &= \frac1n\Bigl|\ipbig{\A_i(k)}{\A_j(k)}-\ipbig{\A_i(k+1)}{\A_j(k+1)}\Bigr|\label{eq:triangle-ineq-rev-1}\\
     &\le \frac1n\Bigl|\ip{\A_i(k)-\A_i(k+1)}{\A_j(k)}\Bigr|+\frac1n \Bigl|\ip{\A_i(k+1)}{\A_j(k)-\A_j(k+1)}\Bigr|\label{eq:overlap-triangleq}\\
     &\le \frac{1}{\sqrt{n}}\Bigl(\|\A_i(k)-\A_i(k+1)\|_2+\|\A_j(k)-\A_j(k+1)\|_2\Bigr)\label{eq:overlap-cauchy}\\
     &\le \frac{1}{\sqrt{n}}\left(4\sqrt{C_1}\sqrt{n}+2\sqrt{L} \|Y_i(\tau_k)-Y_i(\tau_{k+1})\|+2\sqrt{L} \|Y_j(\tau_k)-Y_j(\tau_{k+1})\|\right)\label{eq:overlap-hamming-to-l2}\\
     &\le 4\sqrt{C_1}+\frac{4\sqrt{L}\cdot 12\sqrt{2}n^{\frac12}Q^{-\frac12}}{\sqrt{n}}\label{eq:overlap-y-i-tau-bd}\\
     &=4\sqrt{C_1}+\frac{48\sqrt{2L}}{\sqrt{Q}}.
\end{align}
Above, \eqref{eq:triangle-ineq-rev-1} uses the definition; 
\eqref{eq:overlap-triangleq} uses the triangle inequality; \eqref{eq:overlap-cauchy} uses the Cauchy-Schwarz inequality, and the fact $\|\A_i(k+1)\|_2 = \|\A_j(k)\|=\sqrt{n}$; \eqref{eq:overlap-hamming-to-l2} uses \eqref{eq:hamming-to-l2}; and finally \eqref{eq:overlap-y-i-tau-bd} uses \eqref{eq:Y-i-tau-bd}.
\end{proof}
\paragraph{ Success along the trajectory.} We now study the event that the algorithm $\A$ is ``successful" along each interpolation trajectory.
We claim that we have
\begin{equation}\label{eq:success-prob-along-traj}
    \mathbb{P}\left(\mathcal{E}_4\right)\ge 1-2T(Q+1)p_f
\end{equation}
where the event $\mathcal{E}_4$ is defined as 
\begin{equation}\label{eq:success-along-traj-event}
    \mathcal{E}_4\triangleq \bigcap_{1\le i\le T}\bigcap_{0\le k\le Q}\left\{\frac{1}{\sqrt{n}}\Bigl|\ip{\A_i(k)}{Y_i(\tau_k)}\Bigr|\le 2^{-E_n}\right\}.
\end{equation}
Namely, the event $\mathcal{E}_4$ says that the algorithm $\A$ creates a near ground-state at each ``discrete time instance" $0\le k\le Q$ along each interpolation trajectory $1\le i\le T$. 

We now prove this claim. Note that as $X_i\in\R^n$, $1\le i\le T$ are i.i.d. $\mathcal{N}(0,I_n)$, it follows that for each $1\le i\le T$ and $0\le k\le Q$, $Y_i(\tau_k)\distr \mathcal{N}(0,I_n)$. Using now \eqref{eq:algo-reduce-to-deterministic} together with a union bound over $1\le i\le T$ and $0\le k\le Q$ settles the result.

\paragraph{ The order of growth of parameters.} Before we put everything together, we now study the order of growth of relevant parameters. This is necessary for applying union bound arguments that will follow. 

First, combining \eqref{eq:E-n-satisfies-this} and \eqref{eq:m(n)-2over-s(n)}, we obtain
\begin{equation}\label{eq:m(n)-order}
    \omega(1)\le m \le o\left(\log^{\frac15-\epsilon}n\right).
\end{equation}
This is not vacuous since $\epsilon<\frac15$. 

Next, since $L$ is constant in $n$, the asymptotics of $Q$ given in~\eqref{eq:choose-Q-thisway} becomes
\begin{align*}
Q=\left(8\cdot 480^2\cdot L\right)\cdot \left(\frac{n}{E_n}\right)^{4+\frac{\epsilon}{4}} = \Theta\left(\left(\frac{n}{E_n}\right)^{4+\frac{\epsilon}{4}}\right). 
\end{align*}
Recalling now the condition~\eqref{eq:E-n-satisfies-this} on $E_n$, we obtain
\begin{equation}\label{eq:Q(n)-order}
Q =  o\left(\log^{\left(\frac15-\epsilon\right)\left(4+\frac{\epsilon}{4}\right)} n\right).
\end{equation}
Moreover, \eqref{eq:Q(n)-order} yields also that
\begin{equation}\label{eq:log-Q(n)-order}
    \log Q = \log\left(o\left(\log^{\left(\frac15-\epsilon\right)\left(4+\frac{\epsilon}{4}\right)} n\right)\right) = O\left(\log \log n\right).
\end{equation}
Combining bounds \eqref{eq:m(n)-order}, \eqref{eq:Q(n)-order}, and \eqref{eq:log-Q(n)-order}, we arrive at
\begin{equation}\label{eq:m(n)-Q(n)-logQ(n)}
    mQ\log Q = o\left(\log^{\left(\frac15-\epsilon\right)\left(5+\frac{\epsilon}{4}\right)}n\right)O\left(\log \log n\right) = o\left(\log^{\left(\frac15-\epsilon\right)\left(5+\frac{\epsilon}{2}\right)}n\right),
\end{equation}
where we used
\[
O\left(\log \log n\right) = o\left(\log^{w} n \right)
\]
valid for any constant $w>0$. Next, observe that for $\epsilon>0$, 
\[
\left(\frac15-\epsilon\right)\left(5+\frac{\epsilon}{2}\right) = 1 - \left(\frac{49}{10}\epsilon +\frac{\epsilon^2}{2}\right)<1.
\]
Thus,  by combining \eqref{eq:T-chosen-this-way}, \eqref{eq:m(n)-Q(n)-logQ(n)}, and the fact $mQ\log Q+\log m =\Theta(mQ\log Q)$, we arrive at
\begin{equation}\label{eq:binom-T-choose-m-sub-exp}
    \binom{T}{m} \le T^m =\exp_2\left(m2^{4mQ\log_2 Q}\right) = 2^{o(n)}.
\end{equation}


\paragraph{ Putting everything together.} We now put everything together. 
For any $\mathcal{M}\subset [T]$ with $\left|\mathcal{M}\right|=m$, recall the event $\mathcal{E}_{\mathcal{M}}$ as in \eqref{eq:event-chaos}, and define the event 
\begin{equation}\label{eq:event-chaos-all-M}
    \mathcal{E}_1\triangleq \bigcap_{\mathcal{M}\subset [T]:\left|\mathcal{M}\right|=m} \mathcal{E}_{\mathcal{M}}.
\end{equation}
Using \eqref{eq:chaos}, together with a union bound, we obtain
\[
\mathbb{P}\left(\mathcal{E}_1^c\right)=\mathbb{P}\left(\bigcup_{\mathcal{M}\subset [T]:\left|\mathcal{M}\right|=m} \mathcal{E}_{\mathcal{M}}^c\right)\le \binom{T}{m}\exp\left(-\Theta\left(n\right)\right).
\]
Since $\binom{T}{m} = 2^{o(n)}$ per \eqref{eq:binom-T-choose-m-sub-exp}, we deduce
\begin{equation}\label{eq:chaos-over-M-very-likely}
    \mathbb{P}\left(\mathcal{E}_1\right)\ge 1-\exp\left(-\Theta(n)\right).
\end{equation}
For the events $\mathcal{E}_1$ defined in \eqref{eq:event-chaos-all-M} (refer to \eqref{eq:chaos-over-M-very-likely} for its probability), $\mathcal{E}_3$ defined in \eqref{events-steps-stable} (refer to \eqref{prob-steps-stable} for its probability), and $\mathcal{E}_4$ defined in \eqref{eq:success-along-traj-event} (refer to \eqref{eq:success-prob-along-traj} for its probability), define their intersection by
\[
\mathcal{F} = \mathcal{E}_1\cap \mathcal{E}_3\cap \mathcal{E}_4.
\]
Check that using \eqref{prob-steps-stable}, the fact $T=2^{o(n)}$ per \eqref{eq:binom-T-choose-m-sub-exp} as well as the choice of $p_{\rm st}$ per \eqref{eq:p-rm-st}, we have
\begin{align*}
\mathbb{P}\left(\mathcal{E}_3^c\right)&\le (T+1)\exp\left(-\Theta(n)\right)+3TQ^2 p_{\rm st} \\
&\le \exp\left(-\Theta(n)\right)+\frac13.
\end{align*}
Moreover, using \eqref{eq:success-prob-along-traj} as well as the choice of $p_f$ per \eqref{eq:p_f-chosen-this-way}, we arrive at
\[
\mathbb{P}(\mathcal{E}_4^c)\le \frac12.
\]
A union bound over $\mathcal{E}_1^c,\mathcal{E}_2^c$ and $\mathcal{E}_4^c$ then yields
\begin{equation}\label{eq:prob-of-the-intersection-cal-F}
    \mathbb{P}\left(\mathcal{F}\right) = \mathbb{P}\left(\mathcal{E}_1\cap \mathcal{E}_3\cap \mathcal{E}_4\right)=1-\mathbb{P}\left(\mathcal{E}_1^c\cup \mathcal{E}_3^c\cup \mathcal{E}_4^c\right)\ge \frac16-\exp\left(-\Theta(n)\right).
\end{equation}
In the remainder, assume that we are on the event $\mathcal{F}$.

Note that, from the choice of $C_1$ per~\eqref{eq:C-1-chosen-this-way} and $Q$ per~\eqref{eq:Q-chosen-this-way}, we obtain that on the event $\mathcal{F}$, it holds that
\begin{equation}\label{eq:eta-over-5-enough}
\left|\OBar^{(ij)}\left(\tau_k\right)-\OBar^{(ij)}\left(\tau_{k+1}\right)\right|\le \frac{\eta}{5},
\end{equation}
for $1\le i<j\le T$ and $0\le k\le Q$.

Now, fix any subset $S\subset[T]$ with $|S|=m$. A consequence of the event $\mathcal{E}_1$, through~\eqref{eq:Obar-chaos}, is that there exists distinct $i_S,j_S\in S$ such that \[
\OBar^{(i_S ,j_S)}\left(\tau_Q\right)\in \left[0,1-\frac{3g(n)}{E_n}\right].
\]
We verify that this interval is ``below" the forbidden region, $(\beta-\eta,\beta)$: per~\eqref{eq:eta(n)-s(n)-z(n)} it suffices to ensure
\[
\beta - \eta = \underbrace{1 - \frac{2g(n)}{E_n}}_{=\beta} - \underbrace{\frac{g(n)}{2n} }_{=\eta} >1 - \frac{3g(n)}{E_n}\Leftrightarrow \frac{g(n)}{E_n} > \frac{g(n)}{2n} \Leftrightarrow 2n>E_n.
\]
Since $E_n=o(n)$, this indeed holds for all sufficiently large $n$. 

Take now $\delta =\frac{\eta}{100}$. We next show there exists a $k'\in[1,Q]\cap \mathbb{Z}$ such that 
\[
\OBar^{(i_S ,j_S)}\left(\tau_{k'}\right) \in (\beta-\eta+3\delta,\beta-3\delta),
\]
where $(\beta-\eta,\beta)$ is the forbidden overlap region as per Theorem~\ref{thm:m-ogp-superconstant-m}. Take indeed $K_0$ to be the last index (in $[1,Q]\cap \mathbb{Z}$) where $\OBar^{(i_S j_S)}(\tau_{K_0})\ge \beta-3\delta$. Note that such a $K_0$ must exist since $\OBar^{(ij)}(0)=1$ for every $1\le i<j\le T$. Then if $\OBar^{(i_S j_S)}(\tau_{K_0+1})\le \beta-\eta+3\delta$, we obtain
\[
\left|\OBar^{(i_S, j_S)}(\tau_{K_0})-\OBar^{(i_S j_S)}(\tau_{K_0+1})\right|\ge \eta-6\delta>0,
\]
which contradicts with the event $\mathcal{E}_3$ and in particular with \eqref{eq:eta-over-5-enough} for sufficiently large $n$. Namely,
\[
\OBar^{(i_S ,j_S)}(\tau_{K_0+1}) \in (\beta-\eta+3\delta,\beta-3\delta).
\]
In particular, keeping in mind that $S$ was arbitrary, we conclude that for every subset $S\subset[T]$ of cardinality $|S|=m$, there exists $1\le i_S<j_S\le m$ such that for some $\tau_S\in\{\tau_1,\dots,\tau_Q\}$ it is the case that
\[
\OBar^{(i_S,j_S)}\left(\tau_S\right) \in (\beta-\eta+3\delta,\beta-3\delta)\subsetneq (\beta-\eta,\beta).
\]
Equipped with this, we now construct a certain graph $\Gr=(V,E)$ such that the following holds.
\begin{itemize}
    \item Its vertex set $V$ coincides with $[T]$. That is, $V=\{1,2,\dots,T\}$, where each vertex corresponds to an interpolation trajectory $1\le i\le T$. 
    \item For any $1\le i<j\le T$, $(i,j)\in E$ if and only if there exists a time $\tau\in[0,1]$ such that
    \[
    \OBar^{(ij)}(\tau)\in (\beta-\eta,\beta).
    \]
\end{itemize}
Next, we ``color" each edge of $\Gr$ with one of $Q$ colors. Specifically, for any $1\le i<j\le T$ with $(i,j)\in E$, this edge is colored with color $t$, $1\le k\le Q$ where $t$ is the first time instance $\{\tau_1,\tau_2,\dots,\tau_Q\}$ such that
\[
\OBar^{(ij)}(\tau_t) \in (\beta-\eta,\beta).
\]
In particular, $\Gr$ enjoys the following properties.
\begin{itemize}
    \item $\Gr=(V,E)$ has $|V|=T$ vertices; with the property that for any subset $S\subset V$ of cardinality $|S|=m$, there exists a distinct pair $i,j\in S$ of vertices such that $(i,j)\in E$. 
    \item Any edge $(i,j)\in E$ of $\Gr$ is colored with one of $Q$ colors. 
\end{itemize}
\begin{proposition}\label{prop:G-mc-clique}
The graph $\Gr$ contains a monochromatic $m-$clique $K_m$.
\end{proposition}
\begin{proof}[Proof of Proposition~\ref{prop:G-mc-clique}]
Recall from \eqref{eq:T-chosen-this-way} 
that $\Gr$ has 
\[
T= \exp_2\left(2^{4mQ\log_2 Q}\right)
\]
vertices. Define now
\begin{equation}\label{eq:M-admissible-M-parameter}
    M \triangleq Q^{mQ}=2^{mQ\log_2 Q}
\end{equation}
Note that $\Gr$ is $m-$admissible, in the sense of Definition~\ref{def:m-admissible-graph}. Since $M>m$ for $Q>1$, it is also $M-$admissible. 
Observe that
\[
T=\exp_2\left(2^{4mQ\log_2 Q}\right)\ge \exp_2\left(2\cdot \underbrace{2^{mQ\log_2 Q}}_{=M}\right)=4^M \ge \binom{2M-2}{M-1}.
\]
Applying Proposition~\ref{thm:clique-exist}, we find that $\Gr$ contains an $M$, that is a $Q^{Qm}$, clique, $K_M$. Finally, since each edge of $K_M$ is colored with one of $Q$ colors and $R_Q(m)\le Q^{Qm}$ per Theorem~\ref{thm:ramsey}, we obtain that $K_M$ contains a monochromatic $m-$clique. Namely, $\Gr$ contains a monochromatic $m$-clique $K_m$ since all graphs above we worked with are subgraphs of $\Gr$. This concludes the proof of Proposition~\ref{prop:G-mc-clique}.
\end{proof}
We now complete the proof of Theorem~\ref{thm:Main}. Observe what it means for $\Gr$ to contain a monochromatic $m-$clique: there exists an $m-$tuple $1\le i_1<i_2<\cdots<i_m\le T$ of vertices (i.e. replicas) and a color (i.e. a time $\tau'\in\{\tau_1,\dots,\tau_Q\}$) such that
\[
\OBar^{(i_k,i_\ell)}\left(\tau'\right)\in(\beta-\eta,\beta),\quad 1\le k<\ell \le m.
\]
Now, define
\[
\sigma^{(k)}\triangleq \A\left(Y_{i_k}\left(\tau'\right)\right),\quad 1\le k\le m.
\]
It follows that $\left(\sigma^{(k)}:1\le k\le m\right)$ enjoys the following conditions:
\begin{itemize}
    \item Since we are on the event $\mathcal{F}$ which is a subset of the success event $\mathcal{E}_4$ \eqref{eq:success-along-traj-event}, it holds that
    \[
    \frac{1}{\sqrt{n}}\left|\ip{\sigma^{(k)}}{Y_{i_k}\left(\tau'\right)}\right|\le 2^{-E_n}   
     \]
    \item For $1\le k<\ell\le m$,
    \[
    \OBar\left(\sigma^{(k)},\sigma^{(\ell)}\right) \in(\beta-\eta,\beta).
    \]
\end{itemize}
Namely, for the choice $\zeta\triangleq \{i_1,i_2,\dots,i_m\}$ of the $m-$tuple  of distinct indices, the set $\mathcal{S}_\zeta \triangleq \mathcal{S}\left(\beta,\eta,m,E_n,\mathcal{I}\right)$ introduced in Definition~\ref{def:overlap-set}---with modification that inner products are considered---(where the indices $1,2,\dots,m$ there is replaced with $i_1,\dots,i_m$) with $\mathcal{I}=\{\tau_0,\tau_1,\dots,\tau_Q\}$ is non-empty. Namely,
\[
\mathbb{P}\Bigl(\exists \zeta\subset [T],|\zeta|=m:S_\zeta\ne \varnothing\Bigr)\ge \mathbb{P}\left(\mathcal{F}\right)\ge \frac12-\exp\left(-\Theta(n)\right).
\]
We now use the $m-$OGP result, Theorem~\ref{thm:m-ogp-superconstant-m}. Taking a union bound over $\zeta \subset[T]$ with $|\zeta|=m$ in Theorem~\ref{thm:m-ogp-superconstant-m}, we obtain
\[
\mathbb{P}\Bigl(\exists \zeta\subset [T],|\zeta|=m:S_\zeta\ne \varnothing\Bigr)\le \binom{T}{m}\exp\left(-\Theta(n)\right)=\exp\left(-\Theta(n)\right),
\]
since $\binom{T}{m}=2^{o(n)}$. But this yields
\[
\exp\left(-\Theta(n)\right)\ge \mathbb{P}\Bigl(\exists \zeta\subset [T],|\zeta|=m:S_\zeta\ne \varnothing\Bigr)\ge \frac12-\exp\left(-\Theta(n)\right),
\]
that is
\[
\exp\left(-\Theta(n)\right)\ge\frac16-\exp\left(-\Theta(n)\right).
\]
This is a contradiction for sufficiently large $n$. Therefore, the proof is complete. 
\end{proof}

\subsection{Proof of Theorem~\ref{thm:FEW}}\label{sec:pf-thm-FEW}
\begin{proof}
We start by recalling that 
\[
H\left(\sigma^*\right)=H\left(-\sigma^*\right)=\Theta\left(2^{-n}\right),
\]
with high probability, as noted in the introduction. Now, using Theorem~\ref{thm:2-ogp}, it follows that
\[
\min_{\sigma \in I_2}H(\sigma) = \Omega\left(2^{-n\epsilon}\right)
\]
with high probability. Indeed, for $\rho$ chosen as above, with high probability no two spin configurations with overlap $\left[\rho,\frac{n-2}{n}\right]$ can achieve simultaneously an energy of $O\left(2^{-n\epsilon}\right)$. 

In what follows next, the constants hidden under $\Theta\left(\cdot\right)$ and $\Omega\left(\cdot\right)$ are absorbed into the inverse temperature $\beta>0$. 

We have the following trivial lower bound:
\[
\pi_\beta(I_3)=\pi_\beta\left(\overline{I_3}\right)= \pi_\beta\left(\sigma^*\right)=\frac{1}{Z_\beta}\exp\left(-\beta H(\sigma^*)\right) = \frac{1}{Z_\beta}\exp\left(-\beta 2^{-n}\right).
\]
Notice, on the other hand, that for any $\sigma\in I_2$,
\[
\pi_\beta(\sigma) \le \frac{1}{Z_\beta}\exp\left(-\beta 2^{-n\epsilon}\right).
\]
Next, we upper bound
\[
\left|I_2\right| \le  \sum_{1\le k\le \left\lceil \frac{n(1-\rho)}{2}\right\rceil}\binom{n}{k} = \exp_2\left(nh\left(\frac{1-\rho}{2}\right)+O\left(\log_2 n\right)\right),
\]
where $h(\cdot)$ is the binary entropy function. Consequently
\[
\pi_\beta(I_2)=\sum_{\sigma \in I_2}\pi_\beta(\sigma)\le \frac{|I_2|}{Z_\beta}\left(-\beta 2^{-n\epsilon}\right)\le \frac{1}{Z_\beta}\exp\left(nh\left(\frac{1-\rho}{2}\right)+O\left(\log_2 n\right)-\beta 2^{-n\epsilon}\right).
\]
Hence
\[
\pi_\beta\left(I_3\right)\ge \exp\left(-\beta 2^{-n}+\beta 2^{-n\epsilon}-nh\left(\frac{1-\rho}{2}\right)+O\left(\log_2 n\right)\right)\pi_\beta(I_2).
\]
Finally, in the regime $\beta=\Omega\left(n2^{n\epsilon}\right)$, it is the case that
\[
-\beta 2^{-n}+\beta 2^{-n\epsilon}-nh\left(\frac{1-\rho}{2}\right)+O\left(\log_2 n\right) = \Omega\left(\beta 2^{-n\epsilon}\right)=\Omega(n).
\]
Hence,
\[
\pi_\beta\left(I_3\right)\ge e^{\Omega(n)}\pi_\beta(I_2).
\]
We next apply this reasoning for the set $I_1$, which is slightly more delicate. 

To that end, fix an $\epsilon'\in(\epsilon,1)$ (recall that $\epsilon<1$). We will show that with probability $1-O(1/n)$, there exists a $\sigma'\in\bincube$ such that $H(\sigma')=\Theta\left(2^{-n\epsilon'}\right)$. For this, it suffices to use \cite[Theorem~3.1]{karmarkar1986probabilistic} (with parameters $\beta=\sqrt{n}2^{-n\epsilon'}$ and $\epsilon=\frac{\beta}{2}$, in terms of their notation). 

It is evident, due to the OGP as well as the fact $I_3$ and $\overline{I_3}$ contains only ground states $\pm \sigma^*$,  that $\sigma'\notin \left(\overline{I_2}\cup I_2\right)\cup \left(\overline{I_3}\cup I_3\right)$. Consequently, $\sigma'\in I_1$. With this, we have the trivial lower bound
\[
\pi_\beta\left(I_1\right)\ge \pi_\beta\left(\sigma'\right) =\frac{1}{Z_\beta}\exp\left(-\beta 2^{-n\epsilon'}\right).
\]
Repeating the exact same reasoning while keeping in mind $\epsilon'>\epsilon$, we conclude
\[
\pi_\beta\left(I_1\right)\ge \exp\left(\Omega\left(\beta 2^{-n\epsilon}\right)\right)\pi_\beta(I_2).
\]
This concludes the proof.
\end{proof}
\subsection{Proof of Theorem~\ref{thm:slow-mixing}}\label{sec:pf-slow-mixing}
\begin{proof}
In what follows, we have $\beta=\Omega\left(n2^{n\epsilon}\right)$. 
\subsubsection*{Part ${\rm (a)}$}
Using the FEW property established in Theorem~\ref{thm:FEW}, we have that
    \[
    \min\left\{\pi_\beta(I_1),\pi_\beta(I_3)\right\}\ge \exp\left(\Omega(n)\right) \pi_\beta(I_2).
    \]
    We now use the facts $
    \pi_\beta (I_2)=\pi_\beta\left(\overline{I_2}\right)$,  $\pi_\beta (I_3)=\pi_\beta\left(\overline{I_3}\right)$; 
    and 
$
   \pi_\beta(I_1)+\pi_\beta (I_2)+\pi_\beta\left(\overline{I_2}\right)+\pi_\beta (I_3)+\pi_\beta\left(\overline{I_3}\right)\ge 1
$
    to arrive at $
   \left( \pi_\beta(I_1)+2\pi_\beta\left(I_3\right)\right)\left(1+\exp\left(-\Omega(n)\right)\right)\ge 1$. Consequently, we have $  \pi_\beta(I_1)+2\pi_\beta\left(I_3\right) \ge 1+o_n(1)$. With this we conclude that
    \[
    \pi_\beta(I_1)+\pi_\beta\left(I_3\right) \ge \frac12(1+o_n(1)),
    \]
    as claimed.
\subsubsection*{Part ${\rm (b)}$}
Theorem~\ref{thm:slow-mixing}${\rm (b)}$ is a consequence of following proposition.
\begin{proposition}\label{prop:escape-time-cdf}
Let $\beta=\Omega\left(n2^{n\epsilon}\right)$. Then, for any $T>0$, the ``escape time" $\tau_\beta$ introduced in \eqref{eq:escape-time} satisfies
\[
\mathbb{P}\left(\tau_\beta \le T\right)\le T\exp\left(-\Omega\left(\beta 2^{-n\epsilon}\right)\right),
\]
with high probability (over the randomness of $X\distr\mathcal{N}(0,I_n)$) as $n\to\infty$. 
\end{proposition}
\begin{proof}[Proof of Proposition~\ref{prop:escape-time-cdf}]
The proof uses standard arguments, similar to \cite{gjs2019overlap,gamarnik2019landscape}; and is reproduced herein for completeness.

Consider first the Markov chain $\overline{X}_t$, which is the Markov chain $X_t$ reflected on the boundary $A\triangleq \partial \left(I_3\cup \partial S\right)$ of $I_3\cup \partial S$. Observe that $A$ is nothing but the set of all spin configurations $\sigma\in\bincube$ such that $d_H\left(\sigma,\sigma^*\right)=1$, that is $\frac1n\ip{\sigma}{\sigma^*}=\frac{n-2}{n}$. 

We now specify the transition kernel $\overline{Q}(x,y)$ of $\overline{X}_t$. If $x\in \left(I_3\cup \partial S\right)\setminus A$, then $\overline{Q}(x,y)=Q(x,y)$ for any $y\in I_3\cup \partial S$. If $x\in A$, then $\overline{Q}(x,y)=Q(x,y)$ for $y\in I_3\cup \partial S$; and $\overline{Q}(x,y)=0$ otherwise.

A consequence of the detailed balance equation is $\overline{X}_t$ is reversible with respect to $\pi_\beta\left(\cdot\mid I_3\cup\partial S\right)$. 

We now couple the initialization of the chains; $\overline{X}_0=X_0\sim \pi_\beta\left(\cdot\mid I_3\cup \partial S\right)$, to arrive at the conclusion that so long as $t\le \tau_\beta$, almost surely $\overline{X}_t = X_t$ and $\overline{X}_t \sim \pi_\beta\left(\cdot\mid I_3\cup\partial S\right)$.
From the definition of the ``escape time", it is the case $X_{\tau_\beta-1}\in A$, the ``boundary". Consequently
\begin{align}
    \mathbb{P}\left(\tau_\beta\le T\right)&\le \sum_{1\le i\le T}\mathbb{P}\left(\tau_\beta=i\right) \\
    &\le \sum_{1\le i\le T}\mathbb{P}\left(\tau_\beta=i,X_{i-1}\in A\right) \label{eq:x-tau-beta-1-inA}\\
    &=\sum_{1\le i\le T}\mathbb{P}\left(\tau_\beta=i,\overline{X}_{i-1}\in A\right) \label{eq:almost-sure-equal}\\
    &\le \sum_{1\le i\le T}\mathbb{P}\left(\overline{X}_{i-1}\in A\right)\\
    &=T\pi_\beta\left(A\mid I_3\cup \partial S\right)\label{eq:overline-X-initial}.
\end{align}
Here, \eqref{eq:x-tau-beta-1-inA} follows from the fact $X_{\tau_\beta-1}\in A$ recorded above; \eqref{eq:almost-sure-equal} uses the fact $\overline{X}_t = X_t$ almost surely as long as $t\le \tau_\beta$; and \eqref{eq:overline-X-initial} follows from the fact $\overline{X}_t \sim \pi_\beta\left(\cdot\mid I_3\cup\partial S\right)$.

We now employ the FEW property to conclude the proof. Observe that $A=\partial \left(I_3\cup \partial S\right)\subset I_3\cup \partial S$. Moreover, observe that $A=\partial S \subset I_2$, hence $\pi_\beta(A)\le \pi_\beta(I_2)$; and $\pi_\beta(I_3\cup \partial S)\ge \pi_\beta(I_3)$. Combining these, we obtain
\begin{equation}\label{eq:hadi-bakalim}
\pi_\beta\left(A\mid I_3\cup \partial S\right) = \frac{\pi_\beta\left(A\right)}{\pi_\beta\left(I_3\cup \partial S\right)} \le \frac{\pi_\beta\left(I_2\right)}{\pi_\beta\left(I_3\right)}\le \exp\left(-\Omega\left(\beta 2^{-n\epsilon}\right)\right).
\end{equation}
The last inequality uses the FEW property established in Theorem~\ref{thm:FEW}. Combining \eqref{eq:overline-X-initial} and \eqref{eq:hadi-bakalim} we conclude the proof:
\[
\mathbb{P}\left(\tau_\beta\le T\right)\le T\exp\left(-\Omega\left(\beta 2^{-n\epsilon}\right)\right).
\]
\end{proof}
With this, the proof of Theorem~\ref{thm:slow-mixing} is complete.
\end{proof}
\subsubsection*{Acknowledgments}
Part of this work was done while the authors were visiting the Simons Institute for the Theory of Computing at University of California, Berkeley in Fall 2020.
\bibliographystyle{amsalpha}
\bibliography{bibliography}

\end{document}